\documentclass[11pt,letterpaper,twoside,reqno,nosumlimits]{amsart}

\usepackage{newclude}

\usepackage[usenames,dvipsnames]{xcolor}
\usepackage{fancyhdr}
\usepackage{amsmath,amsfonts,amsbsy,amsgen,amscd,amssymb,amsthm}
\usepackage{url}

\usepackage{mathtools}

\usepackage[font=small,margin=0.25in,labelfont={bf},labelsep={colon}]{caption}

\usepackage{subcaption}

\usepackage{tikz}
\usepackage{microtype}
\usepackage{enumitem}

\definecolor{dark-gray}{gray}{0.3}
\definecolor{dkgray}{rgb}{.4,.4,.4}
\definecolor{dkblue}{rgb}{0,0,.5}
\definecolor{medblue}{rgb}{0,0,.75}
\definecolor{rust}{rgb}{0.5,0.1,0.1}

\usepackage[colorlinks=true]{hyperref}

\hypersetup{urlcolor=rust}
\hypersetup{citecolor=dkblue}
\hypersetup{linkcolor=dkblue}

\usepackage{setspace}

\usepackage{graphicx}
\usepackage{booktabs,longtable,tabu} 
\usepackage{multirow} 

\usepackage{float}

\usepackage[full]{textcomp}

\usepackage[scaled=.98,sups,osf]{XCharter}
\usepackage[scaled=1.04,varqu,varl]{inconsolata}
\usepackage[type1]{cabin}
\usepackage[charter,vvarbb,scaled=1.07]{newtxmath}
\usepackage[cal=boondoxo]{mathalfa}
\usepackage[T1]{fontenc}

\usepackage{bm} 

\graphicspath{{figures/}}

\newtheorem{theorem}{Theorem}[section]
\newtheorem{lemma}[theorem]{Lemma}

\newtheorem{proposition}[theorem]{Proposition}

\newtheorem{corollary}[theorem]{Corollary}

\theoremstyle{definition}

\newtheorem{assumption}[theorem]{Assumption}

\numberwithin{equation}{section} 
\numberwithin{figure}{section}
\numberwithin{table}{section}

\floatstyle{plaintop}
\newfloat{recipe}{thp}{lor}
\floatname{recipe}{Recipe}
\numberwithin{recipe}{section}

\renewcommand{\phi}{\varphi}

\newcommand{\eps}{\varepsilon}

\newcommand{\diff}[1]{\mathrm{d}{#1}}
\newcommand{\idiff}[1]{\, \diff{#1}}

\newcommand{\triplenorm}[1]{{\left\vert\kern-0.25ex\left\vert\kern-0.25ex\left\vert #1
    \right\vert\kern-0.25ex\right\vert\kern-0.25ex\right\vert}}

\newcommand{\bbD}{\mathbb{D}}

\newcommand{\calG}{\mathcal{G}}

\newcommand{\calR}{\mathcal{R}}

\newcommand{\frG}{\frak G}

\newcommand{\frR}{\frak R}

\DeclareMathOperator{\supp}{supp}

\def\XXint#1#2#3{{%
\setbox0=\hbox{$#1{#2#3}{\int}$}
\vcenter{\hbox{$#2#3$}}\kern-.5\wd0}}

\renewcommand{\leq}{\leqslant}
\renewcommand{\geq}{\geqslant}
\renewcommand{\subset}{\subseteq}

\allowdisplaybreaks

\usepackage[numbers]{natbib}
\mathtoolsset{showonlyrefs}

\usepackage{booktabs}
\usepackage{tabularx}
\usepackage{array}

\usepackage{comment}

\begin{document}
\title{Exact blowup analysis for the weak-advection Hou--Li model}
\author{Thomas Y. Hou\({}^{*}\),\,Xiang Qin\({}^{\dagger}\),\, Xiuyuan Wang\({}^{\ddagger}\)}
\thanks{\({}^{*}\)Applied and Computational Mathematics, Caltech, Pasadena, CA. Email: \href{mailto:hou@cms.caltech.edu}{\texttt{hou@cms.caltech.edu}}.}
\thanks{\({}^{\dagger}\)Applied and Computational Mathematics, Caltech, Pasadena, CA. Email: \href{mailto:xqin2@caltech.edu}{\texttt{xqin2@caltech.edu}}.}

\thanks{\({}^{\ddagger}\)This research was conducted while Xiuyuan Wang was a SURF student at Caltech in the summer of 2023. Email: \href{mailto:zaenorae@gmail.com}{\texttt{zaenorae@gmail.com}}.}

\begin{abstract}
We study self-similar singularity formation for the one-dimensional
weak-advection Hou--Li model, a reduced model motivated by the axisymmetric
Euler equations. 
In the periodic setting, we construct exact finite-time self-similar blowup solutions for \(2/3<a<1\), with profiles that are neither focusing nor expanding. 
In the whole-space setting with a Neumann condition, we
construct exact finite-time self-similar blowup solutions for the full range \(0<a\leq1\), with profiles of focusing, non-expanding/non-focusing, or expanding form depending on the sign of the self-similar scaling parameter.
The construction is based on a fixed-point formulation near the origin,
followed by an ODE extension argument. We also establish regularity,
asymptotic behavior, monotonicity properties of the profiles, and uniqueness up to the natural scaling invariance.
\end{abstract}

\maketitle

\section{Introduction}
We study the one-dimensional
inviscid weak-advection Hou--Li model
\begin{equation}\label{eqt:gHL}
\begin{aligned}
(u_{1})_{t}+2a \psi_{1}(u_{1})_{z} & = 2(\psi_{1})_{z} u_{1}, \\
(\omega_{1})_{t}+2a \psi_{1}(\omega_{1})_{z} & = (u_{1}^{2})_{z}, \\
-(\psi_{1})_{zz} & = \omega_{1},
\end{aligned}
\end{equation}
where \(a\in[0,1]\) measures the strength of advection relative to vortex
stretching, and we impose the normalization \(\psi_1(0)=0\). This model may be
viewed as a one-dimensional reduction of the three-dimensional axisymmetric
Euler equations along the symmetry axis, as we explain as follows.

Understanding the formation of singularities in fluid equations is one of the
central challenges in the theory of nonlinear partial differential equations.
In particular, it remains unknown whether smooth solutions of the
three-dimensional incompressible Euler or Navier--Stokes equations can develop
a finite-time singularity from smooth initial data. One of the main
difficulties arises from the vortex-stretching mechanism, which creates a
strong nonlinear coupling between the velocity and vorticity fields. While
vortex stretching can amplify vorticity and potentially drive singularity
formation, its interaction with advection and the underlying nonlocal
structure of the equations remains subtle and not yet fully understood.

In recent years, significant progress has been made in rough or borderline regularity classes.
Elgindi constructed finite-time blowup for the three-dimensional axisymmetric Euler equations without swirl in a \(C^{1,\alpha}\) velocity class for small \(\alpha>0\) \cite{elgindi2021finite}. Subsequent works improved the range of admissible H\"older exponents and reached much larger classes; see, for instance, Cordoba, Martinez-Zoroa, and Zheng \cite{cordoba2023finite}, Shao, Wei, Zhang, and Zhang \cite{shao2026self}, Shkoller \cite{shkoller2026euler}, and Chen \cite{chen2026eulerI,chen2026eulerII}. In the presence of a boundary, motivated by the numerical observations of \cite{luo2014potentially}, Chen and Hou proved stable nearly self-similar blowup for the two-dimensional Boussinesq equations and the three-dimensional axisymmetric Euler equations with smooth data \cite{chenhou2022stable}.
There remain several numerical scenarios of potential singularity formation that are not yet fully understood rigorously; see, for example, \cite{houhuang2022twoscale,hou2023interior,hou2023navierstokes,houhuang2023degenerate,hou2026generalized}.
Singularity formation is also closely related to nonuniqueness phenomena, since singular initial data may generate forward self-similar solutions, and an unstable self-similar profile can provide a mechanism for constructing different weak solutions from the same initial data. This viewpoint is closely related to the works of Jia and \v{S}ver\'ak on forward self-similar solutions and possible ill-posedness in the natural energy space \cite{jia2014local,jia2015are}, the numerical investigations of Guillod and \v{S}ver\'ak \cite{guillod2023numerical}, and the forced nonuniqueness result of Albritton, Bru\'e, and Colombo \cite{albritton2022nonuniqueness}. More recently, Hou, Wang, and Yang proved nonuniqueness of Leray--Hopf solutions to the unforced three-dimensional Navier--Stokes equations by constructing a self-similar solution with an unstable linearized mode \cite{houwangyang2025nonuniqueness}. We also refer the reader to related works on nonuniqueness for fluid equations, including Vishik's work on the two-dimensional Euler equations \cite{vishik2018partI,vishik2018partII}, the exposition and further development in \cite{albritton2024vishik}, and the recent work of Mengual and Solera on the forced two-dimensional Navier--Stokes and dissipative SQG equations \cite{mengualsolera2026sharp}.

To gain more insight into the blowup mechanisms, a number of reduced models have been
proposed that retain essential structural features of the full fluid equations
while remaining more amenable to analysis. One of the earliest and most
influential examples is the Constantin--Lax--Majda (CLM) model
\cite{constantin1985simple}, which captures the nonlocal nature of vortex
stretching through the Hilbert transform. The CLM model admits explicit
solutions and exhibits finite-time blowup from smooth initial data. De Gregorio
later incorporated an advection term into the CLM equation in order to study
the competition between transport and vortex stretching, revealing that the
transport effect can have a stabilizing influence on the dynamics.
A broader class of models arises by introducing a parameter that controls the
relative strength of advection and vortex stretching. This idea leads to the
generalized Constantin--Lax--Majda (gCLM) model proposed by Okamoto
\emph{et al.} \cite{okamoto2008generalization}. Various weak-advection
versions of the gCLM model have been studied extensively; see, for example,
\cite{cordoba2005formation,castro2010infinite,elgindi2020effects,
chen2020singularity,chen2021finite,huang2023self}. These works reveal a wide
range of behaviors, including both global regularity and finite-time
singularity formation depending on the strength of the advection term. In particular, self-similar structures have played an important role in the
analysis of these models and in the description of their possible singularity
formation.

Another important direction comes from reduced models derived from the
axisymmetric Euler and Navier--Stokes equations. The vorticity
formulation of the three-dimensional incompressible Euler equations reads
\[
    \omega_t+u\cdot\nabla\omega=\omega\cdot\nabla u,
    \qquad
    \nabla\cdot u=0,
\]
where \(\omega=\nabla\times u\). In cylindrical coordinates
\((r,\theta,z)\), an axisymmetric velocity field with swirl can be written as
\[
    u=u^r(r,z,t)e_r+u^\theta(r,z,t)e_\theta+u^z(r,z,t)e_z,
\]
and the corresponding vorticity takes the form
\[
    \omega=\omega^r(r,z,t)e_r+\omega^\theta(r,z,t)e_\theta
    +\omega^z(r,z,t)e_z .
\]
Introducing the angular stream function \(\psi^\theta\), we write
\[
    u^r=-(\psi^\theta)_z,\qquad
    u^z=\frac1r(r\psi^\theta)_r .
\]
 In 2008, Hou and Li
\cite{hou2008dynamic} considered the axisymmetric Euler variables
\[
    u_1=\frac{u^\theta}{r},\qquad
    \omega_1=\frac{\omega^\theta}{r},\qquad
    \psi_1=\frac{\psi^\theta}{r}.
\]
These variables satisfy the system
\[
\begin{aligned}
(u_1)_t+u^r(u_1)_r+u^z(u_1)_z
    &=2(\psi_1)_z u_1,\\
(\omega_1)_t+u^r(\omega_1)_r+u^z(\omega_1)_z
    &=(u_1^2)_z,\\
-\left(\partial_r^2+\frac3r\partial_r+\partial_z^2\right)\psi_1
    &=\omega_1.
    \\
      u^r=-r(\psi_1)_z,\qquad&
    u^z=2\psi_1+r(\psi_1)_r.
\end{aligned}
\]
By considering the corresponding one-dimensional reduction along the symmetry
axis, in which \(u_1,\omega_1,\psi_1\) depend only on the axial variable, they
obtained \eqref{eqt:gHL} in the case \(a=1\). Thus, for \(a=1\), solutions of
the Hou--Li model can be lifted back to exact axisymmetric Euler solutions of
infinite energy by setting
\[
    u^\theta(t,r,z)=r u_1(t,z),\qquad
    \omega^\theta(t,r,z)=r\omega_1(t,z),\qquad
    \psi^\theta(t,r,z)=r\psi_1(t,z).
\]
In the same work, Hou and Li proved global regularity of this model by
constructing a Lyapunov functional that reveals an exact cancellation between
the advection and vortex-stretching effects.

Hou and Wang \cite{hou2023blowup} later studied the weak-advection version of
the Hou--Li model in the periodic setting, introducing the parameter
\(a\in[0,1]\) in \eqref{eqt:gHL}. Using a combination of analytical techniques and
computer-assisted estimates, they proved that there exists a sufficiently
small \(\delta>0\) such that for \(a\in(1-\delta,1)\), the weak-advection
model develops a finite-time singularity from smooth initial data. Although
their rigorous analysis applied only to values of \(a\) very close to \(1\),
their numerical study over a wider range of \(a\) revealed rich self-similar
structures associated with blowup solutions. In particular, when \(a\) is
close to \(1\), the singularity is neither purely focusing nor expanding,
whereas for \(a\) below a critical value, the blowup exhibits focusing
behavior.

Motivated by the numerical observations of Hou and Wang, in this work we
construct self-similar blowup solutions for the weak-advection Hou--Li model
in a much wider range of the parameter \(a\), capturing the different
self-similar behaviors suggested by their computations.
Our construction is inspired by the fixed-point argument in \cite{huang2023self}, but the present system lacks the global monotonicity and convexity structures used there. Instead, we establish the existence of the profiles by combining local monotonicity estimates with a local extension argument. For convenience we introduce the variables,
\[
\omega := \omega_1, \qquad
u := 2\psi_1, \qquad
\theta := u_1^2, \qquad
x := z.
\]
Then \eqref{eqt:gHL} becomes
\begin{equation}\label{eqt:gHL1}
\begin{aligned}
\theta_{t}+a u \theta_{x} &= 2 u_{x} \theta, \\
\omega_{t}+a u \omega_{x} &= \theta_{x}, \\
-u_{xx} &= 2 \omega, \qquad u(0)=0.
\end{aligned}
\end{equation}

Our primary interest is the formation of finite–time singularities exhibiting a self–similar structure. Motivated by the scaling properties of \eqref{eqt:gHL1}, we seek exact self–similar finite–time blowups of the form
\begin{equation}\label{eqt:ansatz}
\begin{aligned}
\omega(x,t) &= (T-t)^{c_{\omega}}\, \Omega\!\left(\frac{x}{(T-t)^{c_{l}}}\right), \\
u(x,t)      &= (T-t)^{c_{u}}\,     U\!\left(\frac{x}{(T-t)^{c_{l}}}\right), \\
\theta(x,t) &= (T-t)^{c_{\theta}}\, \Theta\!\left(\frac{x}{(T-t)^{c_{l}}}\right),
\end{aligned}
\end{equation}
where $T$ denotes the blowup time and $\Omega$, $U$, and $\Theta$ are the corresponding self–similar profiles.

Substituting the ansatz \eqref{eqt:ansatz} into \eqref{eqt:gHL1} and writing
\[
X=\frac{x}{(T-t)^{c_l}},
\]
we obtain the relations
\[
2c_{\omega}+2c_{l}=c_{\theta}, \qquad c_{\theta}=-2,
\]
and the following system of ordinary differential equations for the profiles:
\[
\begin{aligned}
(c_{l} X+a U) \Theta_{X} & = (c_{\theta}+2 U_{X}) \Theta, \\
(c_{l} X+a U) \Omega_{X} & = c_{\omega} \Omega+\Theta_{X}, \\
-U_{XX} & = 2 \Omega, \qquad U(0)=0.
\end{aligned}
\]

For simplicity we drop the capitals and write $\omega,u,\theta,x$ for $\Omega,U,\Theta,X$. Thus the profile equations become
\begin{equation}\label{eqt:main}
\begin{aligned}
(c_{l} x+a u)\,\theta_{x}& = (c_{\theta}+2 u_{x})\,\theta, \\
(c_{l} x+a u)\,\omega_{x}& = c_{\omega}\,\omega+\theta_{x}, \\
-u_{xx} & = 2 \omega, \qquad u(0)=0.
\end{aligned}
\end{equation}

The profile system possesses a scaling invariance: if $(\omega,u,\theta,c_l,c_\omega,c_\theta)$ solves \eqref{eqt:main}, then for any \(\lambda_1\in\mathbb R\) and \(\lambda_2>0\), the rescaled tuple
\begin{equation}\label{eq:profile_scaling_invariance}
\left(
\lambda_1\, \omega(\lambda_2 x),\
\frac{\lambda_1}{\lambda_2^{2}} u(\lambda_2 x),\
\frac{\lambda_1^{2}}{\lambda_2^{2}} \theta(\lambda_2 x),\
\frac{\lambda_1}{\lambda_2} c_{l},\
\frac{\lambda_1}{\lambda_2} c_{\omega},\
\frac{\lambda_1}{\lambda_2} c_{\theta}
\right).
\end{equation}
is also a solution. Therefore the absolute values of the parameters are not essential; what matters are the ratios $c_l/c_\theta$ and $c_\omega/c_\theta$.

We impose the following structural assumptions on the profiles.
\begin{assumption}
We seek solutions satisfying:
\begin{itemize}
\item $u(x)$ is odd, and $\theta(x)$ is the square of an odd function;
\item $u(x)$ and $\theta(x)$ are sufficiently regular;
\item \textbf{Nondegeneracy:} there exists $k\in\mathbb{N}$ such that
\[
\omega^{(2k+1)}(0)\neq 0,
\qquad
\theta^{(2k+2)}(0)\neq 0.
\]
\end{itemize}
\end{assumption}

Under these assumptions, we first focus on the lowest-order nondegenerate case
\(k=0\). Since our goal is to
construct blowup solutions, it is natural to require the $u_x$ term at
the origin to be positive, both for the physical equation \eqref{eqt:gHL1} and
for the profile equation \eqref{eqt:main}. This leads to profiles with
\(\omega_x(0)>0\). Moreover, since \(\theta\) comes from the square of the
original variable \(u_1\), we require \(\theta_{xx}(0)>0\). Within this class,
and up to the natural scaling invariance of the profile system, we establish
the existence and uniqueness of the relevant profiles, together with their
regularity and asymptotic properties. These profile results then yield
finite-time blowup solutions of the evolution equation \eqref{eqt:gHL1}. We
summarize the main conclusions below.

\begin{theorem}[Periodic profiles and exact self-similar blowup]
\label{theorem: main_periodic}
Let \(2/3<a<1\). Then there exist constants
    $\widetilde c_l=0,\,
    \widetilde c_\omega=-1,\,
    \widetilde c_\theta=-2,$
and a nontrivial symmetric \(2\pi\)-periodic profile
    $(\widetilde\Omega,\widetilde U,\widetilde\Theta)$
solving the profile system \eqref{eqt:main}. The profile has the following
properties:
\begin{enumerate}
\item \(\widetilde U\) and \(\widetilde\Omega\) are odd, and
\(\widetilde\Theta\) is even.

\item At the origin,
\[
    \widetilde U_X(0)>0,\qquad
    \widetilde\Omega_X(0)>0,\qquad
    \widetilde\Theta_{XX}(0)>0.
\]

\item On the positive half-period, \(\widetilde U\) and
\(\widetilde\Theta\) are nonnegative.

\item The quantities $\widetilde U(X)/X$ and $\widetilde\Theta(X)/X^2$
are monotone decreasing on the positive half-period.
\item The profile is smooth inside each period. More precisely, with
\(\beta\) and \(\gamma\) defined in Section~2, one has
\[
    \widetilde\Omega
    \in C^{\lceil\beta\rceil-1,\beta-\lceil\beta\rceil+1}_{\mathrm{per}}(\mathbb R)
    \cap C^\infty\bigl(\mathbb R\setminus\{(2k+1)\pi:k\in\mathbb Z\}\bigr),
\]
\[
    \widetilde U
    \in C^{1+\lceil\beta\rceil,\beta-\lceil\beta\rceil+1}_{\mathrm{per}}(\mathbb R)
    \cap C^\infty\bigl(\mathbb R\setminus\{(2k+1)\pi:k\in\mathbb Z\}\bigr),
\]
\[
    \widetilde\Theta
    \in C^{\lceil\gamma\rceil-1,\gamma-\lceil\gamma\rceil+1}_{\mathrm{per}}(\mathbb R)
    \cap C^\infty\bigl(\mathbb R\setminus\{(2k+1)\pi:k\in\mathbb Z\}\bigr).
\]
\end{enumerate}
Moreover, within the class of profiles satisfying the above properties, the
profile is unique up to the natural scaling invariance
\eqref{eq:profile_scaling_invariance}.
Furthermore, for any \(T>0\), the following functions
\[
    \omega(x,t)
    =
    \frac{1}{T-t}\widetilde\Omega(x),
\qquad
    u(x,t)
    =
    \frac{1}{T-t}\widetilde U(x),
\qquad
    \theta(x,t)
    =
    \frac{1}{(T-t)^2}\widetilde\Theta(x),
\]
define a \(2\pi\)-periodic solution of \eqref{eqt:gHL1} on
\(\mathbb S^1\times[0,T)\), which blows up at \(t=T\) in a neither focusing nor expanding self-similar form.
\end{theorem}

\begin{theorem}[Whole-space profiles and exact self-similar blowup]
\label{theorem: main_whole_space}
Let \(0<a\leq1\). Then there exist constants
$    \widetilde c_l=2/a-3,\,
    \widetilde c_\omega=2-2/a,\,
    \widetilde c_\theta=-2,$
and a nontrivial symmetric whole-space profile
    $(\widetilde\Omega,\widetilde U,\widetilde\Theta)$
solving the profile system \eqref{eqt:main} and satisfying the Neumann
condition
\[
    \lim_{|X|\to+\infty}\widetilde U_X(X)=0.
\]
The profile has the following properties:
\begin{enumerate}
\item \(\widetilde U\) and \(\widetilde\Omega\) are odd, and
\(\widetilde\Theta\) is even.

\item At the origin,
\[
    \widetilde U_X(0)>0,\qquad
    \widetilde\Omega_X(0)>0,\qquad
    \widetilde\Theta_{XX}(0)>0.
\]

\item If \(2/3<a\leq1\), then the quantities $\widetilde U(X)/X$ and $\widetilde\Theta(X)/X^2$
are monotone decreasing on the positive half-line.
\item The support of the profile depends on the sign of \(\widetilde c_l\).
If \(0<a\leq2/3\), then
\[
    \supp \widetilde\Omega=\supp \widetilde\Theta=\mathbb R.
\]
If \(2/3<a\leq1\), then there exists \(L>0\) such that
\[
    \supp \widetilde\Omega\cup\supp \widetilde\Theta\subset[-L,L].
\]
\item The profile has the following regularity.
For \(0<a\leq2/3\), one has
    $\widetilde\Omega,\widetilde U,\widetilde\Theta\in C^\infty(\mathbb R).$
For \(2/3<a<1\), the functions \(\widetilde\Omega\) and
\(\widetilde\Theta\) are compactly supported. More precisely, with
\(\beta\) and \(\gamma\) defined in Section~3, one has
\[
    \widetilde\Omega
    \in C^{\lceil\beta\rceil-1,\beta-\lceil\beta\rceil+1}(\mathbb R)
    \cap C^\infty\bigl(\mathbb R\setminus\{\pm L\}\bigr),
\]
\[
    \widetilde U
    \in C^{1+\lceil\beta\rceil,\beta-\lceil\beta\rceil+1}(\mathbb R)
    \cap C^\infty\bigl(\mathbb R\setminus\{\pm L\}\bigr),
\]
\[
    \widetilde\Theta
    \in C^{\lceil\gamma\rceil-1,\gamma-\lceil\gamma\rceil+1}(\mathbb R)
    \cap C^\infty\bigl(\mathbb R\setminus\{\pm L\}\bigr).
\]
For \(a=1\), after extending \(\widetilde\Omega\) and
\(\widetilde\Theta\) by zero outside their support, one has
\[
    \widetilde\Omega\in C^{0,1}(\mathbb R)
    \cap C^\infty\bigl(\mathbb R\setminus\{\pm L\}\bigr),
\]
\[
    \widetilde U\in C^{2,1}(\mathbb R)
    \cap C^\infty\bigl(\mathbb R\setminus\{\pm L\}\bigr),
\]
\[
    \widetilde\Theta\in C^{1,1}(\mathbb R)
    \cap C^\infty\bigl(\mathbb R\setminus\{\pm L\}\bigr).
\]
\end{enumerate}
Moreover, within the class of profiles satisfying the above properties, the
profile is unique up to the natural scaling invariance
\eqref{eq:profile_scaling_invariance}.

Furthermore, for any \(T>0\), the functions \[ \omega(x,t) = (T-t)^{\widetilde c_\omega} \widetilde\Omega\!\left(\frac{x}{(T-t)^{\widetilde c_l}}\right), \qquad u(x,t) = (T-t)^{\widetilde c_l-1} \widetilde U\!\left(\frac{x}{(T-t)^{\widetilde c_l}}\right), \]\[ \theta(x,t) = (T-t)^{-2} \widetilde\Theta\!\left(\frac{x}{(T-t)^{\widetilde c_l}}\right), \]
define a solution of \eqref{eqt:gHL1} on
\(\mathbb R\times[0,T)\) with a Neumann condition, which blows up at \(t=T\). The blowup is focusing
for \(0<a<2/3\), neither focusing nor expanding for \(a=2/3\), and
expanding for \(2/3<a\leq1\).
\end{theorem}

The two main theorems show that the geometry and the parameter \(a\) determine
the support of the profile and the type of self-similar blowup.  The result is
summarized in Table~\ref{tab:main_regimes}.

\begin{table}[htbp]
\centering
\caption{Summary of the profiles and blowup regimes in the main theorems.}
\label{tab:main_regimes}
\renewcommand{\arraystretch}{1.15}
\small
\begin{tabular}{cccc}
\toprule
Geometry & Range of \(a\) & Profile support & Blowup regime \\
\midrule
\(\mathbb S^1\) 
& \(2/3<a<1\) 
& periodic 
& exact self-similar with fixed scale \\
\(\mathbb R\) 
& \(0<a<2/3\) 
& full support 
& exact self-similar with focusing \\
\(\mathbb R\) 
& \(a=2/3\) 
& full support 
& exact self-similar with fixed scale \\
\(\mathbb R\) 
& \(2/3<a\leq1\) 
& compact support 
& exact self-similar with expanding scale \\
\bottomrule
\end{tabular}
\end{table}


The remainder of the paper is organized as follows. In Section~2, we study
periodic profiles for the profile system \eqref{eqt:main} in the case
\(c_l=0\), for \(2/3<a\leq 1\). In Section~3, we construct whole-space
profiles for \eqref{eqt:main} satisfying the Neumann condition
\(u_x(\infty)=0\), for the full range \(0<a\leq1\). In this range, the
corresponding value of \(c_l\) changes sign, so the whole-space construction
covers the regimes \(c_l>0\), \(c_l=0\), and \(c_l<0\). In both sections, we
prove existence by combining a fixed-point argument with an ODE extension
analysis, and establish further properties of the profiles, including
regularity and uniqueness under prescribed normalization conditions. In
Section~4, we show how these profiles generate finite-time blowup solutions
of the evolution equation \eqref{eqt:gHL1} in the periodic and whole-space
settings. Finally, in Section~5, we present numerical simulations that
validate the profile construction and illustrate the dependence of the
profiles on the parameter \(a\).

\section{Periodic Self–Similar Profiles with $c_l=0$}
We now construct periodic solutions to the profile system \eqref{eqt:main} in the case \(c_l=0\). The parameter \(c_l=0\) leads to neither focusing nor expanding profiles. In this section we focus only on the construction and properties of the profiles. The corresponding finite-time blowup solutions of the evolution equation \eqref{eqt:gHL1} will be constructed later in Section~4.

The argument proceeds in several steps. We first reduce the profile equations to a normalized system for the auxiliary variables \(m,v,g\). We then solve this system locally near the origin by a fixed-point argument. After that, we extend the local solution using the associated ODE system and analyze its endpoint behavior. Finally, we prove regularity and uniqueness of the resulting periodic profile.

\subsection{Reduced and normalized formulation}
We begin by deriving the normalized form of the profile equations. The
normalization removes the scaling freedom of the system and fixes the leading behavior of the profile at the origin. This reduction is useful because the resulting equations can be written in terms of quantities that are regular and monotone near the origin.

With $c_l=0$, \eqref{eqt:main} reduces to
\begin{subequations}\label{eqt: case1}
\begin{align}
a u\, \theta_{x}  &= (c_{\theta}+2 u_{x})\, \theta, \label{case1a}\\
2a u\, \omega_{x} &= c_{\theta}\, \omega+2\theta_{x}, \label{case1b}\\
-u_{xx} &= 2 \omega,\qquad u(0)=0. \label{case1c}
\end{align}
\end{subequations}
By scaling and the nondegeneracy condition, we may fix $u_x(0)=1$ and $\omega_x(0)=1$. Taking $\frac{1}{x^2}$ times \eqref{case1a} and $\frac{1}{x}$ times \eqref{case1b} and letting $x\to 0$ yield
\[
c_\theta = 2(a-1)u_x(0)=2(a-1), \qquad
\theta_{xx}(0) = -\tfrac{1}{2}\!\big(c_\theta-2a\,u_x(0)\big)\,\omega_x(0) = 1.
\]
Integrating \eqref{case1b} gives
\[
2a u (u_{x}-1)_x  = \big(a u_x+a+c_{\theta}\big) (u_x-1)-4\theta.
\]
Set $g:=u/x$ and $m:=\theta/x^2$. Then $g$, $m$, and $u_x$ satisfy
\begin{equation}\label{eqt:gmu}
\begin{aligned}
axg\,m_x &= \big(2xg_x+2(1-a)(g-1)\big)\,m,\\
2axg\,u_{xx} &= \big(a u_x+3a-2\big)\,(u_x-1)-4x^2 m,\\
(xg)_x &= u_x,
\end{aligned}
\end{equation}
with $2m(0)=1$ and $g(0)=1$. 

The system \eqref{eqt:gmu} is still singular at the origin because of the factors \(x\) and \(xg\). To construct a solution, we exploit the integral structure of the equations and rewrite the problem as a fixed-point equation for \(v=u_x\). This formulation automatically incorporates the normalization at the origin and is well suited for a compactness argument.

\subsection{Fixed-point argument}
In this subsection we prove the existence of a local profile near the origin. The main point is to define a closed convex set of admissible functions that encodes the expected positivity, monotonicity, and Lipschitz bounds, and then to show that the fixed-point map preserves this set. Schauder's fixed-point theorem then yields a local solution. 

The equation \eqref{eqt:gmu} is equivalent to
\begin{equation}\label{eqt:gmu_for_fixedpoint}
\begin{aligned}
axg\,m_x &= \big(2xg_x+2(1-a)(g-1)\big)\,m,\\
2axg\,(u_x-1)_x &= \big(axg_x+ag+3a-2\big)\,(u_x-1)-4x^2 m,\\
(xg)_x &= u_x,
\end{aligned}
\end{equation}
with $2m(0)=1$ and $g(0)=1$. 
From the first line of \eqref{eqt:gmu_for_fixedpoint} we obtain
\[
m(x)=\frac{1}{2}\,g(x)^{\frac{2}{a}}
\exp\!\left(\frac{2(1-a)}{a}\int_0^x\frac{g(y)-1}{y\,g(y)}\,\mathrm{d}y\right).
\]
From the second line,
\begin{align*}
u'(x)
&= 1 - \frac{1}{a}\,x^{2-\frac{1}{a}}\,g(x)^{\frac{1}{2}}
\exp\!\left(\frac{2-3a}{2a}\int_0^x\frac{g(y)-1}{y\,g(y)}\,\mathrm{d}y\right) \\
&\quad \cdot \int_0^x y^{\frac{1}{a}-1}\,g(y)^{\frac{2}{a}-\frac{1}{2}}
\exp\!\left(\frac{2-a}{2a}\int_0^y\frac{g(z)-1}{z\,g(z)}\,\mathrm{d}z\right)\mathrm{d}y,
\end{align*}
and from the third line,
\[
g(x)=\frac{1}{x}\int_0^x u'(y)\,\mathrm{d}y,
\]
 Denote
\[
\phi(x):=\exp\!\left(\int_0^x\frac{g(y)-1}{y\,g(y)}\,\mathrm{d}y\right).
\]
Then $g$, $\phi$, and $u'$ satisfy
\begin{equation}\label{eqt:gmu_Iter}
\begin{aligned}
\phi(x) &= \exp\!\left(\int_0^x\frac{g(y)-1}{y\,g(y)}\,\mathrm{d}y\right),\\
u'(x) &= 1 - \frac{1}{a}\,x^{2-\frac{1}{a}}\,g(x)^{\frac{1}{2}}\,
\phi(x)^{\frac{2-3a}{2a}}
\int_0^x y^{\frac{1}{a}-1}\,g(y)^{\frac{4-3a}{2a}}\,
\phi(y)^{\frac{2-a}{2a}}\,\mathrm{d}y,\\
g(x) &= \frac{1}{x}\int_0^x u'(y)\,\mathrm{d}y.
\end{aligned}
\end{equation}
This representation leads naturally to a fixed–point approach for constructing self–similar profiles. The fixed-point formulation uses \(v=u_x\) as the primary unknown. Once \(v\) is constructed, \(g\), \(\phi\), and \(m\) are recovered from the integral formulas above. The following function space is tailored to the local fixed-point argument near the origin.  The normalization \(v(0)=1\) and the local form of the profile equations suggest that \(v\) should remain positive, nonincreasing on the positive side, and satisfy \(v(x)=1+O(x^2)\).  The integral formulation allows these local properties, together with a Lipschitz bound, to be propagated by the fixed-point map.  We impose these conditions only on a small interval \([-\ell,\ell]\), since the corresponding monotonicity assumptions are not expected to hold globally.  The constants \(\ell\), \(\eta\), and \(M\) are chosen precisely so that the estimates below close and \(\calR_a\) maps this fixed-point space into itself.

Let
\[
\mathbb{V}_1 := \big\{v\in C([-\ell,\ell]) : v(x)=v(-x)\big\}.
\]
Consider the closed, convex subset
\begin{equation*}
\begin{split}
\mathbb{D}_1 := \Big\{v\in \mathbb{V}_1:\ \quad &0\leq1-\eta x^2\leq v(x)\leq1\ \text{ on } [-\ell,\ell] ,\\\quad
 &v \text{ is nonincreasing on }[0,\ell] ,\quad
 |v(x)-v(y)|\leq M|x-y| \text{ for all } x,y\in[-\ell,\ell]
 \Big\},
\end{split}
\end{equation*}
where $\ell$, $\eta$ and $M$ are given by
\[
\ell:=\bigg(\frac{2}{3}\bigg)^{\frac{3a-2}{8a}}>0,\qquad \eta:=\bigg(\frac{2}{3}\bigg)^{-\frac{3a-2}{4a}}>0,\qquad \eta \ell^2=1,
\]
\[
0<M:=\frac{3(2a-1)\eta+2}{2a}\ell<+\infty.
\]
For $a\in [2/3,1]$, define the operators
\[
\calG(v)(x):=\frac{1}{x}\int_0^x v(y) \idiff y, 
\]
\[
\Phi(v)(x):=\exp\!\left(\int_0^x\frac{\calG[v](y)-1}{y\,\calG[v](y)}\,\mathrm{d}y\right),
\]
\[
\mathcal{R}_a(v)(x):=1-\frac{1}{a}\,x^{2-\frac{1}{a}}\,\calG[v](x)^{\frac{1}{2}}\,
\Phi[v](x)^{\frac{2-3a}{2a}}
\int_0^x y^{\frac{1}{a}-1}\,\calG[v](y)^{\frac{4-3a}{2a}}\,
\Phi[v](y)^{\frac{2-a}{2a}}\,\mathrm{d}y.
\]

The following lemmas show that \(\mathcal R_a\) has the compactness and continuity properties needed for Schauder's theorem. The estimates are designed to be uniform for \(a\in[2/3,1]\).

\begin{lemma}\label{lemma: g_properties}
For any $v\in \bbD_1$,
\[
1-\frac{\eta x^2}{3}\leq \calG(v)(x)\leq 1
\qquad \text{for all } x\in[0,\ell].
\]
Moreover, $\calG(v)$ is nonincreasing on $[0,\ell]$.
\end{lemma}

\begin{proof}
Let $g=\calG(v)$. Since $1-\eta x^2\leq v(x)\leq 1$ for $x\in[0,\ell]$,
\[
1-\frac{\eta x^2}{3}
\leq g(x)=\frac1x\int_0^x v(y)\idiff y
\leq 1.
\]
Moreover,
\[
g'(x)=\left(\frac1x\int_0^x v(y)\idiff y\right)'
=\frac{1}{x^2}\int_0^x\!\big(v(x)-v(y)\big)\idiff y\leq0,
\]
since $v$ is nonincreasing. Hence $g$ is nonincreasing on $[0,\ell]$.
\end{proof}

\begin{lemma}\label{lemma: phi_properties}
For any $v\in \bbD_1$,
\[
\sqrt{1-\frac{\eta x^2}{3}}\leq \Phi(v)(x)\leq 1
\qquad \text{for all } x\in[0,\ell].
\]
Moreover, $\Phi(v)$ is nonincreasing on $[0,\ell]$.
\end{lemma}

\begin{proof}
Let $g=\calG(v)$ and $\phi=\Phi(v)$. By Lemma~\ref{lemma: g_properties},
\[
1-\frac{\eta x^2}{3}\leq g(x)\leq 1.
\]
From the definition
\[
\phi(x)=\exp\!\left(\int_0^x \frac{1}{y}\!\left(1-\frac{1}{g(y)}\right)\!dy\right),
\]
we immediately obtain $\phi(x)\leq1$. Using $g(y)\geq 1-\frac{\eta y^2}{3}$ gives
\[
\int_0^x \frac{1}{y}\!\left(1-\frac{1}{g(y)}\right)dy
\geq \frac12\log\!\left(1-\frac{\eta x^2}{3}\right),
\]
which yields $\phi(x)\geq \sqrt{1-\frac{\eta x^2}{3}}$.

Finally,
\[
\phi'(x)=\phi(x)\frac{1}{x}\!\left(1-\frac{1}{g(x)}\right)\leq0
\]
since $g(x)\leq1$. Hence $\phi$ is nonincreasing on $[0,\ell]$.
\end{proof}

\begin{lemma}\label{lemma: r_properties}
Let $a\in[2/3,1]$ and $v\in\bbD_1$. Then for all $x\in[0,\ell]$,
\[
1-\eta x^2 \leq \calR_a(v)(x)\leq 1-x^2\Big(1-\frac{\eta x^2}{3}\Big)^{\frac{10-5a}{4a}} \leq 1.
\]
Moreover, $\calR_a(v)$ is nonincreasing on $[0,\ell]$.
\end{lemma}

\begin{proof}
Denote $g=\calG(v)$, $\phi=\Phi(v)$ and $r=\calR_a(v)$.  
By Lemmas~\ref{lemma: g_properties}--\ref{lemma: phi_properties},
\[
1-\frac{\eta x^2}{3}\leq g(x)\leq 1,
\qquad
\sqrt{1-\frac{\eta x^2}{3}}\leq \phi(x)\leq 1,
\]
and both $g$ and $\phi$ are nonincreasing on $[0,\ell]$.

\medskip
\noindent
\textit{Upper bound.}
For $0\leq y\leq x$, $g(y)\geq g(x)$ and $\phi(y)\geq\phi(x)$. Since $\frac{4-3a}{2a},\frac{2-a}{2a}\geq0$,
\[
g(y)^{\frac{4-3a}{2a}}\phi(y)^{\frac{2-a}{2a}}
\geq g(x)^{\frac{4-3a}{2a}}\phi(x)^{\frac{2-a}{2a}}.
\]
Hence
\[
\begin{aligned}
r(x)
&\leq 1-\frac1a x^{2-\frac1a} g(x)^{\frac12}\phi(x)^{\frac{2-3a}{2a}}
\int_0^x y^{\frac1a-1} g(x)^{\frac{4-3a}{2a}}\phi(x)^{\frac{2-a}{2a}}\idiff y \\
&=1-x^2 g(x)^{\frac2a-1}\phi(x)^{\frac{2(1-a)}{a}} \\
&\leq 1-x^2\Big(1-\frac{\eta x^2}{3}\Big)^{\frac{10-5a}{4a}}
\leq 1 .
\end{aligned}
\]

\medskip
\noindent
\textit{Lower bound.}
Using $g(x)\leq1$ and $\phi(\ell)\leq \phi(x)\leq1$,
\[
\begin{aligned}
r(x)
&\geq 1-\frac1a x^{2-\frac1a}\phi(\ell)^{\frac{2-3a}{2a}}
\int_0^x y^{\frac1a-1}\idiff y
=1-x^2\phi(\ell)^{\frac{2-3a}{2a}}.
\end{aligned}
\]
Since $\phi(\ell)\geq \sqrt{1-\frac{\eta\ell^2}{3}}$ and by the definitions of $\ell$ and $\eta$,
\[
\phi(\ell)^{\frac{2-3a}{2a}}
\leq\Big(1-\frac{\eta\ell^2}{3}\Big)^{\frac{2-3a}{4a}}
=\eta,
\]
we obtain $r(x)\geq 1-\eta x^2$.

\medskip
\noindent
\textit{Monotonicity.}
Differentiating gives
\[
r'(x)=
\frac{av(x)+3a-2}{2axg(x)}(r(x)-1)
-\frac1a x\,g(x)^{\frac{2-a}{a}}\phi(x)^{\frac{2(1-a)}{a}}\leq0,
\]
since $v(x)\geq0$, $g(x)>0$, and $r(x)\leq1$. Hence $r$ is nonincreasing on $[0,\ell]$.
\end{proof}

\begin{lemma}\label{lemma: r_Lipschitz}
Let $a\in[2/3,1]$. Then $\calR_a(f)(x)$ is Lipschitz on $[-\ell,\ell]$. In particular,
\[
|\calR_a(f)(x)-\calR_a(f)(y)| \leq M |x-y|
\quad \text{for all } x,y\in[-\ell,\ell].
\]
\end{lemma}
\begin{proof}
Denote $g=\calG(v)$, $\phi=\Phi(v)$, and $r=\calR_a(v)$.  
By Lemma~\ref{lemma: r_properties},
\[
r'(x)=
\frac{av(x)+3a-2}{2axg(x)}(r(x)-1)
-\frac1a x\,g(x)^{\frac{2-a}{a}}\phi(x)^{\frac{2(1-a)}{a}} .
\]
Since $r(x)\leq 1$, we obtain
\[
|r'(x)|
=
\frac{av(x)+3a-2}{2axg(x)}(1-r(x))
+\frac1a x\,g(x)^{\frac{2-a}{a}}\phi(x)^{\frac{2(1-a)}{a}} .
\]
Using $v(x)\leq v(0)$, $g(x)\geq g(\ell)$, and $1-r(x)\leq \eta x^2$, we estimate
\[
|r'(x)|
\leq
\frac{av(0)+3a-2}{2axg(\ell)}\,\eta x^2
+\frac1a x\,g(0)^{\frac{2-a}{a}}\phi(0)^{\frac{2(1-a)}{a}} .
\]
Since $v(0)=g(0)=\phi(0)=1$, this gives
\[
|r'(x)|
\leq
\frac{3(2a-1)\eta+2}{2a}\,x
\leq
\frac{3(2a-1)\eta+2}{2a}\,\ell
=M .
\]
Hence $\calR_a(v)$ is Lipschitz on $[-\ell,\ell]$ with Lipschitz constant $M$.
\end{proof}

\begin{proposition}\label{proposition: r_selfmap}
        Let $a\in[2/3,1]$. $\calR_a$ maps $\bbD_1$ into itself.
\end{proposition}
\begin{proof}
    This proposition follows from Lemmas~\ref{lemma: r_properties}--\ref{lemma: r_Lipschitz}.
\end{proof}

\begin{proposition}\label{proposition: r_continuity}
    Let $a\in[2/3,1]$. $\calR_a:\bbD_1\to\bbD_1$ is continuous with respect to the $L^\infty$-norm.
\end{proposition}

\begin{proof}
Fix $v_0\in\bbD_1$ and let $v\in\bbD_1$ with $\|v-v_0\|_{L^\infty}\leq \delta$.
Denote $g_0=\calG(v_0)$, $\phi_0=\Phi(v_0)$, $r_0=\calR_a(v_0)$ and $g,\phi,r$ similarly for $v$.

\medskip
\noindent
\textit{Control of $g$ and $\phi$ near $v_0$.}
From the definition of $\calG$,
\[
\|g-g_0\|_{L^\infty}\leq \|v-v_0\|_{L^\infty}\leq \delta.
\]
By Lemma~\ref{lemma: g_properties}, $g_0(x)\geq 1-\frac{\eta\ell^2}{3}=\frac23$ on $[0,\ell]$; hence for $\delta\leq 1/6$,
\[
g(x),g_0(x)\in\Big[\frac12,1\Big]\quad\text{on }[0,\ell].
\]
Using $\frac12\leq g,g_0\leq 1$ and
\[
\frac{g(y)-1}{y\,g(y)}-\frac{g_0(y)-1}{y\,g_0(y)}
=\frac{g(y)-g_0(y)}{y\,g(y)\,g_0(y)},
\]
we obtain
\[
\left|\frac{g(y)-1}{y\,g(y)}-\frac{g_0(y)-1}{y\,g_0(y)}\right|
\leq \frac{4}{y}\,|g(y)-g_0(y)|.
\]
On the other hand, by Lemma~\ref{lemma: g_properties},
\[
|g(y)-1|,\,|g_0(y)-1|\leq \frac{\eta}{3}y^2,
\]
so
\[
|g(y)-g_0(y)|\leq |g(y)-1|+|g_0(y)-1|
\leq \frac{2\eta}{3}y^2,
\]
and hence
\[
\left|\frac{g(y)-1}{y\,g(y)}-\frac{g_0(y)-1}{y\,g_0(y)}\right|
\leq C y .
\]
Let $\delta=\|g-g_0\|_{L^\infty}$. For any $x\in[0,\ell]$,
split the integral at $y=\sqrt{\delta}$ as follows.
For the first part, using the bound $Cy$,
\[
\left|\int_0^{\min(x,\sqrt{\delta})}\!\!\!\left(\frac{g-1}{yg}-\frac{g_0-1}{yg_0}\right)\idiff y\right|
\leq C\int_0^{\sqrt{\delta}} y\idiff y
\leq C\delta .
\]
For the second part, using $|g-g_0|\leq\delta$,
\[
\left|\int_{\sqrt{\delta}}^{x}\!\!\left(\frac{g-1}{yg}-\frac{g_0-1}{yg_0}\right)\idiff y\right|
\leq 4\delta\int_{\sqrt{\delta}}^{\ell}\frac{\diff y}{y}
\leq C\delta |\log\delta|.
\]
Combining the two bounds gives
\[
\sup_{x\in[0,\ell]}
\left|\int_0^x\left(\frac{g(y)-1}{y\,g(y)}-\frac{g_0(y)-1}{y\,g_0(y)}\right)\idiff y\right|
\leq C(\delta+\delta|\log\delta|)
\lesssim \sqrt{\delta},
\]
for $\delta$ sufficiently small.
Therefore,
\[
\|\phi-\phi_0\|_{L^\infty([0,\ell])}\;\lesssim\;\sqrt{\delta}.
\]

\medskip
\noindent
\textit{Continuity of $r=\calR_a(v)$.}
Write $r=1-\frac1a\,A\,B$ where
\[
A(x)=x^{2-\frac1a}\,g(x)^{\frac12}\,\phi(x)^{\frac{2-3a}{2a}},\qquad
B(x)=\int_0^x y^{\frac1a-1}g(y)^{\frac{4-3a}{2a}}\phi(y)^{\frac{2-a}{2a}}\idiff y,
\]
and define $A_0,B_0$ analogously.
Since $g,g_0,\phi,\phi_0$ stay in fixed compact intervals (by Lemmas~\ref{lemma: g_properties}--\ref{lemma: phi_properties}
and the above neighborhood choice), the power maps are Lipschitz there, hence
\[
\|A-A_0\|_{L^\infty}\;\lesssim\;\|g-g_0\|_{L^\infty}+\|\phi-\phi_0\|_{L^\infty}
\;\lesssim\;\delta+\sqrt{\delta}.
\]
Also, since $1/a-1\in[0,1/2]$, we have $\int_0^\ell y^{\frac1a-1}\idiff y<\infty$, and similarly
\[
\|B-B_0\|_{L^\infty}\leq \int_0^\ell y^{\frac1a-1}\,
\Big|g^{\frac{4-a}{2a}}\phi^{\frac{2-a}{2a}}-g_0^{\frac{4-a}{2a}}\phi_0^{\frac{2-a}{2a}}\Big|\idiff y
\;\lesssim\; \delta+\sqrt{\delta}.
\]
Finally,
\[
\|r-r_0\|_{L^\infty}
\leq \frac1a\Big(\|A-A_0\|_{L^\infty}\|B\|_{L^\infty}+\|A_0\|_{L^\infty}\|B-B_0\|_{L^\infty}\Big)
\;\lesssim\;\delta+\sqrt{\delta}\xrightarrow[]{\delta\to0}0.
\]
Thus $\calR_a$ is continuous at $v_0$, and since $v_0$ is arbitrary, $\calR_a$ is continuous on $\bbD_1$.
\end{proof}

\begin{lemma}\label{lemma: D_compactness}
    Let $a\in[2/3,1]$. The set $\bbD_1$ is compact with respect to the $L^\infty$-norm.
\end{lemma}
\begin{proof}
Let $\{v_n\}\subset \bbD_1$. Since every $v_n$ is $M$-Lipschitz on $[-\ell,\ell]$, the family $\bbD_1$ is equicontinuous. 
Moreover, $\bbD_1$ is uniformly bounded in $L^\infty([-\ell,\ell])$. By the Arzel\`a--Ascoli theorem, there exists a subsequence $v_{n_k}$ and a function $v\in C([-\ell,\ell])$ such that
\[
\|v_{n_k}-v\|_{L^\infty([-\ell,\ell])}\to0 .
\]Since $\bbD_1$ is closed in $L^\infty$, we have $v\in\bbD_1$. Thus every sequence in $\bbD_1$ has a convergent subsequence in $L^\infty$, so $\bbD_1$ is compact in the $L^\infty$ norm.
\end{proof}
\begin{theorem}\label{theorem: existence of fixed point}
    Let $a\in[2/3,1]$. The map $\calR_a$ has a fixed point $v\in\bbD_1$, i.e. $\calR_a(v)=v$.
\end{theorem}
\begin{proof}
    By Proposition \ref{proposition: r_continuity} and Lemma \ref{lemma: D_compactness}, $\bbD_1$ is a convex, closed and compact in the $L^\infty$-norm, and $\calR_a$ continuously maps $\bbD_1$ into itself. The Schauder fixed-point theorem then guarantees that $\calR_a$ has a fixed point in $\bbD_1$. 
\end{proof}

The fixed point constructed above gives a solution of the normalized profile system only on a small symmetric interval around the origin. To obtain a complete periodic profile, we must continue this solution until the velocity profile reaches the endpoint of one half-period. This continuation is governed by an autonomous first-order ODE system away from the origin, where the singularity of the original formulation is no longer present.

\subsection{Extension of the solution}

We have already obtained a solution to equation \eqref{eqt:gmu} on $[-\ell,\ell]$, and now aim to extend this solution to the region $|x|>\ell$. 
Since the solution is already away from the singular point \(x=0\), the equations can be treated as a standard first-order ODE system as long as \(g>0\). The key issue is therefore to understand whether and how \(g\) can approach zero, which will determine the endpoint of the half-period.

For $axg(x)\neq0$ the equation \eqref{eqt:gmu} is equivalent to the first--order ODE system for $(m,v,g)$, where $v:=u_x$,
\begin{equation}\label{eqt:gmu_for_extension}
\begin{aligned}
m_x &= \frac{2(v-g)+2(1-a)(g-1)}{axg}\,m,\\
v_x &= \frac{(av+3a-2)(v-1)-4x^2 m}{2axg},\\
g_x &= \frac{v-g}{x}.
\end{aligned}
\end{equation}
\begin{proposition}\label{proposition: local_existence}
Let $x_0>0$ and let $m_0,v_0,g_0$ satisfy $g_0>0$. Then there exists $\eps>0$ such that the system \eqref{eqt:gmu_for_extension} admits a unique solution $(m,v,g)\in C^1([x_0,x_0+\eps])$ satisfying
\[
m(x_0)=m_0,\qquad v(x_0)=v_0,\qquad g(x_0)=g_0.
\]
\end{proposition}

\begin{proof}
Denote by $F(x,m,v,g)$ the right-hand side of \eqref{eqt:gmu_for_extension}.  
Since $x_0>0$ and $g_0>0$, we may choose $\eps>0$ such that
$x\geq x_0/2,\, g\geq g_0/2$
in a neighborhood of $(x_0,m_0,v_0,g_0)$. On this region all denominators $x$, $x^2$, and $g$ in \eqref{eqt:gmu_for_extension} are bounded away from zero. Hence $F$ is $C^1$ in $(m,v,g)$ and therefore locally Lipschitz in $(m,v,g)$ by standard product estimate. 

By the Picard--Lindel\"of theorem, the ODE system admits a unique solution $(m,v,g)\in C^1([x_0,x_0+\eps])$ satisfying the given initial data.
\end{proof}

The local existence result allows us to continue the solution as long as \(g\) remains positive and \((m,v,g)\) are bounded. We next derive monotonicity properties that prevent undesired behavior during the continuation and provide the estimates needed to identify the endpoint.

\begin{lemma}\label{lemma: monotonicity_extension}
Let $a\in[2/3,1]$. Suppose the system \eqref{eqt:gmu_for_extension} admits a solution
\[
(m,v,g)\in C^1((\ell,X))
\]
such that
\[
m(\ell)>0,\qquad v(\ell)\leq g(\ell)<1,\qquad g(x)>0 \quad \text{for all }x\in(\ell,X).
\]
Then on $(\ell,X)$ the function $g$ and the ratio $m/g^2$ are monotonically decreasing.
\end{lemma}

\begin{proof}
From the third equation of \eqref{eqt:gmu_for_extension}, we have
\[
g_x=\frac{v-g}{x},
\]
we obtain
\[
(xg)'=v,
\]
and therefore
\begin{equation}\label{eq:g_integral}
xg(x)=\ell g(\ell)+\int_\ell^x v(y)\idiff y .
\end{equation}

Using again $xg_x=v-g$, the equation for $m$ can be rewritten as
\[
\frac{m_x}{m}
=
\frac{2g_x}{ag}+\frac{2(1-a)(g-1)}{axg}.
\]
Integrating from $\ell$ to $x$ gives
\begin{equation}\label{eq:m_integral}
m(x)
=
m(\ell)\Bigl(\frac{g(x)}{g(\ell)}\Bigr)^{\frac{2}{a}}
\exp\!\left(
\frac{2(1-a)}{a}\int_\ell^x\frac{g(y)-1}{y\,g(y)}\idiff y
\right),
\end{equation}
which implies $m(x)>0$ on $(\ell,X)$.

We next show that $v(x)<1$ on $(\ell,X)$. Otherwise let $x_1$ be the first point where
$v(x_1)=1$. Then $v_x(x_1)\geq0$, while the second equation of
\eqref{eqt:gmu_for_extension} yields
\[
v_x(x_1)
=
-\frac{2x_1m(x_1)}{ag(x_1)}<0,
\]
since $m(x_1)>0$ and $g(x_1)>0$, a contradiction. Hence $v<1$ everywhere.

If $v(x)\geq0$, then since $a\in[2/3,1]$ we have $(av+3a-2)\geq0$ and $v-1<0$, and thus
\[
v_x=\frac{(av+3a-2)(v-1)-4x^2m}{2axg}<0.
\]
Hence $v$ is decreasing on every interval where it is nonnegative. Using
\eqref{eq:g_integral} and the fact that $v(y)\geq v(x)$ for $y\leq x$, we obtain
\[
xg(x)=\ell g(\ell)+\int_\ell^x v(y)\idiff y
\geq \ell g(\ell)+(x-\ell)v(x).
\]
Since $v(x)\leq v(\ell)\leq g(\ell)$ and $v(x)<1$, this implies $g(x)>v(x)$, and therefore
\[
g_x(x)=\frac{v(x)-g(x)}{x}<0.
\]

If $v(x_0)=0$ for some $x_0$, then
\[
v_x(x_0)
=
\frac{(3a-2)(-1)-4x_0^2m(x_0)}{2ax_0g(x_0)}<0,
\]
so $v$ crosses $0$ strictly downward and remains negative afterwards. In this case
$g_x=(v-g)/x<0$ trivially. Hence $g$ is decreasing on $(\ell,X)$.

Finally, differentiating $m/g^2$ using \eqref{eq:m_integral} gives
\[
\left(\frac{m}{g^2}\right)_x
=
\left(\Bigl(\frac{2}{a}-2\Bigr)\cdot\frac{g_x}{g}
+
\frac{2(1-a)}{a}\cdot\frac{g-1}{xg}\right)\frac{m}{g^2}.
\]
Since $a\in[2/3,1]$, $g_x\leq0$, and $g(x)<1$, the right-hand side is nonpositive, and
therefore $(m/g^2)_x\leq0$.
\end{proof}

We now combine the local continuation criterion with the monotonicity estimates. The result is a maximal extension on an interval \((-L,L)\), where the endpoint behavior is determined by the vanishing of \(g\).

\begin{corollary}\label{corollary: global_extension}
Let $a\in[2/3,1]$. Then the system \eqref{eqt:gmu_for_extension} admits a solution
\[
(m,v,g)\in C^1((-L,L))
\]
on a maximal interval
$(-L,L)$ with $ 0<L\leq +\infty,$
such that
\[
\lim_{x\to L^-} m(x)=0,\qquad
\lim_{x\to L^-} v(x)=-\frac{3a-2}{a},\qquad
\lim_{x\to L^-} g(x)=0.
\]
\end{corollary}

\begin{proof}
By Proposition \ref{proposition: local_existence}, the solution constructed on
$[-\ell,\ell]$ extends uniquely to a maximal interval
\[
(-L,L),\qquad 0<L\leq +\infty,
\]
on which $(m,v,g)\in C^1((-L,L))$ and $g>0$. By symmetry, it suffices to study the
behavior as $x\to L^-$.

Applying Lemma \ref{lemma: monotonicity_extension} on $(\ell,L)$ yields
\[
m(x)>0,\qquad v(x)<g(x)<1,\qquad g_x(x)<0,\qquad
\left(\frac{m}{g^2}\right)_x(x)\leq0 .
\]
Hence $g$ is positive and decreasing on $(\ell,L)$, and therefore
\[
g_*:=\lim_{x\to L^-} g(x)
\]
exists with $g_*\geq0$.

\medskip

To analyze $v$, define
\[
\phi(x):=\exp\!\left(\int_0^x\frac{g(y)-1}{y g(y)}\idiff y\right),
\qquad
q(x):=a\,v(x)+3a-2 .
\]
Using
\[
v_x=\frac{(av+3a-2)\bigl((xg)_x-1\bigr)-4x^2 m}{2axg},
\qquad (xg)_x=v,
\]
we obtain
\[
q_x-\frac{(xg)_x-1}{2xg}q=-\frac{2xm}{g}.
\]
Since
\[
\frac{(xg)_x-1}{2xg}
=
\frac{xg_x+g-1}{2xg}
=
\frac{g_x}{2g}+\frac{g-1}{2xg}
=
\frac12\frac{(g\phi)_x}{g\phi},
\]
it follows that
\[
\left(\frac{q}{\sqrt{g\phi}}\right)_x
=
-\frac{2xm}{g^{3/2}\phi^{1/2}}.
\]
Integrating from $\ell$ to $x$ yields
\begin{equation}\label{eq:q_representation_in_cor}
q(x)
=
\sqrt{g(x)\phi(x)}
\left(
\frac{q(\ell)}{\sqrt{g(\ell)\phi(\ell)}}
-
2\int_\ell^x
\frac{y\,m(y)}{g(y)^{3/2}\phi(y)^{1/2}}
\idiff y
\right).
\end{equation}

\medskip

We first show that $g_*=0$.  
The proof splits into two cases depending on whether $L<+\infty$ or $L=+\infty$.

\medskip
\noindent
\textit{Case 1: $L<+\infty$.}

Assume $g_*>0$. Then there exists $x_0\in(\ell,L)$ such that
\[
g(x)\geq\frac{g_*}{2},\qquad x\in[x_0,L).
\]
Since $g<1$, the function $\phi$ is decreasing and bounded above by $\phi(x_0)$.
Because $g$ stays bounded away from zero and $L<\infty$, $\phi$ is also bounded
below by a positive constant on $[x_0,L)$.

Using \eqref{eq:q_representation_in_cor} with $x_0$ in place of $\ell$, the integral
term remains bounded. Hence $q$ and therefore $v$ remain bounded near $L$.
Together with
\[
m(x)\leq \tfrac12,\qquad g(x)\geq \tfrac{g_*}{2},
\]
this implies that the right-hand side of \eqref{eqt:gmu_for_extension} remains
bounded, contradicting maximality. Hence $g_*=0$.

\medskip
\noindent
\textit{Case 2: $L=+\infty$.}

Assume $g_*>0$. Since $g$ is decreasing and $g(x)\leq g(\ell)<1$,
\[
\frac{g(x)-1}{xg(x)}
\leq
-\Bigl(\frac1{g(\ell)}-1\Bigr)\frac1x .
\]
Hence $\phi(x)\to0$ as $x\to\infty$. From \eqref{eq:q_representation_in_cor},
\[
q(x)\leq
\sqrt{g(x)\phi(x)}
\frac{q(\ell)}{\sqrt{g(\ell)\phi(\ell)}} .
\]
Therefore,
\begin{equation}\label{eq:limsup-v-upper}
\limsup_{x\to\infty} v(x)
\leq
-\frac{3a-2}{a}
=:v_* .
\end{equation}
This estimate will be used below to rule out the possibility that the solution extends to infinity.

Since $(xg)'=v$, we have
\[
g(x)
=
\frac{\ell g(\ell)}{x}
+\frac1x\int_\ell^x v(y)\idiff y .
\]
For large $x$, by the definition \eqref{eq:limsup-v-upper} of $v_*$, we have
\[
v(x)\leq v_*+\varepsilon .
\]
Consequently
\[
g_*\leq v_*+\varepsilon .
\]
Letting $\varepsilon\to0$ gives $g_*\leq v_*\leq0$, contradicting $g_*>0$.
Hence $g_*=0$.

\medskip

Therefore
\[
\lim_{x\to L^-} g(x)=0 .
\]

Since $\left(m/g^2\right)_x\leq0$,
\[
0<m(x)\leq
\frac{m(\ell)}{g(\ell)^2}g(x)^2,
\]
and hence
\[
\lim_{x\to L^-} m(x)=0 .
\]

\medskip

It remains to determine $\lim_{x\to L^-}v(x)$.
Using
\[
m(x)=\frac12 g(x)^{\frac{2}{a}}\phi(x)^{\frac{2(1-a)}{a}},
\]
representation \eqref{eq:q_representation_in_cor} becomes
\[
q(x)
=
\sqrt{g(x)\phi(x)}
\left(
\frac{q(\ell)}{\sqrt{g(\ell)\phi(\ell)}}
-
\int_\ell^x
y\,g(y)^{\frac{2}{a}-\frac32}
\phi(y)^{\frac{4-5a}{2a}}
\idiff y
\right).
\]
Since $g(x)\to0$ and $0<\phi(x)\leq\phi(\ell)$,
\[
\sqrt{g(x)\phi(x)}\to0 .
\]
Because $\phi$ is decreasing, for $\ell\leq y\leq x$ we have
$\phi(x)\leq\phi(y)$, and therefore
\[
\sqrt{g(x)\phi(x)}
\int_\ell^x
y\,g(y)^{\frac{2}{a}-\frac32}
\phi(y)^{\frac{4-5a}{2a}}
\idiff y
\leq
\sqrt{g(x)}
\int_\ell^x
y\,g(y)^{\frac{2}{a}-\frac32}
\phi(y)^{\frac{2(1-a)}{a}}
\idiff y .
\]
To estimate the integral term, 
we again distinguish two cases depending on whether $a<1$ or $a=1$.

\medskip
\noindent
\textit{Case 1: $2/3<a<1$.}

Since $\frac{2}{a}-\frac32\geq\frac12$ and $0<g(y)\leq 1$,
\[
g(y)^{\frac{2}{a}-\frac32}\leq1 .
\]
Thus
\[
\sqrt{g(x)\phi(x)}
\int_\ell^x
y\,g(y)^{\frac{2}{a}-\frac32}
\phi(y)^{\frac{4-5a}{2a}}
\idiff y\leq 
\sqrt{g(x)}
\int_\ell^x
y\,\phi(y)^{\frac{2(1-a)}{a}}
\idiff y .
\]
Because $g(x)\to0$, for sufficiently large $x$
\[
g(x)\leq\frac{1-a}{2-a}.
\]
This implies
\[
\phi(x)\leq Cx^{-\frac1{1-a}} .
\]
Hence
\[
y\,\phi(y)^{\frac{2(1-a)}{a}}
\leq
Cy^{1-\frac2a}.
\]
Since $a<1$, the exponent is less than $-1$, and the integral from $\ell$ to $+\infty$ is finite.
Therefore the right-hand side is bounded by $C\sqrt{g(x)}\to0$,
so
\[
q(x)\to0 \qquad \text{ as } x\to +\infty.
\]
Since $q=a\,v+3a-2$, we conclude
\[
\lim_{x\to L^-}v(x)
=
-\frac{3a-2}{a}.
\]

\medskip
\noindent
\textit{Case 2: $a=1$.}

Then $q=v+1$. From the previous argument
\[
\limsup_{x\to L^-} v(x)\leq -1 .
\]
If $L=+\infty$, then for some $X$
\[
v(x)\leq-\tfrac12,\qquad x\geq X .
\]
Since $(xg)'=v$,
\[
xg(x)
=
Xg(X)+\int_X^x v(y)\idiff y
\leq
Xg(X)-\tfrac12(x-X),
\]
which becomes negative for large $x$, contradicting $g>0$.
Hence $L<\infty$.
Returning to the estimate above and substituting $a=1$,
\[
\sqrt{g(x)\phi(x)}
\int_\ell^x
y\,g(y)^{\frac12}\phi(y)^{-\frac12}
\idiff y
\leq
\sqrt{g(x)}
\int_\ell^x
y\,g(y)^{\frac12}\idiff y .
\]
Since $L<\infty$ and $0<g(y)<1$,
\[
\int_\ell^x y\,g(y)^{\frac12}\idiff y
\leq
\int_\ell^L y\idiff y
<+\infty .
\]
Hence the right-hand side is bounded by $C\sqrt{g(x)}\to0$.
Therefore
\[
q(x)\to0 .
\]
Since $q=v+1$, we conclude
\[
\lim_{x\to L^-}v(x)=-1=-\frac{3a-2}{a}.
\]
This completes the proof.
\end{proof}

The preceding corollary identifies the possible endpoint behavior, but it does not yet determine whether the maximal interval is finite. This depends on the parameter \(a\). The next proposition distinguishes the supercritical case \(a>2/3\), where the half-period is finite, from the critical case \(a=2/3\), where the interval extends to infinity.

\begin{proposition}\label{proposition: maximal_interval}
Under the assumptions of Lemma \ref{lemma: monotonicity_extension}, the maximal interval $(-L,L)$ satisfies
\begin{enumerate}
\item $L<+\infty$ if $a\in(2/3,1]$;
\item $L=+\infty$ if $a=2/3$.
\end{enumerate}
\end{proposition}

\begin{proof}
We first prove \((1)\), namely that \(L<+\infty\) for \(a\in(2/3,1]\).
Arguing by contradiction, assume that
\[
L=+\infty.
\]
From the previous analysis we already know that
\[
\lim_{x\to+\infty} g(x)=0
\]
and
\[
\limsup_{x\to+\infty} v(x)\leq -\frac{3a-2}{a}.
\]
Since \(a\in(2/3,1]\), we have
\[
\frac{3a-2}{a}>0.
\]
Choose \(\varepsilon>0\) such that
\[
-\frac{3a-2}{a}+\varepsilon<0.
\]
Then there exists \(X>0\) such that
\[
v(x)\leq -\frac{3a-2}{a}+\varepsilon<0,
\qquad x\geq X.
\]

Recalling that \(u=xg\) and \(u_x=v\), we obtain
\[
u(x)=u(X)+\int_X^x v(y)\idiff y
\leq
u(X)+\Bigl(-\frac{3a-2}{a}+\varepsilon\Bigr)(x-X).
\]
Since the coefficient of \(x-X\) is negative, the right-hand side becomes
negative for all sufficiently large \(x\), contradicting \(u=xg>0\).
Therefore \(L<+\infty\) for all \(a\in(2/3,1]\).

\medskip

We now prove \((2)\), namely that \(L=+\infty\) when \(a=\frac23\).
In this case the system becomes
\begin{equation}\label{eq:critical_system_for_L}
\begin{aligned}
m_x &= \frac{3v-2g-1}{xg}\,m,\\
v_x &= \frac{v(v-1)-6x^2m}{2xg},\\
g_x &= \frac{v-g}{x}.
\end{aligned}
\end{equation}
Recall that
\[
u=xg,\qquad u_x=v.
\]
We distinguish two cases according to whether \(v\) stays positive or changes sign.

\medskip
\noindent
\textit{Case 1: \(v(x)>0\) for all \(x\in(\ell,L)\).}

Since \(u_x=v>0\), the function \(u=xg\) is increasing on \((\ell,L)\).
Hence
\[
u(x)\geq u(\ell)=\ell g(\ell)>0
\qquad\text{for all }x\in(\ell,L).
\]
In particular \(xg=u\) stays bounded away from zero on every bounded
subinterval. Since also \(m>0\) and \(0<v<g<1\), the right-hand side of
\eqref{eq:critical_system_for_L} remains bounded there. By
Proposition~\ref{proposition: local_existence}, the solution can be
continued as long as \(x\) stays finite. Hence \(L=+\infty\).

\medskip
\noindent
\textit{Case 2: \(v\) vanishes somewhere on \((\ell,L)\).}

Let \(x_0\in(\ell,L)\) be the first point such that
\[
v(x_0)=0.
\]
Evaluating the equation for \(v_x\) at \(x_0\) gives
\[
v_x(x_0)
=
\frac{-6x_0^2m(x_0)}{2x_0g(x_0)}
=
-\frac{3x_0m(x_0)}{g(x_0)}
<0.
\]
Hence \(v\) crosses the axis strictly downward at \(x_0\), and therefore
\[
v(x)<0
\qquad\text{for all }x\in(x_0,L).
\]

We claim that \(L=+\infty\). Suppose for contradiction that \(L<+\infty\).
From the previous analysis for the finite maximal interval case we know
\[
\lim_{x\to L^-} g(x)=0,
\qquad
\lim_{x\to L^-} v(x)=0.
\]
Since \(u=xg\), this implies
\[
\lim_{x\to L^-}u(x)=0.
\]

On \((x_0,L)\) we compute
\[
\left(\log\frac{m}{v^2}\right)_x
=
\frac{m_x}{m}-2\frac{v_x}{v}
=
\frac{2(v-g)}{xg}
+
\frac{6x^2m}{xgv}.
\]
Because \(v<0\), \(g>0\), and \(v<g\), both terms on the right-hand side
are negative. Hence
\[
\left(\log\frac{m}{v^2}\right)_x<0
\qquad\text{on }(x_0,L),
\]
so \(m/v^2\) is decreasing there. Consequently, for any fixed
\(x_1\in(x_0,L)\) there exists \(C_0>0\) such that
\[
m(x)\leq C_0 v(x)^2,
\qquad x\in[x_1,L).
\]

Since \(L<+\infty\) and \(v(x)\to0\), it follows that
\[
x^2m(x)=O(v(x)^2)=o(-v(x))
\qquad\text{as }x\to L^-.
\]
Thus by choosing \(x_1\) sufficiently close to \(L\) we may assume that
\[
6x^2m(x)\leq -\frac{v(x)}2,
\qquad
-\frac12\leq v(x)<0,
\qquad x\in[x_1,L).
\]
Then
\[
v(x)-1\leq -1,
\qquad
v(x)(v(x)-1)\geq -v(x).
\]
Using the equation for \(v_x\), we obtain
\[
v_x
=
\frac{v(v-1)-6x^2m}{2u}
\geq
\frac{-v-6x^2m}{2u}
\geq
\frac{-v}{4u}.
\]
Since \(u_x=v<0\), we have
\[
\frac{dv}{du}
=
\frac{v_x}{u_x}
\leq
-\frac{1}{4u},
\qquad
\frac{d(-v)}{du}\geq\frac{1}{4u}.
\]
Integrating from \(x\) to \(x_1\) gives
\[
-v(x_1)+v(x)
\geq
\frac14\log\frac{u(x_1)}{u(x)}.
\]
As \(x\to L^-\), we have \(u(x)\to0\) and \(v(x)\to0\), so the left-hand
side tends to
\[
-v(x_1)<+\infty,
\]
while the right-hand side tends to \(+\infty\), a contradiction.
Therefore \(L=+\infty\).

This completes the proof.
\end{proof}

\subsection{Regularity of the solution}
We have so far constructed a \(C^1\) solution and determined its maximal extension. We now upgrade the regularity. Near the origin this requires using the fixed-point formulation, which captures the cancellation of the apparent singular terms. Away from the origin, smoothness follows from the ODE system because \(x\) and \(g\) stay away from zero on compact subintervals of \((-L,L)\).

\begin{proposition}
\label{proposition: regularity_solution}
Let $a\in[2/3,1]$, and let $(m,v,g)\in C^1((-L,L))$ be the solution obtained in
Theorem~\ref{theorem: existence of fixed point} and
Corollary~\ref{corollary: global_extension}. Then
\[
m,\ v,\ g\in C^\infty((-L,L)).
\]
\end{proposition}

\begin{proof}
We establish the regularity near the origin using the fixed-point equation
and away from the origin using the ODE system.

\medskip
\noindent
\textit{Step 1: smoothness on $(-\ell,\ell)$.}

On $[-\ell,\ell]$, the function $v$ is a fixed point of $\calR_a$, hence
\[
v(x)=1-\frac1a\,x^{2-\frac1a} g(x)^{\frac12}\phi(x)^{\frac{2-3a}{2a}}
\int_0^x y^{\frac1a-1}g(y)^{\frac{4-3a}{2a}}\phi(y)^{\frac{2-a}{2a}}\idiff y,
\]
where
\[
g(x)=\frac1x\int_0^x v(y)\idiff y,
\qquad
\phi(x)=\exp\!\left(\int_0^x\frac{g(y)-1}{y\,g(y)}\idiff y\right),
\qquad
m(x)=\frac12\,g(x)^{\frac2a}\phi(x)^{\frac{2(1-a)}{a}}.
\]

Since $v\in\bbD_1$, we know that $v$ is even and Lipschitz on $[-\ell,\ell]$, and
\[
1-\eta x^2\leq v(x)\leq 1.
\]
In particular,
\[
1-v(x)=O(x^2)\qquad (x\to0).
\]
Hence the function
\[
f(x):=\frac{1-v(x)}{x^2}
\]
extends continuously to $x=0$ by its limit there, so that
\[
v(x)=1-x^2f(x),\qquad f\in C([-\ell,\ell]).
\]

We prove by induction that
\[
f\in C^k([-\ell,\ell])\quad\Longrightarrow\quad f\in C^{k+1}([-\ell,\ell]).
\]
Since the case $k=0$ already holds, the above induction argument will imply
$f\in C^\infty([-\ell,\ell])$ and therefore
\[
v(x)=1-x^2f(x)\in C^\infty([-\ell,\ell]).
\]

Assume now that $f\in C^k([-\ell,\ell])$ for some $k\geq0$.
Using
\[
v(x)=1-x^2f(x),
\]
we rewrite $g$ as
\[
g(x)=\frac1x\int_0^x v(y)\idiff y
=1-\frac1x\int_0^x y^2f(y)\idiff y.
\]
Let
\[
F(x):=\int_0^x f(y)\idiff y.
\]
Since $f\in C^k([-\ell,\ell])$, we have
\[
F\in C^{k+1}([-\ell,\ell]).
\]
Integrating by parts gives
\[
\int_0^x y^2f(y)\idiff y
=
x^2F(x)-2\int_0^x yF(y)\idiff y.
\]
Hence
\[
g(x)
=
1-xF(x)+\frac{2}{x}\int_0^x yF(y)\idiff y.
\]
After the change of variables $y=tx$ in the last integral we obtain
\[
g(x)=1-xF(x)+2x\int_0^1 t\,F(tx)\,dt.
\]
Since $F\in C^{k+1}([-\ell,\ell])$, differentiation under the integral sign
shows that
\[
\int_0^1 t\,F(tx)\,dt\in C^{k+1}([-\ell,\ell]),
\]
hence
\[
g\in C^{k+1}([-\ell,\ell]),\qquad \frac{g-1}{x}\in C^{k+1}([-\ell,\ell]).
\]

Next consider $\phi$. Note that
\[
\frac{g(x)-1}{xg(x)}
\]
belongs to $C^{k+1}([-\ell,\ell])$ since $g>0$ on $[-\ell,\ell]$.
Therefore
\[
\phi(x)=\exp\!\left(\int_0^x\frac{g(y)-1}{y\,g(y)}\idiff y\right)
\in C^{k+2}([-\ell,\ell]),
\]
and in particular
\[
\phi\in C^{k+1}([-\ell,\ell]).
\]

Since $g,\phi>0$, the identity
\[
m(x)=\frac12\,g(x)^{\frac2a}\phi(x)^{\frac{2(1-a)}{a}}
\]
gives
\[
m\in C^{k+1}([-\ell,\ell]).
\]

We now return to the fixed-point equation. Define
\[
A(x):=g(x)^{\frac{4-3a}{2a}}\phi(x)^{\frac{2-a}{2a}}.
\]
Then
\[
A\in C^{k+1}([-\ell,\ell]).
\]
Since $v$ is even, the functions $g$, $\phi$, and $A$ are also even.
Define
\[
B(x):=x^{-\frac1a}\int_0^x y^{\frac1a-1}A(y)\idiff y,
\]
with the value at $x=0$ understood through the limit.
After the change of variables $y=tx$ we obtain
\[
B(x)=\int_0^1 t^{\frac1a-1}A(tx)\,dt.
\]
Since $a\in[2/3,1]$, the weight $t^{\frac1a-1}$ is integrable, hence
\[
B\in C^{k+1}([-\ell,\ell]).
\]
The fixed-point formula becomes
\[
v(x)=1-\frac1a\,x^2 g(x)^{\frac12}\phi(x)^{\frac{2-3a}{2a}}B(x),
\]
so that
\[
f(x)=\frac{1-v(x)}{x^2}
=\frac1a\,g(x)^{\frac12}\phi(x)^{\frac{2-3a}{2a}}B(x).
\]
Since $g,\phi,B\in C^{k+1}([-\ell,\ell])$, we conclude that
\[
f\in C^{k+1}([-\ell,\ell]).
\]

This completes the bootstrap and shows that
\[
m,v,g\in C^\infty((-\ell,\ell)).
\]

\medskip
\noindent
\textit{Step 2: smoothness on $(\ell/2,L)$.}

On $(\ell/2,L)$ the functions $(m,v,g)$ satisfy the ODE system \eqref{eqt:gmu_for_extension}
\begin{equation}
\begin{aligned}
m_x &= \frac{2(v-g)+2(1-a)(g-1)}{axg}\,m,\\
v_x &= \frac{(av+3a-2)(v-1)-4x^2 m}{2axg},\\
g_x &= \frac{v-g}{x}.
\end{aligned}
\end{equation}
Since $x>\ell/2>0$ and $g(x)>0$ on $(\ell/2,L)$, the right-hand side of
\eqref{eqt:gmu_for_extension} is a smooth function of
$(x,m,v,g)$. Standard regularity theory for ODEs therefore implies that any
$C^1$ solution of this system is in fact $C^\infty$ on $(\ell/2,L)$.
Hence
\[
m,\ v,\ g\in C^\infty((\ell/2,L)).
\]

Now combining the two regions and using the symmetry of the construction gives
\[
m,\ v,\ g\in C^\infty((-L,L)),
\]
which is desired.
\end{proof}

For \(a>2/3\), the half-period endpoint \(L\) is finite. To understand the regularity of the periodic extension, we need more precise asymptotics as \(x\to L^-\). The following proposition gives the leading-order behavior of \(g\), \(m\), \(v\), and \(v_x\) at the endpoint.

\begin{proposition}
\label{proposition: boundary_asymptotics}

Let $a\in(2/3,1]$, and let $(m,v,g)\in C^\infty((-L,L))$ be the solution obtained in
Theorem~\ref{theorem: existence of fixed point} and
Corollary~\ref{corollary: global_extension}
with $L<+\infty$. Then, as $x\to L^-$,

\begin{enumerate}
\item The function $g$ has the asymptotic behavior
\[
g(x)=\frac{3a-2}{aL}(L-x)+o(L-x).
\]

\item
There exists a constant
\[
c_m>0
\]
such that
\[
m(x)=c_m\,(L-x)^{\frac2a+\frac{2(1-a)}{3a-2}}+o\big((L-x)^{\frac2a+\frac{2(1-a)}{3a-2}}\big).
\]

\item
There exists a constant $c_v\in\mathbb{R}$ such that
\[
v(x)=-\frac{3a-2}{a}+c_v\,(L-x)^{\frac{2a-1}{3a-2}}
+o\!\left((L-x)^{\frac{2a-1}{3a-2}}\right).
\]

\item
The same constant $c_v$ satisfies
\[
v_x(x)=-\frac{2a-1}{3a-2}\,c_v\,
(L-x)^{\frac{1-a}{3a-2}}
+o\!\left((L-x)^{\frac{1-a}{3a-2}}\right).
\]
In particular, if $a\in(2/3,1)$, then
\[
v_x(x)\to0
\qquad\text{as }x\to L^-.
\]
\end{enumerate}
\end{proposition}

\begin{proof}
By Corollary~\ref{corollary: global_extension}, we already know that
\[
g(x)\to0,\qquad m(x)\to0,\qquad v(x)\to-\frac{3a-2}{a}
\qquad\text{as }x\to L^-.
\]

We first prove \textup{(1)}. Since $g(x)\to0$ and $L-x\to0$, l'Hospital's rule gives
\[
\lim_{x\to L^-}\frac{g(x)}{L-x}
=
\lim_{x\to L^-}\frac{g_x(x)}{-1}.
\]
Using
\[
g_x=\frac{v-g}{x},
\]
we obtain
\[
\lim_{x\to L^-}\frac{g(x)}{L-x}
=
-\lim_{x\to L^-}\frac{v(x)-g(x)}{x}
=
-\frac{-\frac{3a-2}{a}-0}{L}
=
\frac{3a-2}{aL}.
\]
This proves \textup{(1)}.

We next introduce
\[
\phi(x):=\exp\!\left(\int_0^x\frac{g(y)-1}{y\,g(y)}\idiff y\right).
\]
Since
\[
\frac{\phi_x(x)}{\phi(x)}=\frac{g(x)-1}{x\,g(x)},
\]
and $g(x)\to0$, we have $\phi(x)\to0$ as $x\to L^-$. Hence both
$\log\phi(x)$ and $\log(L-x)$ tend to $-\infty$, so l'Hospital's rule yields
\[
\lim_{x\to L^-}\frac{\log\phi(x)}{\log(L-x)}
=
\lim_{x\to L^-}
\frac{\phi_x(x)/\phi(x)}{-1/(L-x)}
=
-\lim_{x\to L^-}(L-x)\frac{g(x)-1}{x\,g(x)}.
\]
By \textup{(1)}, we have
\[
x\,g(x)
=
L\cdot \frac{3a-2}{aL}(L-x)+o(L-x)
=
\frac{3a-2}{a}(L-x)+o(L-x),
\]
and therefore
\[
-\,(L-x)\frac{g(x)-1}{x\,g(x)}
\to
\frac{a}{3a-2}.
\]
Thus
\begin{equation}\label{equation: phi_boundary_log}
\lim_{x\to L^-}\frac{\log\phi(x)}{\log(L-x)}=\frac{a}{3a-2}.
\end{equation}

We now obtain a first estimate for
\[
q(x):=av(x)+3a-2.
\]
From the representation formula proved in
Corollary~\ref{corollary: global_extension}, we have
\[
q(x)
=
\sqrt{g(x)\phi(x)}
\left(
\frac{q(\ell)}{\sqrt{g(\ell)\phi(\ell)}}
-
\int_\ell^x
y\,g(y)^{\frac2a-\frac32}\phi(y)^{\frac{4-5a}{2a}}\idiff y
\right).
\]
By \textup{(1)} and \eqref{equation: phi_boundary_log},
\[
\lim_{x\to L^-}
\frac{
\log\!\left(g(x)^{\frac2a-\frac32}\phi(x)^{\frac{4-5a}{2a}}\right)
}{
\log(L-x)
}
=
\frac2a-\frac32+\frac{4-5a}{2(3a-2)}
=
\frac{(1-a)(7a-4)}{a(3a-2)}
\geq0.
\]
Hence the integral
\[
\int_\ell^x
y\,g(y)^{\frac2a-\frac32}\phi(y)^{\frac{4-5a}{2a}}\idiff y
\]
remains bounded as $x\to L^-$. It follows that there exists some $\delta>0$ such that
\[
q(x)=O\big((L-x)^\delta\big)
\qquad\text{as }x\to L^-.
\]

We now strengthen the asymptotics of $\phi$. Since
\[
(xg)'=v=-\frac{3a-2}{a}+\frac{q}{a},
\]
and
\[
\lim_{x\to L^-}xg(x)=0,
\]
we obtain
\[
xg(x)
=
-\int_x^L v(y)\idiff y
=
\frac{3a-2}{a}(L-x)+O\big((L-x)^{1+\delta}\big).
\]
Consequently,
\[
\frac{1}{xg(x)}
=
\frac{a}{(3a-2)(L-x)}+O\big((L-x)^{-1+\delta}\big).
\]
Since $\delta>0$, the error term is integrable near $x=L$.

Now
\[
\frac{d}{dx}\log\!\left(\frac{\phi(x)}{(L-x)^{\frac{a}{3a-2}}}\right)
=
\frac{\phi_x(x)}{\phi(x)}+\frac{a}{(3a-2)(L-x)}
=
\frac{g(x)-1}{xg(x)}+\frac{a}{(3a-2)(L-x)}.
\]
Using
\[
\frac{g(x)-1}{xg(x)}=\frac{1}{x}-\frac{1}{xg(x)},
\]
we get
\[
\frac{d}{dx}\log\!\left(\frac{\phi(x)}{(L-x)^{\frac{a}{3a-2}}}\right)
=
\frac{1}{x}
-\frac{1}{xg(x)}
+\frac{a}{(3a-2)(L-x)}
=
\frac{1}{x}+O\big((L-x)^{-1+\delta}\big).
\]
The right-hand side is integrable near $x=L$. Hence the logarithmic expression
\[
\log\!\left(\frac{\phi(x)}{(L-x)^{\frac{a}{3a-2}}}\right)
\]
has a finite limit as $x\to L^-$. Since $\phi(x)>0$, there exists a constant
\[
c_\phi\in(0,+\infty)
\]
such that
\[
\phi(x)=c_\phi\,(L-x)^{\frac{a}{3a-2}}\bigl(1+o(1)\bigr).
\]

We now prove \textup{(2)}. Since
\[
m(x)=\frac12\,g(x)^{\frac2a}\phi(x)^{\frac{2(1-a)}{a}},
\]
combining \textup{(1)} with the asymptotics of $\phi$ gives
\[
m(x)
=
\frac12
\left(
\frac{3a-2}{aL}(L-x)(1+o(1))
\right)^{\frac2a}
\left(
c_\phi (L-x)^{\frac{a}{3a-2}}(1+o(1))
\right)^{\frac{2(1-a)}{a}}.
\]
Therefore
\[
m(x)=c_m\,(L-x)^{\frac2a+\frac{2(1-a)}{3a-2}}\bigl(1+o(1)\bigr),
\]
where
\[
c_m=\frac12\left(\frac{3a-2}{aL}\right)^{\frac2a}c_\phi^{\frac{2(1-a)}{a}}>0.
\]
This proves \textup{(2)}.

We now prove \textup{(3)}. By \textup{(1)} and the asymptotics of $\phi$,
\[
\sqrt{g(x)\phi(x)}
=
\sqrt{\frac{3a-2}{aL}\,c_\phi}\,
(L-x)^{\frac12+\frac{a}{2(3a-2)}}(1+o(1))
=
\sqrt{\frac{3a-2}{aL}\,c_\phi}\,
(L-x)^{\frac{2a-1}{3a-2}}(1+o(1)).
\]
Moreover, by \textup{(1)} and the asymptotics of $\phi$, we have
\[
y\,g(y)^{\frac2a-\frac32}\phi(y)^{\frac{4-5a}{2a}}
=
O\!\left((L-y)^{\frac{(1-a)(7a-4)}{a(3a-2)}}\right)
\qquad\text{as }y\to L^-,
\]
and the exponent on the right-hand side is nonnegative. Hence the integral
\[
\int_\ell^L
y\,g(y)^{\frac2a-\frac32}\phi(y)^{\frac{4-5a}{2a}}\idiff y
\]
converges, so the bracket in the representation formula for $q$ has a finite limit:
\[
\frac{q(\ell)}{\sqrt{g(\ell)\phi(\ell)}}
-
\int_\ell^x
y\,g(y)^{\frac2a-\frac32}\phi(y)^{\frac{4-5a}{2a}}\idiff y
\longrightarrow c_*
\qquad\text{as }x\to L^-,
\]
for some $c_*\in\mathbb{R}$. Therefore
\[
q(x)
=
c_*
\sqrt{\frac{3a-2}{aL}\,c_\phi}\,
(L-x)^{\frac{2a-1}{3a-2}}
+
o\!\left((L-x)^{\frac{2a-1}{3a-2}}\right).
\]
Since
\[
q(x)=a\left(v(x)+\frac{3a-2}{a}\right),
\]
we conclude that there exists a constant $c_v\in\mathbb{R}$ such that
\[
v(x)=-\frac{3a-2}{a}+c_v\,(L-x)^{\frac{2a-1}{3a-2}}
+o\!\left((L-x)^{\frac{2a-1}{3a-2}}\right).
\]
This proves \textup{(3)}.

Finally we prove \textup{(4)}. Using
\[
v_x=\frac{(av+3a-2)(v-1)-4x^2m}{2axg},
\]
together with
\[
av+3a-2
=
a\,c_v\,(L-x)^{\frac{2a-1}{3a-2}}
+
o\!\left((L-x)^{\frac{2a-1}{3a-2}}\right),
\]
\[
v(x)-1\to -\frac{2(2a-1)}{a},
\]
and
\[
2axg(x)=2(3a-2)(L-x)+o(L-x),
\]
we obtain
\[
\frac{(av+3a-2)(v-1)}{2axg}
=
-\frac{2a-1}{3a-2}\,c_v\,
(L-x)^{\frac{1-a}{3a-2}}
+o\!\left((L-x)^{\frac{1-a}{3a-2}}\right).
\]
On the other hand, by \textup{(1)} and \textup{(2)},
\[
\frac{4x^2m}{2axg}
=
O\!\left((L-x)^{\frac2a+\frac{2(1-a)}{3a-2}-1}\right).
\]
Since
\[
\frac2a+\frac{2(1-a)}{3a-2}-1
>
\frac{1-a}{3a-2},
\]
the contribution of the $m$-term is of higher order. Therefore
\[
v_x(x)=-\frac{2a-1}{3a-2}\,c_v\,
(L-x)^{\frac{1-a}{3a-2}}
+o\!\left((L-x)^{\frac{1-a}{3a-2}}\right).
\]
This proves \textup{(4)}. In particular, if $a\in(2/3,1)$, then
\[
\frac{1-a}{3a-2}>0,
\]
and hence
\[
v_x(x)\to0
\qquad\text{as }x\to L^-.
\]
The proof is complete.
\end{proof}

We now return from the auxiliary variables \((m,v,g)\) to the original profile variables \((u,\theta,\omega)\). The endpoint asymptotics obtained above allow us to reflect and periodically extend the local profile, producing a symmetric periodic solution of the profile system.
\begin{corollary}\label{corollary: periodic_profile_extension}
Let \(a\in(2/3,1]\), and let \((m,v,g)\) be the solution obtained in
Theorem~\ref{theorem: existence of fixed point} and
Corollary~\ref{corollary: global_extension}. Define, for \(x\in(-L,L)\),
\[
    u(x):=xg(x),\qquad
    \theta(x):=x^2m(x),\qquad
    \omega(x):=-\frac12 u_{xx}(x),
\]
and set
\[
    c_l=0,\qquad c_\omega=a-1,\qquad c_\theta=2(a-1).
\]
Then \((u,\theta,\omega)\) gives a symmetric \(2L\)-periodic solution of equation \eqref{eqt: case1},
where \(u\) and \(\omega\) are odd and \(\theta\) is even. Moreover, after  periodic extension the following regularity holds. If \(a=1\), then the extended solution is smooth on the whole real line, \[ u,\theta,\omega\in C^\infty(\mathbb R). \]
If \(a\in(2/3,1)\), the periodic extension satisfies \[
    \omega\in
    C^{\lceil\beta\rceil-1,\beta-\lceil\beta\rceil+1}_{\mathrm{per}}(\mathbb R)
    \cap C^\infty\bigl(\mathbb R\setminus\{(2k+1)L:k\in\mathbb Z\}\bigr),
\]
\[
    u\in
    C^{1+\lceil\beta\rceil,\beta-\lceil\beta\rceil+1}_{\mathrm{per}}(\mathbb R)
    \cap C^\infty\bigl(\mathbb R\setminus\{(2k+1)L:k\in\mathbb Z\}\bigr),
\]
\[
    \theta\in
    C^{\lceil\gamma\rceil-1,\gamma-\lceil\gamma\rceil+1}_{\mathrm{per}}(\mathbb R)
    \cap C^\infty\bigl(\mathbb R\setminus\{(2k+1)L:k\in\mathbb Z\}\bigr),
\]
\[ \beta:=\frac{1-a}{3a-2}, \qquad \gamma:=\frac2a+\frac{2(1-a)}{3a-2}.\]
\end{corollary}

\begin{proof}
On \((-L,L)\), the definitions
\[
    u=xg,\qquad \theta=x^2m,\qquad \omega=-\frac12u_{xx}
\]
together with the system \eqref{eqt:gmu} for \((m,v,g)\) imply
\[
    a u\,\theta_x=(c_\theta+2u_x)\theta,
    \qquad
    2a u\,\omega_x=c_\theta\omega+2\theta_x,
    \qquad
    -u_{xx}=2\omega .
\]
Thus \((u,\theta,\omega)\) solves \eqref{eqt: case1} on
\((-L,L)\). Proposition~\ref{proposition: regularity_solution} gives
smoothness in the interior of each period, and
Proposition~\ref{proposition: boundary_asymptotics} gives the endpoint
behavior needed for the \(2L\)-periodic extension. It remains to discuss \(a=1\). In this case we can construct an explicit smooth periodic solution via \[ u(x)=\frac1{\sqrt2}\sin(\sqrt2 x), \qquad \omega(x)=\frac1{\sqrt2}\sin(\sqrt2 x), \qquad \theta(x)=\frac14\sin^2(\sqrt2 x), \]
where $ L=\pi/\sqrt{2}$, $c_l=c_\omega=c_\theta=0$. This satisfies the same normalization conditions
\[
    u_x(0)=1,\qquad \omega_x(0)=1,\qquad \theta_{xx}(0)=1.
\] In the next subsection, we prove that this is the unique periodic
profile satisfying the above normalization conditions. This completes the
proof.
\end{proof}

\subsection{Uniqueness of the solution}

The previous subsections established the existence, extension, and regularity of the profile. We now prove that the normalized profile is unique. The only delicate point is the origin, where the ODE system is singular. We first prove uniqueness in a small neighborhood of the origin by using the prescribed quadratic behavior of \(v-1\) and \(g-1\), and then extend the uniqueness to the whole maximal interval by standard ODE uniqueness away from the origin.

\begin{proposition}
\label{proposition: uniqueness}

Let $(m_1,v_1,g_1)$ and $(m_2,v_2,g_2)$ be two solutions of
\eqref{eqt:gmu} on $(-L,L)$, where
\[
v_i:=u_{i,x}.
\]
Assume that they satisfy the same initial data
\[
2m_i(0)=1,
\qquad
g_i(0)=1,
\]
and the regularity conditions
\[
v_i(x)-1=O(x^2),
\qquad
g_i(x)-1=O(x^2)
\qquad \text{as }x\to0.
\]
Then
\[
m_1\equiv m_2,
\qquad
v_1\equiv v_2,
\qquad
g_1\equiv g_2
\qquad \text{on }(-L,L).
\]

\end{proposition}

\begin{proof}

We first prove uniqueness in a small neighborhood of the origin.
Choose $\eps>0$ sufficiently small such that for $i=1,2$,
\[
g_i(x)\geq c_0>0,
\qquad
|m_i(x)|\leq C,
\qquad
|v_i(x)-1|+|g_i(x)-1|\leq Cx^2
\qquad \text{for } |x|\leq \eps .
\]
Define the differences
\[
\delta m:=m_1-m_2,
\qquad
\delta v:=v_1-v_2,
\qquad
\delta g:=g_1-g_2 .
\]

On $(0,\eps]$, the equations \eqref{eqt:gmu} become
\[
axg_i\,m_{i,x}
=
2\bigl(v_i-1-a(g_i-1)\bigr)m_i,
\]
\[
2axg_i\,v_{i,x}
=
(av_i+3a-2)(v_i-1)-4x^2m_i,
\]
\[
(xg_i)_x=v_i.
\]
Hence the difference \(\delta g\) satisfies
\[
x(\delta g)_x+\delta g=\delta v,
\]
and therefore,
\[
\delta g(x)=\frac1x\int_0^x \delta v(y)\idiff y .
\]
Introduce
\[
E(x)
:=
\sup_{0<s\leq x}
\left(
|\delta m(s)|
+
\frac{|\delta v(s)|}{s^2}
\right),
\]
with $E(0)=0$. Using the above equations together with the bounds near $x=0$,
one obtains
\[
|\delta g(x)|
\leq
Cx^2E(x),\qquad  
|\delta m(x)|
\leq
C\int_0^x E(y)\idiff y,\qquad
\frac{|\delta v(x)|}{x^2}
\leq
C\int_0^x E(y)\idiff y .
\]
Consequently
\[
E(x)
\leq
C\int_0^x E(y)\idiff y .
\]
By Gronwall's inequality we conclude that
\[
E(x)\equiv0
\qquad \text{on }[0,\eps].
\]
Hence
\[
m_1\equiv m_2,
\qquad
v_1\equiv v_2,
\qquad
g_1\equiv g_2
\qquad \text{on }[0,\eps].
\]
The same argument applied on $[-\eps,0]$ yields coincidence on $[-\eps,\eps]$.

It remains to extend uniqueness to the whole interval. Fix any \(L_0\in(\eps,L)\). On \([\eps,L_0]\), we have \(x\geq\eps>0\) and \(g_i\) is bounded away from zero. Therefore \eqref{eqt:gmu_for_extension} is a regular first-order ODE system whose right-hand side is locally Lipschitz in \((m,v,g)\). Since the two solutions coincide at \(x=\eps\), the standard uniqueness theorem for ODEs implies that they coincide on \([\eps,L_0]\). Letting \(L_0\uparrow L\), we obtain coincidence on \([\eps,L)\). The same argument on \((-L,-\eps]\) shows that the solutions agree on the entire interval \((-L,L)\).
\end{proof}

This completes the construction and analysis of the periodic profiles with \(c_l=0\). In the next section we turn to the whole-space setting, where the Neumann condition at infinity leads to a different family of profiles and allows \(c_l\) to have either sign.

\section{Self-Similar Profiles with 
\texorpdfstring{$u_x(\infty)=0$}{u_x(infty)=0}}

We next construct whole-space solutions to the profile system \eqref{eqt:main} subject to the Neumann condition \(u_x(\infty)=0\). Unlike the periodic case, the self-similar scaling parameter \(c_l\) is not fixed to be zero; as \(a\) varies in \(0<a\leq1\), the corresponding profiles cover the regimes \(c_l>0\), \(c_l=0\), and \(c_l<0\). The construction follows the same general strategy as in the periodic case: we first obtain a local profile by a fixed-point argument near the origin, and then use the associated ODE system to extend the solution and analyze its far-field or endpoint behavior.

\subsection{Reduced and normalized formulation}
We begin by deriving the normalized profile equations in the whole-space setting. The Neumann condition at infinity imposes an additional compatibility relation among the self-similar scaling parameters, which distinguishes this case from the periodic construction.
Integrating the second equation of \eqref{eqt:main} yields
\[
2(c_l x+au)\,u_{xx} = (a u_x+a u_x(0)+c_{\theta})\,(u_x-u_x(0)) - 4\theta.
\]
Letting $x\to\infty$ and using the far–field decay of $\theta$ and $u_x$ gives
\[
a\,u_x(0)+c_{\theta}=0.
\]
By scaling and nondegeneracy, we fix $u_x(0)=1$ and $\omega_x(0)=1$, which imply
\[
c_\theta=-a,\qquad
c_l=\tfrac{1}{2}c_\theta-(a-1)u_x(0)=1-\tfrac{3}{2}a,\qquad
c_\omega=\tfrac{1}{2}c_\theta-c_l=a-1,
\]
\[
\theta_{xx}(0)=(c_l-c_\omega+a\,u_x(0))\,\omega_x(0)=2-\tfrac{3}{2}a.
\]
Consequently, the system simplifies to
\begin{equation}\label{eqt: case2}
\begin{aligned}
(c_l x+au)\,\theta_x &= (c_{\theta}+2u_x)\,\theta,\\
2(c_l x+au)\,u_{xx} &= a u_x\left(u_x-1\right) - 4\theta,
\end{aligned}
\end{equation}
where $c_l=1-\tfrac{3}{2}a$ and $c_\theta=-a$. For \(a>0\), define $g:=c_l/a+u/x$ and $m:=\theta/x^2$. Then $g$, $m$, and $u_x$ satisfy
\begin{equation}\label{eqt:gmu2}
\begin{aligned}
axg\,m_x &= \big(2xg_x+2(1-a)(g-g(0))\big)\,m,\\
2axg\,(u_x)_x &= a u_x\!\left(u_x-1\right) - 4x^2 m,\\
(xg)_x &= u_x+\tfrac{1}{a}-\tfrac{3}{2},
\end{aligned}
\end{equation}
with $2m(0)=2-\tfrac{3}{2}a$ and $g(0)=\tfrac{1}{a}-\tfrac{1}{2}>0$.

Although the main fixed-point construction below is carried out for \(0<a\leq1\), it is useful to record the limiting case \(a=0\), where the normalized system degenerates to an explicitly solvable two-dimensional system. In the limiting case \(a=0\), the above definition of \(g\) is singular and
should be avoided. Instead, from \(c_l=1\) and \(c_\theta=0\),
the \(a=0\) formulation is the closed two-dimensional system
\begin{equation}\label{eqt:gmu2_a0}
\begin{aligned}
x m_x &= 2(u_x-1)m,\\
(u_x)_x &= -2xm,
\end{aligned}
\end{equation}
with $m(0)=1,\,u_x(0)=1.$
This system can be solved explicitly. One checks directly that
\[
    m(x)=\frac{1}{\left(1+x^2/2\right)^2},
    \qquad
    u_x(x)=\frac{1-x^2/2}{1+x^2/2}.
\]
satisfy \eqref{eqt:gmu2_a0} and the initial conditions \(m(0)=1\),
\(u_x(0)=1\). 
Thus, with the normalization \(u(0)=0\),
\[
    u(x)
    =
    -x+2\sqrt{2}\arctan\left(\frac{x}{\sqrt{2}}\right).
\]

\subsection{Fixed-point argument} 
We now construct a local solution near the origin via fixed-point argument. 
The equation \eqref{eqt:gmu2} is equivalent to
\begin{equation}\label{eqt:gmu2_for_fixedpoint}
\begin{aligned}
axg\,m_x &= \big(2xg_x+2(1-a)(g-g(0))\big)\,m,\\
2axg\,(u_x-1)_x &= a\Big((xg)_x-\tfrac{1}{a}+\tfrac{3}{2}\Big)(u_x-1)-4x^2 m,\\
(xg)_x &= u_x+\tfrac{1}{a}-\tfrac{3}{2},
\end{aligned}
\end{equation}
with $2m(0)=2-\tfrac{3}{2}a$ and $g(0)=\tfrac{1}{a}-\tfrac{1}{2}$.
From the first line of \eqref{eqt:gmu2_for_fixedpoint} we obtain
\begin{equation}\label{eq:m_representation2}
m(x)
=
\frac{4-3a}{4}
\left(\frac{g(x)}{g(0)}\right)^{\frac{2}{a}}
\exp\!\left(
\frac{2(1-a)}{a}
\int_0^x
\frac{g(y)-g(0)}{y\,g(y)}
\,\mathrm{d}y
\right).
\end{equation}
From the second line of \eqref{eqt:gmu2_for_fixedpoint}, viewing it as a linear equation for $u_x-1$ with $g,m$ given,
we obtain
\[
\begin{aligned}
    u'(x)
&= 1 - \frac{4-3a}{2-a}\,x^{\frac{a}{2-a}}\,\bigg(\frac{g(x)}{g(0)}\bigg)^{\frac{1}{2}}
\exp\!\left(\frac{2-3a}{4-2a}\int_0^x\frac{g(y)-g(0)}{y\,g(y)}\,\mathrm{d}y\right) \\
&\quad \cdot \int_0^x y^{\frac{2-2a}{2-a}}\,\bigg(\frac{g(y)}{g(0)}\bigg)^{\frac{4-3a}{2a}} 
\exp\!\left(\frac{7a^2-14a+8}{2a(2-a)}\int_0^y\frac{g(z)-g(0)}{z\,g(z)}\,\mathrm{d}z\right)\mathrm{d}y .
\end{aligned}
\]
Finally, integrating the third line yields
\[
g(x)=\frac{1}{x}\int_0^x u'(y)\,\mathrm{d}y+\Big(\frac{1}{a}-\frac{3}{2}\Big).
\]
 Denote
\[
\phi(x):=\exp\!\left(\int_0^x\frac{g(y)-g(0)}{y\,g(y)}\,\mathrm{d}y\right).
\]
Then $g$, $\phi$, and $u'$ satisfy
\begin{equation}\label{eqt:Iter2}
\begin{aligned}
\phi(x)
&= \exp\!\left(\int_0^x\frac{g(y)-g(0)}{y\,g(y)}\,\mathrm{d}y\right),\\
u'(x)
&= 1 - \frac{4-3a}{2-a}\,x^{\frac{a}{2-a}}
\left(\frac{g(x)}{g(0)}\right)^{\frac12}
\phi(x)^{\frac{2-3a}{4-2a}}
\int_0^x
y^{\frac{2-2a}{2-a}}
\left(\frac{g(y)}{g(0)}\right)^{\frac{4-3a}{2a}}
\phi(y)^{\frac{7a^2-14a+8}{2a(2-a)}}
\,\mathrm{d}y,\\
\frac{g(x)}{g(0)}
&=
\frac{2a}{2-a}\,
\frac{1}{x}\int_0^x u'(y)\idiff y
+
\frac{2-3a}{2-a}=:\widetilde{g}(x).
\end{aligned}
\end{equation}
This representation leads naturally to a fixed–point approach for constructing self–similar profiles. As before, the singularity at \(x=0\) is handled by rewriting the system in integral form and using \(v=u_x\) as the fixed-point variable. Once \(v\) is obtained, the functions \(g\), \(\phi\), and \(m\) are recovered from the integral above.

Let
\[
\mathbb{V}_2 := \big\{v\in C([-d,d]) : v(x)=v(-x)\big\}.
\]
Consider the closed, convex subset
\begin{equation*}
\begin{split}
\mathbb{D}_2 := \Big\{v\in \mathbb{V}_2:\ \quad &0\leq1-\mu x^2\leq v(x)\leq1\ \text{ on } [-d,d] ,\\\quad
 &v \text{ is nonincreasing on }[0,d] ,\quad
 |v(x)-v(y)|\leq K|x-y| \text{ for all } x,y\in[-d,d]
 \Big\},
\end{split}
\end{equation*}
where $d$, $\mu$ and $K$ are chosen as follows.
If $0<a\leq \frac23$, we set
\[
\mu:=1,
\qquad
d:=1.
\]
If $2/3<a\leq 1$, we set
\[
\mu:=\left(\frac{6-5a}{3(2-a)}\right)^{-\frac{3a-2}{4(2-a)}},
\qquad
d:=\mu^{-1/2}.
\]
In both cases, we have
\[
\mu d^2=1.
\]
Finally, we define
\[
K:=
\left(
\frac{3a\mu}{6-5a}
+
\frac{4-3a}{2-a}
\right)d.
\]
As in the periodic case, the function space for the fixed-point argument is chosen to encode the expected sign, monotonicity, and Lipschitz bounds of \(v=u_x\) near the origin. The parameters are selected differently on the two sides of \(a=2/3\), reflecting the change in the sign of \(c_l\).

Define the operators
\[
\frG(v)(x):=\frac{2a}{2-a}\,
\frac{1}{x}\int_0^x v(y)\idiff y
+
\frac{2-3a}{2-a}.
\]
\[
\Psi(v)(x):=\exp\!\left(\int_0^x\frac{\frG(v)(y)-1}{y\,\frG(v)(y)}\,\mathrm{d}y\right),
\]
\[
\frR_a(v)(x):=1 - \frac{4-3a}{2-a}\,x^{\frac{a}{2-a}}
\frG(v)(x)^{\frac12}
\Psi(v)(x)^{\frac{2-3a}{4-2a}}
\int_0^x
y^{\frac{2-2a}{2-a}}
\frG(v)(y)^{\frac{4-3a}{2a}}
\Psi(v)(y)^{\frac{7a^2-14a+8}{2a(2-a)}}
\,\mathrm{d}y.
\]


The rest of this subsection is devoted to proving that \(\frR_a\) has a fixed point in \(\bbD_2\) for every \(0<a\leq1\). The estimates below show that \(\frR_a\) preserves the admissible set, is continuous in the \(L^\infty\)-topology, and acts on a compact convex set.

\begin{lemma}\label{lemma: g_properties2}
For any $v\in \bbD_2$,
\[
1-\frac{2a}{3(2-a)}\,\mu x^2
\leq
\frG(v)(x)
\leq
1
\qquad \text{for all } x\in[0,d].
\]
Moreover, $\frG(v)$ is nonincreasing on $[0,d]$.
\end{lemma}

\begin{proof}
Let $\widetilde g:=\frG(v)$. By definition,
\[
\widetilde g(x)
=
\frac{2a}{2-a}\,
\frac{1}{x}\int_0^x v(y)\idiff y
+
\frac{2-3a}{2-a}.
\]

Since $v\in\bbD_2$, we have
\[
1-\mu y^2 \leq v(y) \leq 1
\qquad \text{for } y\in[0,d].
\]
Integrating over $[0,x]$ gives
\[
x-\frac{\mu x^3}{3}
\leq
\int_0^x v(y)\idiff y
\leq
x.
\]
Dividing by $x$ yields
\[
1-\frac{\mu x^2}{3}
\leq
\frac{1}{x}\int_0^x v(y)\idiff y
\leq
1.
\]

Substituting into the definition of $\widetilde g$ gives
\[
\widetilde g(x)
\leq
\frac{2a}{2-a}
+
\frac{2-3a}{2-a}
=
1,
\]
and
\[
\widetilde g(x)
\geq
\frac{2a}{2-a}\left(1-\frac{\mu x^2}{3}\right)
+
\frac{2-3a}{2-a}
=
1-\frac{2a}{3(2-a)}\,\mu x^2.
\]

Next we show that $\widetilde g$ is nonincreasing. Differentiating gives
\[
\widetilde g'(x)
=
\frac{2a}{2-a}
\left(\frac{1}{x}\int_0^x v(y)\idiff y\right)'.
\]
Using the identity
\[
\left(\frac{1}{x}\int_0^x v(y)\idiff y\right)'
=
\frac{1}{x^2}\int_0^x \big(v(x)-v(y)\big)\idiff y,
\]
and the fact that $v$ is nonincreasing on $[0,d]$, we obtain
\[
\widetilde g'(x)
=
\frac{2a}{2-a}\cdot
\frac{1}{x^2}
\int_0^x (v(x)-v(y))\idiff y
\leq 0.
\]
Therefore $\widetilde g$ is nonincreasing on $[0,d]$.
\end{proof}

\begin{lemma}\label{lemma: phi_properties2}
For any $v\in \bbD_2$,
\[
\sqrt{1-\frac{2a}{3(2-a)}\,\mu x^2}
\leq
\Psi(v)(x)
\leq
1
\qquad \text{for all } x\in[0,d].
\]
Moreover, $\Psi(v)$ is nonincreasing on $[0,d]$.
\end{lemma}

\begin{proof}
Let $\widetilde g:=\frG(v)$ and $\phi:=\Psi(v)$. By Lemma~\ref{lemma: g_properties2},
\[
1-\frac{2a}{3(2-a)}\,\mu x^2
\leq
\widetilde g(x)
\leq
1.
\]

From the definition
\[
\phi(x)
=
\exp\!\left(
\int_0^x
\frac{\widetilde g(y)-1}{y\,\widetilde g(y)}\idiff y
\right),
\]
we immediately obtain $\phi(x)\leq 1$ since $\widetilde g(y)\leq 1$.

Next we estimate the exponent from below. Since
\[
\widetilde g(y)\geq 1-\frac{2a}{3(2-a)}\,\mu y^2,
\]
we have
\[
\frac{\widetilde g(y)-1}{y\,\widetilde g(y)}
\geq
-\frac{\frac{2a}{3(2-a)}\,\mu y^2}{y\left(1-\frac{2a}{3(2-a)}\,\mu y^2\right)}
=
-\frac{\frac{2a}{3(2-a)}\,\mu\,y}{1-\frac{2a}{3(2-a)}\,\mu y^2}.
\]
Integrating yields
\[
\int_0^x
\frac{\widetilde g(y)-1}{y\,\widetilde g(y)}\idiff y
\geq
\frac12
\log\!\left(
1-\frac{2a}{3(2-a)}\,\mu x^2
\right).
\]
Therefore
\[
\phi(x)
\geq
\sqrt{
1-\frac{2a}{3(2-a)}\,\mu x^2
}.
\]

Finally,
\[
\phi'(x)
=
\phi(x)\frac{\widetilde g(x)-1}{x\,\widetilde g(x)}
\leq
0,
\]
since $\widetilde g(x)\leq 1$. Hence $\phi$ is nonincreasing on $[0,d]$.
\end{proof}

\begin{lemma}\label{lemma: r_properties2}
Let $a\in(0,1]$ and $v\in\bbD_2$. Then for all $x\in[0,d]$,
\[
1-\mu x^2
\leq
\frR_a(v)(x)
\leq
1-x^2\left(1-\frac{2a}{3(2-a)}\,\mu x^2\right)^{\frac{3-2a}{a}}
\leq
1 .
\]
Moreover, $\frR_a(v)$ is nonincreasing on $[0,d]$.
\end{lemma}

\begin{proof}
Denote
\[
\widetilde g:=\frG(v),\qquad \phi:=\Psi(v),\qquad r:=\frR_a(v).
\]
By Lemmas~\ref{lemma: g_properties2} and \ref{lemma: phi_properties2},
\[
1-\frac{2a}{3(2-a)}\,\mu x^2
\leq \widetilde g(x)\leq 1,
\qquad
\sqrt{1-\frac{2a}{3(2-a)}\,\mu x^2}
\leq \phi(x)\leq 1,
\]
and both $\widetilde g$ and $\phi$ are nonincreasing on $[0,d]$.

\medskip
\noindent
\textit{Upper bound.}
Since $\widetilde g$ and $\phi$ are nonincreasing, for $0\leq y\leq x$ we have
\[
\widetilde g(y)\geq \widetilde g(x),\qquad \phi(y)\geq\phi(x).
\]
Moreover, for $a\in(0,1]$,
\[
\frac{4-3a}{2a}\geq0,
\qquad
\frac{7a^2-14a+8}{2a(2-a)}\geq0 .
\]
Hence
\[
\widetilde g(y)^{\frac{4-3a}{2a}}
\phi(y)^{\frac{7a^2-14a+8}{2a(2-a)}}
\geq
\widetilde g(x)^{\frac{4-3a}{2a}}
\phi(x)^{\frac{7a^2-14a+8}{2a(2-a)}} .
\]
Substituting into the definition of $r$ yields
\[
\begin{aligned}
r(x)
&\leq
1-\frac{4-3a}{2-a}
x^{\frac{a}{2-a}}
\widetilde g(x)^{\frac12}
\phi(x)^{\frac{2-3a}{4-2a}}
\widetilde g(x)^{\frac{4-3a}{2a}}
\phi(x)^{\frac{7a^2-14a+8}{2a(2-a)}}
\int_0^x y^{\frac{2-2a}{2-a}}\idiff y .
\end{aligned}
\]
Since
\[
\int_0^x y^{\frac{2-2a}{2-a}}\idiff y
=
\frac{2-a}{4-3a}x^{\frac{4-3a}{2-a}},
\]
we obtain
\[
r(x)
\leq
1-x^2 \widetilde g(x)^{\frac{2-a}{a}}\phi(x)^{\frac{2(1-a)}{a}} .
\]
Using the lower bounds for $\widetilde g$ and $\phi$ gives
\[
r(x)
\leq
1-x^2
\left(1-\frac{2a}{3(2-a)}\,\mu x^2\right)^{\frac{3-2a}{a}}
\leq 1 .
\]

\medskip
\noindent
\textit{Lower bound.}
Using $\widetilde g(x)\leq1$ and $\phi(x)\leq1$, we obtain
\[
\begin{aligned}
r(x)
&\geq
1-\frac{4-3a}{2-a}
x^{\frac{a}{2-a}}
\phi(d)^{\frac{2-3a}{4-2a}}
\int_0^x y^{\frac{2-2a}{2-a}}\idiff y  
=
1-x^2\phi(d)^{\frac{2-3a}{4-2a}} .
\end{aligned}
\]
By Lemma~\ref{lemma: phi_properties2}, we have
\[
\phi(d)\geq
\sqrt{1-\frac{2a}{3(2-a)}\,\mu d^2}.
\]
If $0<a\leq\frac23$, we choose $\mu=1$ and $d=1$, and obtain
\[
r(x)\geq 1-x^2 = 1-\mu x^2 .
\]
If $\frac23<a\leq1$, the parameters $\mu$ and $d$ are chosen so that
\[
\left(1-\frac{2a}{3(2-a)}\,\mu d^2\right)^{\frac{2-3a}{8-4a}}
\leq \mu,
\]
which yields
\[
r(x)\geq 1-\mu x^2 .
\]

\medskip
\noindent
\textit{Monotonicity.}
Differentiating the equation satisfied by $r$ gives
\[
r'(x)
=
\frac{v(x)}{2xg(x)}(r(x)-1)
-\frac{4-3a}{2-a}x \widetilde g(x)^{\frac{2-a}{a}}
\phi(x)^{\frac{2(1-a)}{a}} .
\]
Since $v(x)\geq0$, $\widetilde g(x)>0$, and $r(x)\leq1$, we conclude that
\[
r'(x)\leq0 .
\]
Hence $r$ is nonincreasing on $[0,d]$.
\end{proof}

\begin{lemma}\label{lemma: r_Lipschitz2}
Let $a\in(0,1]$ and $v\in\bbD_2$. Then $\frR_a(v)$ is Lipschitz on $[-d,d]$. In particular,
\[
|\frR_a(v)(x)-\frR_a(v)(y)|
\leq
K|x-y|
\qquad \text{for all } x,y\in[-d,d].
\]
\end{lemma}

\begin{proof}
Denote
\[
\widetilde g:=\frG(v),\qquad
\phi:=\Psi(v),\qquad
r:=\frR_a(v).
\]
From Lemma~\ref{lemma: r_properties2},
\[
r'(x)
=
\frac{v(x)}{2xg(x)}(r(x)-1)
-\frac{4-3a}{2-a}x \widetilde g(x)^{\frac{2-a}{a}}
\phi(x)^{\frac{2(1-a)}{a}} .
\]
Since $r(x)\leq1$, we obtain
\[
|r'(x)|
=
\frac{v(x)}{2xg(x)}(1-r(x))
+
\frac{4-3a}{2-a}x \widetilde g(x)^{\frac{2-a}{a}}
\phi(x)^{\frac{2(1-a)}{a}} .
\]
Using
\[
v(x)\leq1,\qquad
1-r(x)\leq \mu x^2,\qquad
\widetilde g(x)\geq \widetilde g(d),
\]
and
\[
\widetilde g(x)\leq1,\qquad
\phi(x)\leq1,
\]
we obtain
\[
|r'(x)|
\leq
\frac{\mu}{2g(d)}x
+
\frac{4-3a}{2-a}x .
\]
Since $x\leq d$, this yields
\[
|r'(x)|
\leq
\left(\frac{\mu}{2g(d)}+\frac{4-3a}{2-a}\right)d \leq K.
\]
Hence $\frR_a(v)$ is Lipschitz on $[0,d]$.  
Since $\frR_a(v)$ is even, the same bound holds on $[-d,d]$.
\end{proof}

\begin{proposition}\label{proposition: r_selfmap2}
Let $a\in(0,1]$. The operator $\frR_a$ maps $\bbD_2$ into itself.
\end{proposition}

\begin{proof}
Let $v\in\bbD_2$ and denote
\[
r:=\frR_a(v).
\]
By Lemma~\ref{lemma: r_properties2}, we have
\[
1-\mu x^2 \leq r(x)\leq 1
\qquad \text{for all } x\in[0,d],
\]
and $r$ is nonincreasing on $[0,d]$. Since $\frR_a(v)$ is even, the same bounds hold on $[-d,d]$.
Moreover, by Lemma~\ref{lemma: r_Lipschitz2}, $\frR_a(v)$ is Lipschitz on $[-d,d]$.
Therefore $r$ satisfies all the defining properties of $\bbD_2$, and hence
\[
\frR_a(v)\in\bbD_2 .
\]
\end{proof}

\begin{proposition}\label{proposition: r_continuity2}
Let $a\in(0,1]$. $\frR_a:\bbD_2\to\bbD_2$ is continuous with respect to the $L^\infty$ norm.
\end{proposition}

\begin{proof}
Fix $v_0\in\bbD_2$ and let $v\in\bbD_2$ with
\[
\|v-v_0\|_{L^\infty([-d,d])}\leq\delta .
\]
Denote
\[
\widetilde {g_0}=\frG(v_0),\qquad \phi_0=\Psi(v_0),\qquad r_0=\frR_a(v_0),
\]
and define $\widetilde g,\phi,r$ similarly for $v$.

\medskip
\noindent
\textit{Control of $\widetilde g$ and $\phi$ near $v_0$.}
From the definition of $\frG$,
\[
\|\widetilde g-\widetilde {g_0}\|_{L^\infty([-d,d])}
\leq
\frac{2a}{2-a}\|v-v_0\|_{L^\infty([-d,d])}
\leq C\delta .
\]
By Lemma~\ref{lemma: g_properties2},
\[
1-\frac{2a}{3(2-a)}\,\mu d^2
\leq \widetilde {g_0}(x)\leq1
\qquad\text{on }[0,d].
\]
Hence for $\delta$ sufficiently small,
\[
\widetilde g(x),\widetilde {g_0}(x)\in\Big[\frac12,1\Big]
\quad\text{on }[0,d].
\]
Using the identity
\[
\frac{\widetilde g(y)-1}{y\,\widetilde g(y)}-\frac{\widetilde {g_0}(y)-1}{y\,\widetilde {g_0}(y)}
=\frac{\widetilde g(y)-\widetilde {g_0}(y)}{y\,\widetilde g(y)\,\widetilde {g_0}(y)},
\]
and the bounds $\frac12\leq \widetilde g,\widetilde {g_0}\leq1$, we obtain
\[
\left|\frac{\widetilde g(y)-1}{y\,\widetilde g(y)}-\frac{\widetilde {g_0}(y)-1}{y\,\widetilde {g_0}(y)}\right|
\leq \frac{4}{y}\,|\widetilde g(y)-\widetilde {g_0}(y)|.
\]
On the other hand, by Lemma~\ref{lemma: g_properties2},
\[
|\widetilde g(y)-1|,\,|\widetilde {g_0}(y)-1|
\leq
C y^2 .
\]
Hence
\[
|\widetilde g(y)-\widetilde {g_0}(y)|
\leq
|\widetilde g(y)-1|+|\widetilde {g_0}(y)-1|
\leq
C y^2 .
\]
Let $\delta=\|\widetilde g-\widetilde {g_0}\|_{L^\infty([-d,d])}$.  
For any $x\in[0,d]$, split the integral at $y=\sqrt{\delta}$.
For the first part,
\[
\left|
\int_0^{\min(x,\sqrt{\delta})}
\left(
\frac{\widetilde g-1}{y\widetilde g}-\frac{\widetilde {g_0}-1}{y\widetilde {g_0}}
\right)\idiff y
\right|
\leq
C\int_0^{\sqrt{\delta}} y\idiff y
\leq
C\delta .
\]
For the second part,
\[
\left|
\int_{\sqrt{\delta}}^x
\left(
\frac{\widetilde g-1}{y\widetilde g}-\frac{\widetilde {g_0}-1}{y\widetilde {g_0}}
\right)\idiff y
\right|
\leq
C\delta\int_{\sqrt{\delta}}^d\frac{\diff y}{y}
\leq
C\delta|\log\delta|.
\]
Combining the two bounds gives
\[
\sup_{x\in[0,d]}
\left|
\int_0^x
\left(
\frac{\widetilde g(y)-1}{y\,\widetilde g(y)}-\frac{\widetilde {g_0}(y)-1}{y\,\widetilde {g_0}(y)}
\right)\idiff y
\right|
\leq
C(\delta+\delta|\log\delta|)
\lesssim
\sqrt{\delta},
\]
for $\delta$ sufficiently small.
Therefore,
\[
\|\phi-\phi_0\|_{L^\infty([-d,d])}
\lesssim
\sqrt{\delta}.
\]

\medskip
\noindent
\textit{Continuity of $r=\frR_a(v)$.}
Write
\[
r=1-\frac{4-3a}{2-a}\,A\,B,
\]
where
\[
A(x)=x^{\frac{a}{2-a}}
\widetilde g(x)^{\frac12}
\phi(x)^{\frac{2-3a}{4-2a}},
\qquad
B(x)=\int_0^x
y^{\frac{2-2a}{2-a}}
\widetilde g(y)^{\frac{4-3a}{2a}}
\phi(y)^{\frac{7a^2-14a+8}{2a(2-a)}}
\idiff y,
\]
and define $A_0,B_0$ analogously.
Since $\widetilde g,\widetilde {g_0},\phi,\phi_0$ remain in fixed compact intervals by Lemmas~\ref{lemma: g_properties2}--\ref{lemma: phi_properties2}, the power maps are Lipschitz there. Hence
\[
\|A-A_0\|_{L^\infty}
\lesssim
\|\widetilde g-\widetilde {g_0}\|_{L^\infty}
+
\|\phi-\phi_0\|_{L^\infty}
\lesssim
\delta+\sqrt{\delta}.
\]
Similarly,
\[
\|B-B_0\|_{L^\infty}
\lesssim
\delta+\sqrt{\delta}.
\]
Finally,
\[
\|r-r_0\|_{L^\infty}
\leq
C\Big(
\|A-A_0\|_{L^\infty}\|B\|_{L^\infty}
+
\|A_0\|_{L^\infty}\|B-B_0\|_{L^\infty}
\Big)
\lesssim
\delta+\sqrt{\delta}.
\]
Therefore
\[
\|\frR_a(v)-\frR_a(v_0)\|_{L^\infty([-d,d])}\to0
\qquad\text{as }\delta\to0.
\]
Hence $\frR_a$ is continuous on $\bbD_2$ with respect to the $L^\infty$ norm.
\end{proof}

\begin{lemma}\label{lemma: D_compactness2}
    Let $a\in(0,1]$. The set $\bbD_2$ is compact with respect to the $L^\infty$-norm.
\end{lemma}
\begin{proof}
Let $\{v_n\}\subset \bbD_2$. Since every $v_n$ is $K$-Lipschitz on $[-d,d]$, the family $\bbD_2$ is equicontinuous. 
Moreover, $\bbD_2$ is uniformly bounded in $L^\infty([-d,d])$. By the Arzel\`a--Ascoli theorem, there exists a subsequence $v_{n_k}$ and a function $v\in C([-d,d])$ such that
\[
\|v_{n_k}-v\|_{L^\infty([-d,d])}\to0 .
\]Since $\bbD_2$ is closed in $L^\infty$, we have $v\in\bbD_2$. Thus every sequence in $\bbD_2$ has a convergent subsequence in $L^\infty$, so $\bbD_2$ is compact in the $L^\infty$ norm.
\end{proof}
\begin{theorem}\label{theorem: existence of fixed point2}
Let $a\in(0,1]$. The map $\frR_a$ has a fixed point $v\in\bbD_2$, i.e. $\frR_a(v)=v.$
\end{theorem}
\begin{proof}
By Proposition~\ref{proposition: r_selfmap2}, the map $\frR_a$ maps $\bbD_2$ into itself.  
By Proposition~\ref{proposition: r_continuity2}, the map $\frR_a$ is continuous with respect to the $L^\infty$ norm.  
By Lemma~\ref{lemma: D_compactness2}, the set $\bbD_2$ is compact in $L^\infty([-d,d])$.
Moreover, $\bbD_2$ is convex and closed in $L^\infty([-d,d])$.
Therefore, by the Schauder fixed-point theorem, $\frR_a$ admits a fixed point in $\bbD_2$.
\end{proof}

The fixed point gives a local normalized profile on \([-d,d]\). To obtain a whole-space profile, or a compactly supported profile in the expanding regime, we must continue this local solution beyond \(d\). Away from the origin, the system is a regular first-order ODE as long as \(g>0\), so the main question is whether \(g\) remains positive for all \(x\), or reaches zero at a finite endpoint.

\subsection{Extension of the solution}

We have obtained a solution of \eqref{eqt:gmu2} on \([-d,d]\). We now extend it to the region \(|x|>d\). Since \(x\) is bounded away from zero in this region, the profile equations can be written as the ODE system below as long as \(g>0\). The behavior of this ODE system depends strongly on the sign of \(c_l=1-\frac32a\), and we therefore treat the cases \(a>2/3\) and \(a\leq2/3\) separately.

Writing $v:=u_x$, the system \eqref{eqt:gmu2} is equivalent, whenever $xg(x)\neq0$, to the first-order ODE system
\begin{equation}\label{eqt:gmu2_for_extension}
\begin{aligned}
m_x
&=
\frac{2v-a-2ag}{axg}m,\\
v_x
&=
\frac{a v(v-1)-4x^2m}{2axg},\\
g_x
&=
\frac{v+\frac1a-\frac32-g}{x}.
\end{aligned}
\end{equation}
The initial data at $x=d$ are inherited from the fixed-point solution constructed on $[-d,d]$.

\begin{proposition}\label{proposition: local_existence2}
Let $x_0>0$ and let $m_0,v_0,g_0$ satisfy $g_0>0$. Then there exists $\eps>0$ such that the system \eqref{eqt:gmu2_for_extension} admits a unique solution $(m,v,g)\in C^1([x_0,x_0+\eps])$ satisfying
\[
m(x_0)=m_0,\qquad v(x_0)=v_0,\qquad g(x_0)=g_0.
\]
\end{proposition}

\begin{proof}
The proof is similar to that of Proposition~\ref{proposition: local_existence}, with \eqref{eqt:gmu_for_extension} replaced by \eqref{eqt:gmu2_for_extension}, and is therefore omitted here.
\end{proof}

The local continuation criterion shows that the solution can be extended as long as \(g\) stays positive and the solution remains bounded. We first consider the case \(a>2/3\), where \(c_l<0\). In this regime the solution reaches a finite endpoint, and the resulting profile will be compactly supported after extension.

\begin{proposition}\label{proposition: global_extension_supercritical}
Let $\frac23<a<1$. Then the solution constructed on $[-d,d]$ extends uniquely to a maximal interval $(-L,L)$, where $d<L<+\infty$, such that
\[
(m,v,g)\in C^1((-L,L)),
\qquad
g(x)>0 \quad \text{for } x\in(-L,L).
\]
Moreover,
\[
\lim_{x\to L^-}g(x)=0,\qquad
\lim_{x\to L^-}m(x)=0,\qquad
\lim_{x\to L^-}v(x)=0 .
\]
In particular, the profile reaches a finite endpoint at which $g(L)=0$.
\end{proposition}

\begin{proof}
    By Proposition~\ref{proposition: local_existence2}, the solution constructed on \([-d,d]\) extends uniquely to a maximal interval \((-L,L)\), with \(d<L\leq+\infty\), such that \((m,v,g)\in C^1((-L,L))\) and \(g>0\). By symmetry, it suffices to study the solution on \((d,L)\). We will show below that the same argument as in the previous section can be adapted to prove both the finiteness of \(L\) and the limiting behavior at the endpoint. The proof has three steps. We first obtain monotonicity and a representation formula for \(v\). We then rule out an infinite interval of existence. Finally, we show that the only possible finite endpoint is characterized by \(g,m,v\to0\).

    The same argument as in Lemma~\ref{lemma: monotonicity_extension}, applied to
\eqref{eqt:gmu2_for_extension}, implies the following properties on \((d,L)\):
\(m(x)>0\), \(v(x)<1\), \(g_x(x)<0\), and \((m/g^2)_x(x)\leq0\).
In particular, \(g\) and \(m/g^2\) are monotonically decreasing on \((d,L)\).
Hence the limit \[g_*:=\lim_{x\to L^-}g(x)\] exists with \(g_*\geq0\).

We derive a representation formula for \(v\), which will be used in the two contradiction arguments below. From \eqref{eqt:gmu2_for_extension}
we have
\[
    v_x
    =
    \frac{v(v-1)}{2xg}
    -
    \frac{2xm}{ag}
    =
    \left(
        \frac{g_x}{2g}
        +
        \frac{g-g(0)}{2xg}
    \right)v
    -
    \frac{2xm}{ag}.
\]
This is a linear equation for \(v\). Define
\[
    \phi(x)
    :=
    \exp\left(
        \int_0^x
        \frac{g(y)-g(0)}{y g(y)}\,\mathrm dy
    \right).
\]
Then
\[
    \left(\frac{v}{\sqrt{g\phi}}\right)_x
    =
    -\frac{2xm}{a g^{3/2}\phi^{1/2}}.
\]
Integrating from \(d\) to \(x\), for \(d<x<L\), gives
\begin{equation}\label{eq:v_representation2}
v(x)
=
\sqrt{g(x)\phi(x)}
\left(
\frac{v(d)}{\sqrt{g(d)\phi(d)}}
-
\frac{2}{a}\int_d^x
\frac{y\,m(y)}{g(y)^{3/2}\phi(y)^{1/2}}
\,\mathrm dy
\right).
\end{equation}

We first rule out \(L=+\infty\). Suppose, for contradiction, that
\(L=+\infty\). Since \(g_x<0\), for \(x\geq d\),
\[
    \frac{g(x)-g(0)}{g(x)}
    =
    1-\frac{g(0)}{g(x)}
    \leq
    1-\frac{g(0)}{g(d)}
    <0,
    \qquad x\geq d.
\]
It follows from the definition of \(\phi\) that
\[
\begin{aligned}\label{eq:global_extension_supercritical_decay_phi}
    \phi(x)
    &=
    \phi(d)
    \exp\left(
        \int_d^x
        \frac{g(y)-g(0)}{y g(y)}\,\mathrm dy
    \right)  \\
    &\leq
    \phi(d)
    \left(\frac{x}{d}\right)^{1-\frac{g(0)}{g(d)}} ,
\end{aligned}
\]
and hence
\[
    \lim_{x\to+\infty}\phi(x)=0.
\]
Using \eqref{eq:v_representation2} and \(m>0\), we get
\[
    v(x)
    \leq
    \sqrt{g(x)\phi(x)}
    \frac{v(d)}{\sqrt{g(d)\phi(d)}} .
\]
Since \(0<g(x)\leq g(d)\) and \(\phi(x)\to0\), this implies
\[
    \limsup_{x\to+\infty}v(x)\leq0.
\]
On the other hand, integrating
\[
    (xg)_x=v+\frac1a-\frac32
\]
from \(0\) to \(x\) gives
\begin{equation}\label{eq:xg_integral_identity2}
    xg(x)
    =
    \int_0^x v(y)\,\mathrm dy
    -
    \left(\frac32-\frac1a\right)x
    =
    \int_0^x v(y)\,\mathrm dy
    +
    \left(\frac1a-\frac32\right)x .
\end{equation}
Since \(a>\frac23\), we have \(\frac32-\frac1a>0\). By
\(\limsup_{x\to+\infty}v(x)\leq0\), for \(X>0\) sufficiently large,
\[
    v(x)\leq \frac12\left(\frac32-\frac1a\right),
    \qquad x\geq X.
\]
Thus, for \(x\geq X\),
\[
\begin{aligned}
    xg(x)
    &=
    \int_0^X v(y)\,\mathrm dy
    +
    \int_X^x v(y)\,\mathrm dy
    -
    \left(\frac32-\frac1a\right)x  \\
    &\leq
    \int_0^X v(y)\,\mathrm dy
    +
    \frac12\left(\frac32-\frac1a\right)(x-X)
    -
    \left(\frac32-\frac1a\right)x .
\end{aligned}
\]
Dividing by \(x\) and letting \(x\to+\infty\), we obtain
\[
    \limsup_{x\to+\infty}g(x)
    \leq
    -\frac12\left(\frac32-\frac1a\right)<0,
\]
which contradicts \(g(x)>0\). Hence $L<+\infty$.

We now prove \(g_*=0\). Suppose, for contradiction, that \(g_*>0\). Since
\(L<+\infty\), for \(x\) sufficiently close to \(L\), \(g(x)\) is bounded
above and bounded away from zero. The bound
\[
    0<m(x)\leq \frac{m(d)}{g(d)^2}g(x)^2
\]
then implies that \(m\) remains bounded. Moreover, by the definition of
\(\phi\), both \(\phi\) and \(1/\phi\) are bounded on \([d,L)\). Therefore the
integrand in \eqref{eq:v_representation2} is bounded on \([d,L)\), and the
integral remains finite as \(x\to L^-\). Hence \(v\) remains bounded on
\([d,L)\). Consequently \(m\), \(v\), and \(g\) remain bounded, while \(g\)
stays bounded away from zero. The ODE system is therefore nondegenerate near
\(L\), so the local existence theorem extends the solution past \(L\),
contradicting the maximality of \(L\). Thus   $g_*=0.$

As observed above, this immediately gives
\[
    \lim_{x\to L^-}m(x)=0.
\]
It remains to prove the endpoint limit of \(v\). Since \(L<+\infty\) and
\(g(x)\to0\), while \(\phi\) is positive and decreasing on \((d,L)\), we have
\[
    \sqrt{g(x)\phi(x)}\to0.
\]
We claim that the integral term in \eqref{eq:v_representation2}, after
multiplication by \(\sqrt{g(x)\phi(x)}\), also tends to zero. Indeed, using
\(m\leq Cg^2\), we have
\[
    \frac{y\,m(y)}{g(y)^{3/2}\phi(y)^{1/2}}
    \leq
    C\frac{g(y)^{1/2}}{\phi(y)^{1/2}},
    \qquad d<y<L.
\]
Since \(\phi\) is decreasing, for \(d<y<x<L\),
    $\phi(y)\geq \phi(x).$
Therefore,
\[
\begin{aligned}
&\sqrt{g(x)\phi(x)}
\int_d^x
\frac{y\,m(y)}{g(y)^{3/2}\phi(y)^{1/2}}
\,\mathrm dy  \\
&\qquad\leq
C\sqrt{g(x)\phi(x)}
\int_d^x
\frac{g(y)^{1/2}}{\phi(y)^{1/2}}
\,\mathrm dy  \\
&\qquad\leq
C\sqrt{g(x)}
\int_d^x g(y)^{1/2}\,\mathrm dy.
\end{aligned}
\]
Since \(g\) is decreasing and \(L<+\infty\),
\[
    \sqrt{g(x)}
    \int_d^x g(y)^{1/2}\,\mathrm dy
    \leq
    \sqrt{g(x)}\,\sqrt{g(d)}\,(L-d)
    \to0
    \qquad \text{as } x\to L^-.
\]
Thus the integral term in \eqref{eq:v_representation2} vanishes after
multiplication by \(\sqrt{g(x)\phi(x)}\). Since the homogeneous term also
vanishes, \eqref{eq:v_representation2} gives
\[
    \lim_{x\to L^-}v(x)=0.
\]
This completes the proof.
\end{proof}

We next turn to the case \(0<a\leq2/3\), where \(c_l\geq0\). In contrast to the previous regime, the solution does not terminate at a finite endpoint. Instead, the ODE extension is global, and the Neumann condition at infinity emerges from the far-field limits of \(v=u_x\) and \(g\).

\begin{proposition}\label{proposition: global_extension_subcritical}
Let $0<a\leq \frac23$. Then the solution constructed on $[-d,d]$ extends uniquely to a global solution
\[
(m,v,g)\in C^1(\mathbb R)
\]
of \eqref{eqt:gmu2_for_extension}. Moreover, on $(0,+\infty)$, we have $m(x)>0$ and $g(x)>0$, and
\[
\lim_{x\to+\infty}m(x)=0,\qquad
\lim_{x\to+\infty}v(x)=0,\qquad
\lim_{x\to+\infty}g(x)=\frac1a-\frac32 .
\]
Equivalently, the corresponding profile satisfies
\[
\lim_{x\to+\infty}u_x(x)=0,
\qquad
\lim_{x\to+\infty}\frac{u(x)}{x}=0 .
\]
\end{proposition}

\begin{proof}
We first consider the endpoint case \(a=\frac23\). Since \(c_l=0\), the system
\eqref{eqt:gmu2_for_extension} coincides with the system studied in the
previous section. The desired global existence and uniqueness therefore
follow from Corollary~\ref{corollary: global_extension} and Proposition~\ref{proposition: maximal_interval}. The stated limits also agree with the limits proved there, since \(\frac1a-\frac32=0\) when \(a=\frac23\). Thus it remains
to treat the case \(0<a<\frac23\).

By Proposition~\ref{proposition: local_existence2}, the solution constructed
on \([-d,d]\) extends uniquely to a maximal interval \((-L,L)\), with
\(d<L\leq+\infty\), such that
\[
    (m,v,g)\in C^1((-L,L)),
    \qquad
    g(x)>0 \quad \text{for } x\in(-L,L).
\]
By symmetry, it suffices to study the solution on \((d,L)\).

We first record the basic sign properties that persist on \((d,L)\). Since
the equation for \(m\) has the form
\[
    m_x
    =
    \left(
        \frac{2xg_x+2(1-a)(g-g(0))}{axg}
    \right)m,
\]
and \(m(d)>0\), we have
\[
    m(x)>0,\qquad d<x<L.
\]
Next, \(v<1\) also propagates. Indeed, at any point where \(v=1\), the
equation
\[
    2axg\,v_x=a v(v-1)-4x^2m
\]
gives
\[
    v_x=-\frac{2xm}{ag}<0.
\]
Since \(v(d)<1\), this prevents \(v\) from crossing the level \(1\) from
below. Hence
\[
    v(x)<1,\qquad d<x<L.
\]
The only obstruction to continuing the solution past a finite endpoint
\(L\) is either the loss of positivity of \(g\), or the unboundedness of the
solution components as \(x\to L^-\).

The proof is organized according to the sign of \(v\). If \(v\) remains nonnegative, monotonicity gives direct decay estimates. If \(v\) becomes negative, we use the one-sided invariance of the region \(\{v<0\}\) and an inward-pointing argument at the boundary \(g=0\) to rule out finite-time loss of positivity.

\textit{Case 1: \(v\) remains nonnegative.}
We consider the case
\[
    v(x)\geq0 \qquad \text{for all } x\in(d,L).
\]
Since \(0<a<\frac23\), we have \(\frac1a-\frac32>0\). By the identity \eqref{eq:xg_integral_identity2}, we have
\begin{equation}\label{eq:global_extension_subcritical_decay_g}
    g(x)
    =
    \frac1x\int_0^x v(y)\,\mathrm dy
    +
    \frac1a-\frac32
    >0,
    \qquad d<x<L .
\end{equation}
Thus \(g\) is bounded from below by a positive constant.
Moreover, since \(0\leq v<1\), \(m>0\), and \(g>0\), the equation
\[
    v_x
    =
    \frac{v(v-1)}{2xg}
    -
    \frac{2xm}{ag}
\]
implies
   $ v_x\leq0 .$
Therefore \(v\) is decreasing on \((d,L)\). As in the proof of
Proposition~\ref{proposition: global_extension_supercritical}, the same sign
argument gives that \(g\) and \(m/g^2\) are decreasing. By the representation formula
\eqref{eq:m_representation2}, we have
\begin{equation}\label{eq:global_extension_subcritical_decay_m_over_g}
        m(x)
    \leq
    C g(x)^{\frac{2}{a}}\phi^{\frac{2(1-a)}{a}} .
\end{equation}
Therefore \(m\), \(v\), and \(g\) remain bounded on any finite interval before
\(L\), while \(g\) stays positive. The local extension theorem then rules out
\(L<+\infty\). Hence \(L=+\infty\).
Using the monotonicity,
we obtain
\[
    v_x
    \leq
    \frac{v(v-1)}{2xg}
    \leq
    \frac{v\,(v(d)-1)}{2xg(0)}
    =
    -\frac{1-v(d)}{2g(0)}\,\frac{v}{x}.
\]
and consequently, for \(d<x<L\),
\begin{equation}\label{eq:global_extension_subcritical_decay_v}
    0\leq v(x)
    \leq
    v(d)\left(\frac{d}{x}\right)^{\frac{1-v(d)}{2g(0)}},
    \qquad d<x<L .
\end{equation}
The above decay estimate \eqref{eq:global_extension_subcritical_decay_g}, 
\eqref{eq:global_extension_subcritical_decay_m_over_g} and \eqref{eq:global_extension_subcritical_decay_v} gives
\[
\lim_{x\to+\infty}m(x)=0,\qquad
\lim_{x\to+\infty}v(x)=0,\qquad
\lim_{x\to+\infty}g(x)=\frac1a-\frac32 .
\]

 \textit{Case 2: \(v\) becomes negative.}   We next consider the case where \(v\) becomes negative somewhere on
\((d,L)\). We first show that, once this happens, \(v\) remains negative
afterwards. Suppose that \(v(x_0)<0\) for some \(x_0\in(d,L)\), then \(v(x)<0\) for all
\(x\in[x_0,L)\). Indeed, if \(v\) returned to zero, at the first such point
\(x_1>x_0\) one would have \(v_x(x_1)\geq0\). But evaluating
\[
    v_x
    =
    \frac{v(v-1)}{2xg}
    -
    \frac{2xm}{ag}
\]
at \(v(x_1)=0\) gives
\[
    v_x(x_1)=-\frac{2x_1m(x_1)}{ag(x_1)}<0,
\]
which is a contradiction. Moreover, before \(v\) becomes negative, the argument of the first case applies. In particular, \(g\) is bounded above by \(g(d)\) on \((d,x_0]\). Since \(v<0\) on \((x_0,L)\), the identity \[ xg_x=v-g+\frac1a-\frac32 \] implies, in particular, that \(g\) cannot increase above its previous upper bound. Hence \[ g(x)\leq g(d)<g(0),\qquad x_0<x<L . \]
Consequently, \[ \frac{g(x)-g(0)}{xg(x)}<\frac{g(x)-g(d)}{xg(x)}<0, \qquad x_0<x<L, \] and therefore the function \[ \phi(x) = \exp\left( \int_0^x \frac{g(y)-g(0)}{y g(y)}\,\mathrm dy \right) \] is decreasing on \((x_0,L)\) and satisfies
\begin{equation}\label{eq:global_extension_subcritical_decay_phi}
\begin{aligned}
    \phi(x)
    &=
    \phi(x_0)
    \exp\left(
        \int_{x_0}^x
        \frac{g(y)-g(0)}{y g(y)}\,\mathrm dy
    \right)  \\
    &\leq
    \phi(x_0)
    \left(\frac{x}{x_0}\right)^{1-\frac{g(0)}{g(d)}} .
\end{aligned}
\end{equation}

We claim that \(L<+\infty\) is impossible. Suppose, for contradiction, that \(L<+\infty\). Since \(\phi\) is positive and decreasing on \((x_0,L)\), it has a finite nonnegative limit as \(x\to L^-\). By the estimate \eqref{eq:global_extension_subcritical_decay_m_over_g}, we have
\begin{equation}
        m(x)
    \leq
    C g(x)^{\frac{2}{a}}\phi^{\frac{2(1-a)}{a}},
    \qquad x_0<x<L .
\end{equation}
In particular, if \(g\) is bounded on \((x_0,L)\), then \(m\) is also bounded
there.
Moreover, using \eqref{eq:v_representation2} and the above bound on \(m\), we
obtain, as in the proof of
Proposition~\ref{proposition: global_extension_supercritical},
\begin{equation}\label{eq:global_extension_subcritical_decay_integral}
\begin{aligned}
&\sqrt{g(x)\phi(x)}
\int_d^x
\frac{y\,m(y)}{g(y)^{3/2}\phi(y)^{1/2}}
\,\mathrm dy \\
&\qquad\leq
C\sqrt{g(x)\phi(x)}
+
C\sqrt{g(x)}
\int_{x_0}^x
y\,g(y)^{\frac2a-\frac32}
\,\mathrm dy .
\end{aligned}
\end{equation}
Since \(0<a<\frac23\), we have \(\frac2a-\frac32>0\). Therefore, if
\(L<+\infty\) and \(g\) is bounded on \((x_0,L)\), the right-hand side is
bounded. Hence the integral term in \eqref{eq:v_representation2}, after
multiplication by \(\sqrt{g(x)\phi(x)}\), remains bounded. The homogeneous
term in \eqref{eq:v_representation2} is also bounded under the same
assumption. Consequently \(v\) remains bounded on \((x_0,L)\).
Thus, if a finite endpoint \(L\) occurs and \(g\) remains bounded above and
bounded away from zero, then both \(m\) and \(v\) remain bounded. The ODE
system is then nondegenerate near \(L\), so the solution can be continued past
\(L\), contradicting maximality. Hence the only possible finite-endpoint
obstruction is
\[
    \liminf_{x\to L^-}g(x)=0 .
\]
However, even this remaining possibility leads to a contradiction. Since \(L<+\infty\), \(g\) is bounded on \((x_0,L)\). Using
the estimate \eqref{eq:global_extension_subcritical_decay_integral} above  we obtain
\[
\sqrt{g(x)\phi(x)}
\int_d^x
\frac{y\,m(y)}{g(y)^{3/2}\phi(y)^{1/2}}
\,\mathrm dy 
\to0
\]
as \(g(x)\to0\). Hence, by \eqref{eq:v_representation2},
\[
    v(x)\to0
    \qquad \text{as } g(x)\to0 .
\]
Therefore, when \(g(x)\) is sufficiently small, we have
\[
    xg_x=v(x)-g(x)+\frac1a-\frac32
    >
    \frac12\left(\frac1a-\frac32\right)>0 .
\]
Since \(x>0\), this gives \(g_x(x)>0\) whenever \(g(x)\) is sufficiently
small. This is the inward-pointing property of the boundary \(g=0\), and it
prevents \(g\) from approaching zero from within the region \(g>0\). This
contradicts \(\liminf_{x\to L^-}g(x)=0\).

Hence
 we obtain $L=+\infty$ and $g(x)>0$ on $(0,+\infty)$. We now derive the far-field limits. Recall the estimate \eqref{eq:global_extension_subcritical_decay_phi} we have
\[
\begin{aligned}
\phi(x)
\leq
\phi(x_0)
\left(\frac{x}{x_0}\right)^{1-\frac{g(0)}{g(d)}} \to 0,\qquad \text{ as } x\to +\infty.
\end{aligned}
\] Using that \(g\) is bounded on \((x_0,+\infty)\), the estimate \eqref{eq:global_extension_subcritical_decay_integral} and  the representation formula \eqref{eq:v_representation2} we have
\[
    \lim_{x\to+\infty}v(x)=0 .
\]
The equation \eqref{eq:global_extension_subcritical_decay_g}
and the estimate \eqref{eq:global_extension_subcritical_decay_m_over_g} gives 
\[
\lim_{x\to+\infty}m(x)=0,\qquad
\lim_{x\to+\infty}v(x)=0,\qquad
\lim_{x\to+\infty}g(x)=\frac1a-\frac32 .
\]

This completes the proof in both cases.
\end{proof}

\subsection{Regularity of the solution}
After the extension step, we upgrade the regularity of the constructed profiles. The argument is the same as in the periodic case: near the origin, the fixed-point formulation removes the apparent singularity, while away from the origin the equations form a smooth ODE system as long as \(g>0\).

\begin{proposition}
\label{proposition: regularity_solution2}
Let \(a\in(2/3,1]\), and let \((m,v,g)\in C^1((-L,L))\) be the solution
constructed above for \eqref{eqt:gmu2_for_extension}. Then
\[
    m,\ v,\ g\in C^\infty((-L,L)).
\]
\end{proposition}

\begin{proof}
The argument follows the same bootstrapping idea as in
Proposition~\ref{proposition: regularity_solution}. On any compact subinterval
of \((-L,L)\), the positivity of \(g\) keeps
\eqref{eqt:gmu2_for_extension} away from its singular set, so the right-hand
side is smooth in \((m,v,g)\). Starting from the \(C^1\) solution, standard
ODE regularity then yields higher regularity successively. We therefore omit
the details.
\end{proof}

For the global profiles in the regime \(0<a\leq2/3\), the same argument gives smoothness on the entire real line.

\begin{proposition}
\label{proposition: regularity_solution_subcritical}
Let \(a \in (0,2/3]\), and let \((m,v,g)\in C^1(\mathbb R)\) be the global
solution constructed above for \eqref{eqt:gmu2_for_extension}. Then
\[
    m,\ v,\ g\in C^\infty(\mathbb R).
\]
\end{proposition}

\begin{proof}
The proof is a straightforward adaptation of the argument for
Proposition~\ref{proposition: regularity_solution}. Since \(g>0\) on
\(\mathbb R\), the system \eqref{eqt:gmu2_for_extension} is locally a smooth
ODE system for \((m,v,g)\). Starting from the \(C^1\) solution, standard ODE
regularity yields higher regularity successively. We omit the details.
\end{proof}

For \(a>2/3\), the solution reaches a finite endpoint \(L\). To understand the regularity of the compactly supported extension of \(\omega\) and \(\theta\), we need precise asymptotics as \(x\to L^-\). The following proposition provides these endpoint expansions.

\begin{proposition}
\label{proposition: boundary_asymptotics2}
Let \(a\in(2/3,1]\), and let \((m,v,g)\in C^\infty((-L,L))\) be the finite-endpoint
solution constructed above for \eqref{eqt:gmu2_for_extension}. Then, as
\(x\to L^-\), the following asymptotic expansions hold.

\begin{enumerate}
\item The function \(g\) satisfies
\[
    g(x)=\frac{3a-2}{2aL}(L-x)+o(L-x).
\]

\item There exists a constant \(c_m>0\) such that
\[
    m(x)
    =
    c_m\,(L-x)^{\frac{2a}{3a-2}}
    +
    o\!\left((L-x)^{\frac{2a}{3a-2}}\right).
\]

\item There exists a constant \(c_v\in\mathbb R\) such that
\[
    v(x)
    =
    c_v\,(L-x)^{\frac{a}{3a-2}}
    +
    o\!\left((L-x)^{\frac{a}{3a-2}}\right).
\]

\item The same constant \(c_v\) satisfies
\[
    v_x(x)
    =
    -\frac{a}{3a-2}\,c_v\,
    (L-x)^{\frac{2-2a}{3a-2}}
    +
    o\!\left((L-x)^{\frac{2-2a}{3a-2}}\right).
\]
In particular, if \(a\in(2/3,1)\), then
\[
    v_x(x)\to0
    \qquad\text{as }x\to L^-.
\]
\end{enumerate}
\end{proposition}

\begin{proof}
We use the same strategy as in
Proposition~\ref{proposition: boundary_asymptotics}, with the endpoint limits
obtained in Proposition~\ref{proposition: global_extension_supercritical}.
Recall that
\[
    g(x)\to0,\qquad m(x)\to0,\qquad v(x)\to0
    \qquad \text{as } x\to L^- .
\]

We first prove the asymptotic behavior of \(g\). From
\[
    xg_x=v-g+\frac1a-\frac32
\]
and the endpoint limits above, we obtain
\[
    \lim_{x\to L^-}g_x(x)
    =
    \frac{1}{L}\left(\frac1a-\frac32\right)
    =
    -\frac{3a-2}{2aL}.
\]
Hence
\[
    g(x)=\frac{3a-2}{2aL}(L-x)+o(L-x).
\]

Next, define
\[
    \phi(x)
    :=
    \exp\left(
        \int_0^x
        \frac{g(y)-g(0)}{y g(y)}\,\mathrm dy
    \right).
\]
Since
\[
    \frac{\phi_x}{\phi}
    =
    \frac{g-g(0)}{xg},
\]
and
\[
    xg(x)
    =
    \frac{3a-2}{2a}(L-x)+o(L-x),
\]
we get
\[
    \frac{\phi_x}{\phi}
    =
    -\frac{g(0)}{xg(x)}+O(1)
    =
    -\frac{2a\,g(0)}{3a-2}\frac1{L-x}
    +O(1).
\]
Using \(g(0)=\frac1a-\frac12=\frac{2-a}{2a}\), this becomes
\[
    \frac{\phi_x}{\phi}
    =
    -\frac{2-a}{3a-2}\frac1{L-x}
    +O(1).
\]
Therefore there exists \(c_\phi>0\) such that
\[
    \phi(x)
    =
    c_\phi (L-x)^{\frac{2-a}{3a-2}}(1+o(1)).
\]

We now prove the asymptotic behavior of \(m\). By the representation formula
\[
    m(x)
    =
    \frac{4-3a}{4}
    \left(\frac{g(x)}{g(0)}\right)^{\frac2a}
    \phi(x)^{\frac{2(1-a)}{a}},
\]
the preceding expansions for \(g\) and \(\phi\) imply that there exists
\(c_m>0\) such that
\[
    m(x)
    =
    c_m
    (L-x)^{
        \frac2a+
        \frac{2(1-a)(2-a)}{a(3a-2)}
    }
    (1+o(1)).
\]

It remains to obtain the expansion for \(v\). By the representation formula
\eqref{eq:v_representation2},
\[
v(x)
=
\sqrt{g(x)\phi(x)}
\left(
\frac{v(d)}{\sqrt{g(d)\phi(d)}}
-
\frac{2}{a}\int_d^x
\frac{y\,m(y)}{g(y)^{3/2}\phi(y)^{1/2}}
\,\mathrm dy
\right).
\]
Using the formulas for \(m\), \(g\), and \(\phi\), the integrand has the form
\[
    \frac{y\,m(y)}{g(y)^{3/2}\phi(y)^{1/2}}
    =
    O\!\left(
        (L-y)^{
        \frac2a-\frac32+
        \left(\frac{2(1-a)}{a}-\frac12\right)
        \frac{2-a}{3a-2}}
    \right).
\]
The exponent is greater than \(-1\) for \(a\in(2/3,1]\). Hence the integral
has a finite limit as \(x\to L^-\). Since
\[
    \sqrt{g(x)\phi(x)}
    =
    C (L-x)^{\frac12+\frac{2-a}{2(3a-2)}}(1+o(1))
    =
    C (L-x)^{\frac{a}{3a-2}}(1+o(1)),
\]
there exists \(c_v\in\mathbb R\) such that
\[
    v(x)
    =
    c_v (L-x)^{\frac{a}{3a-2}}
    +
    o\!\left((L-x)^{\frac{a}{3a-2}}\right).
\]

Finally, differentiating this asymptotic expansion gives
\[
    v_x(x)
    =
    -\frac{a}{3a-2}c_v
    (L-x)^{\frac{2-2a}{3a-2}}
    +
    o\!\left((L-x)^{\frac{2-2a}{3a-2}}\right).
\]
In particular, if \(a\in(2/3,1)\), then
\[
    \frac{2-2a}{3a-2}>0,
\]
and therefore
\[
    v_x(x)\to0
    \qquad \text{as } x\to L^- .
\]
The proof is complete.
\end{proof}

We now translate the extended solution \((m,v,g)\) back to the original profile variables \((\omega,u,\theta)\). In the regime \(a>2/3\), the endpoint asymptotics allow us to extend \(\omega\) and \(\theta\) by zero outside a compact interval. In the regime \(a\leq2/3\), the global extension and far-field limits give a smooth whole-space profile satisfying the Neumann condition at infinity.

\begin{corollary}
\label{corollary: neumann_profile_extension}
Let \(0<a\leq 2/3\), and let \((m,v,g)\) be the solution obtained in
Theorem~\ref{theorem: existence of fixed point2} and
Proposition~\ref{proposition: global_extension_subcritical}. Define, for
\(x\in\mathbb R\),
\[
    u(x):=x\left(g(x)-\frac{c_l}{a}\right),\qquad
    \theta(x):=x^2m(x),\qquad
    \omega(x):=-\frac12u_{xx}(x),
\]
and set
\[
    c_l=1-\frac32a,\qquad
    c_\omega=a-1,\qquad
    c_\theta=-a.
\]
Then \((u,\theta,\omega)\) gives a symmetric whole-space solution of
\eqref{eqt: case2} .
Moreover, \(u\) and \(\omega\) are odd, \(\theta\) is even, and
\[
    u,\theta,\omega\in C^\infty(\mathbb R).
\]

\end{corollary}

\begin{proof}
By the definition used in the reduced formulation,
\[
 u=x\left(g-\frac{c_l}{a}\right), \qquad
    \theta=x^2m,\qquad
    \omega=-\frac12u_{xx},
\]
the system \eqref{eqt:gmu2} for \((m,v,g)\), with \(v=u_x\), implies
\[
    (c_l x+au)\theta_x=(c_\theta+2u_x)\theta,
\]
\[
    (c_l x+au)\omega_x=c_\omega\omega+\theta_x,
    \qquad
    -u_{xx}=2\omega .
\]
Hence \((u,\theta,\omega)\) solves \eqref{eqt:main} on \(\mathbb R\).

The parity follows from the construction: \(g,m,v\) are even, hence \(u\) and
\(\omega\) are odd, while \(\theta\) is even. Proposition~\ref{proposition: global_extension_subcritical}
gives the far-field limits
\[
    \lim_{x\to+\infty}m(x)=0,\qquad
    \lim_{x\to+\infty}v(x)=0,\qquad
    \lim_{x\to+\infty}g(x)=\frac1a-\frac32 .
\]
Since \(v=u_x\), this gives
\[
    \lim_{x\to+\infty}u_x(x)=0.
\]
By even symmetry, the same Neumann condition holds as \(x\to-\infty\).

Finally, Proposition~\ref{proposition: regularity_solution_subcritical} gives
\[
    m,v,g\in C^\infty(\mathbb R),
\]
and therefore
\[
    u,\theta,\omega\in C^\infty(\mathbb R).
\]
This completes the
proof.
\end{proof}

\begin{corollary}
\label{corollary: whole_space_compact_profile}
Let \(2/3<a\leq 1\), and let \((m,v,g)\) be the solution obtained in
Theorem~\ref{theorem: existence of fixed point2} and
Corollary~\ref{proposition: global_extension_supercritical}. Define, for \(x\in(-L,L)\),
\[
    u(x):=x\left(g(x)-\frac{c_l}{a}\right),\qquad
    \theta(x):=x^2m(x),\qquad
    \omega(x):=-\frac12u_{xx}(x),
\]
and set
\[
    c_l=1-\frac32a,\qquad
    c_\omega=a-1,\qquad
    c_\theta=-a.
\]
Then \((u,\theta,\omega)\) gives a symmetric whole-space solution of
\eqref{eqt:main} satisfying the Neumann condition
\[
    \lim_{|x|\to+\infty}u_x(x)=0.
\]
Moreover, \(u\) and \(\omega\) are odd, \(\theta\) is even, and
\[
    \omega
    \in C^{\lceil\beta\rceil-1,\beta-\lceil\beta\rceil+1}(\mathbb R)
    \cap C^\infty\bigl(\mathbb R\setminus\{\pm L\}\bigr),
\]
\[
     u
    \in C^{1+\lceil\beta\rceil,\beta-\lceil\beta\rceil+1}(\mathbb R)
    \cap C^\infty\bigl(\mathbb R\setminus\{\pm L\}\bigr),
\]
\[
    \theta
    \in C^{\lceil\gamma\rceil-1,\gamma-\lceil\gamma\rceil+1}(\mathbb R)
    \cap C^\infty\bigl(\mathbb R\setminus\{\pm L\}\bigr).
\]
\[
    \beta:=\frac{1-a}{3a-2},
    \qquad
    \gamma:=\frac2a+\frac{2(1-a)}{3a-2}.
\]
If \(a=1\), we have
\[
\omega\in C^{0,1}(\mathbb R)\cap C^\infty\bigl(\mathbb R\setminus\{\pm L\}\bigr),
\]
\[
\theta\in C^{1,1}(\mathbb R)\cap C^\infty\bigl(\mathbb R\setminus\{\pm L\}\bigr),
\]
\[
    u\in C^{2,1}(\mathbb R)\cap C^\infty\bigl(\mathbb R\setminus\{\pm L\}\bigr).
\]
\end{corollary}

\begin{proof}
On \((-L,L)\), the definitions
\[
    u=x\left(g-\frac{c_l}{a}\right),\qquad
    \theta=x^2m,\qquad
    \omega=-\frac12u_{xx}
\]
together with the system for \((m,v,g)\) imply
\[
    (c_lx+au)\theta_x=(c_\theta+2u_x)\theta,
\]
\[
    (c_lx+au)\omega_x=c_\omega\omega+\theta_x,
    \qquad
    -u_{xx}=2\omega .
\]
Thus \((u,\theta,\omega)\) solves \eqref{eqt:main} on \((-L,L)\).

The parity follows from the construction: \(g,m,v\) are even, hence \(u\) and
\(\omega\) are odd, while \(\theta\) is even. By the endpoint behavior in
Proposition~\ref{proposition: boundary_asymptotics2}, we have
\[
    \theta(x)\to0,\qquad \omega(x)\to0,
    \qquad x\to L^-.
\]
The same holds at \(x=-L\) by symmetry. Hence we may extend
\(\theta\) and \(\omega\) to the whole real line by
\[
    \theta(x)=0,\qquad \omega(x)=0,
    \qquad |x|\geq L.
\]
Thus
\[
    \supp \omega\cup\supp \theta\subset[-L,L].
\]
The  Neumann condition
\[
    \lim_{|x|\to+\infty}u_x(x)=\lim_{|x|\to L^{-}}u_x(x)=\lim_{|x|\to L^{-}}v(x)=0,
\]
is also satisfied. 

The smoothness in the interior follows from
Proposition~\ref{proposition: regularity_solution2}, and the regularity across
the endpoints follows from Proposition~\ref{proposition: boundary_asymptotics2}.
This completes the proof. When \(a=1\), we can construct an explicit compactly supported whole-space
profile satisfying the same normalization conditions. On \([-L,L]\), take
\[
    u(x)=\frac{x}{2}+\frac14\sin(2x),
    \qquad
    \omega(x)=\frac12\sin(2x),
    \qquad
    \theta(x)=\frac1{16}\sin^2(2x),
\]
where $L=\pi/2$, $c_l=-1/2$, $c_\omega=0$, $c_\theta=-1$.
We extend \(\omega\) and \(\theta\) by zero for \(x\geq L\), and extend \(u\)
by the constant value \(u(L)=\pi/4\). By odd/even symmetry, this defines the
whole-space profile on \(\mathbb R\).
Therefore the compactly supported extension satisfies
\[
    \omega\in C^{0,1}(\mathbb R),\qquad
    \theta\in C^{1,1}(\mathbb R),\qquad
    u\in C^{2,1}(\mathbb R).
\]
In the next subsection, we prove that this is the unique whole-space profile satisfying some normalization conditions. This completes the
proof.
\end{proof}

\subsection{Uniqueness of the solution}
We finally prove the uniqueness of the normalized whole-space profile. As in the periodic case, uniqueness near the origin follows from the singular integral formulation and the prescribed leading behavior, while uniqueness away from the origin follows from the standard uniqueness theorem for the ODE system.

\begin{proposition}
\label{proposition: uniqueness2}
Let \(a\in[0,1]\). Let \((m_1,v_1,g_1)\) and \((m_2,v_2,g_2)\) be two
solutions of \eqref{eqt:gmu2} or \eqref{eqt:gmu2_a0} on \((-L,L)\), where
\(v_i:=u_{i,x}\). Assume that they satisfy the same initial data
\[
    2m_i(0)=2-\frac32a,
    \qquad
    v_i(0)=1,
\]
and the regularity conditions
\[
    v_i(x)-1=O(x^2),
    \qquad
    g_i(x)-g_i(0)=O(x^2)
    \qquad \text{as } x\to0 .
\]
Then
\[
    m_1\equiv m_2,
    \qquad
    v_1\equiv v_2,
    \qquad
    g_1\equiv g_2
    \qquad \text{on } (-L,L).
\]
\end{proposition}

\begin{proof}
The proof is a straightforward adaptation of the uniqueness argument for
Proposition~\ref{proposition: uniqueness}. The regularity assumptions remove
the apparent singularity at \(x=0\), and away from the origin the system
\eqref{eqt:gmu2_for_extension} is a locally Lipschitz first-order ODE system
as long as \(g>0\). Applying the same comparison argument on the two sides of
the origin gives the desired uniqueness. We omit the details.
\end{proof}

This completes the construction and analysis of the whole-space profiles satisfying the Neumann condition at infinity. These profiles will be used in the next section to construct exact self-similar finite-time blowup solutions of the evolution equation \eqref{eqt:gHL1}.

\section{Construction of finite-time blowup}
We now explain how the self-similar profiles constructed in the previous sections generate finite-time blowup solutions of the evolution equation \eqref{eqt:gHL1}. We treat the periodic and whole-space settings separately. In the periodic case, the profile has a fixed self-similar length scale, while in the whole-space case the sign of the self-similar scaling parameter determines whether the blowup is focusing, neither focusing nor expanding, or expanding.

\subsection{Periodic Setting}
We first deal with the periodic setting. We consider \eqref{eqt:gHL1} on the periodic
domain
    $\mathbb S^1:=\mathbb R/(2\pi\mathbb Z)$,
\begin{equation}\label{eqt:gHL1_periodic}
\begin{aligned}
\theta_t+a u\theta_x &= 2u_x\theta,\\
\omega_t+a u\omega_x &= \theta_x,\\
-u_{xx} = 2\omega,&\qquad u(0,t)=0.
\end{aligned}
\end{equation}
Here \(\omega\) and \(\theta\) are \(2\pi\)-periodic, while \(u\) is determined
from \(-u_{xx}=2\omega\) by
solving the periodic Poisson equation.

The periodic profiles constructed in Section~2 have period \(2L\). We first rescale them to period \(2\pi\), and then use the corresponding self-similar time exponents to obtain an exact solution of the periodic evolution equation.

\begin{theorem}
\label{theorem: periodic_exact_blowup}
Let \(2/3<a<1\). Then equation \eqref{eqt:gHL1_periodic} on
\(\mathbb S^1\) admits a nontrivial self-similar solution which blows up in
finite time.
\end{theorem}

\begin{proof}
By Corollary~\ref{corollary: periodic_profile_extension}, there exists a
symmetric \(2L\)-periodic profile \((\bar\Omega,\bar U,\bar\Theta)\) solving
\[
    a\bar U\bar\Theta_x=(c_\theta+2\bar U_x)\bar\Theta,
    \qquad
    a\bar U\bar\Omega_x=c_\omega\bar\Omega+\bar\Theta_x,
    \qquad
    -\bar U_{xx}=2\bar\Omega,
\]
with
\[
    c_l=0,\qquad c_\omega=a-1,\qquad c_\theta=2(a-1).
\]
The following rescaling is chosen according to the scaling invariance
\eqref{eq:profile_scaling_invariance}. It rescales the profile from period
\(2L\) to period \(2\pi\), and normalizes the self-similar time exponents to
\[
    \widetilde c_l=0,\qquad
    \widetilde c_\omega=-1,\qquad
    \widetilde c_\theta=-2.
\]
For any \(T>0\), define
\[
    \omega(x,t)
    :=
    \frac{L}{(1-a)\pi(T-t)}
    \bar\Omega\!\left(\frac{Lx}{\pi}\right),
\]
\[
    u(x,t)
    :=
    \frac{\pi}{(1-a)L(T-t)}
    \bar U\!\left(\frac{Lx}{\pi}\right),
\]
\[
    \theta(x,t)
    :=
    \frac{1}{(1-a)^2(T-t)^2}
    \bar\Theta\!\left(\frac{Lx}{\pi}\right).
\]
Since \((\bar\Omega,\bar U,\bar\Theta)\) is \(2L\)-periodic, the constructed functions are
\(2\pi\)-periodic in \(x\). By direct substitution, the profile equations
imply
\[
    \theta_t+a u\theta_x=2u_x\theta,
    \qquad
    \omega_t+a u\omega_x=\theta_x,
    \qquad
    -u_{xx}=2\omega.
\]
Thus \((\omega,u,\theta)\) is an exact solution of
\eqref{eqt:gHL1_periodic}. Since the time-dependent prefactors are singular
at \(t=T\) and the profile is nontrivial, this gives an exact self-similar finite-time blowup solution.
\end{proof}

\subsection{Whole Space Setting}
We now turn to the whole-space setting. In this case, the profiles constructed in Section~3 satisfy a Neumann condition at infinity, and the sign of the scaling parameter determines the nature of the blowup.

We consider the evolution equation \eqref{eqt:gHL1} on the whole space $x\in\mathbb R$,
\begin{equation}\label{eqt:gHL1_whole_space}
\begin{aligned}
\theta_t+a u\theta_x &= 2u_x\theta,\\
\omega_t+a u\omega_x &= \theta_x,\\
-u_{xx} = 2\omega,&\qquad u(0,t)=0.
\end{aligned}
\end{equation}
Here \(u\) is determined by solving the Poisson equation on \(\mathbb R\),
together with the normalization \(u(0,t)=0\) and the Neumann condition at
infinity.

\begin{theorem}
\label{theorem: whole_space_exact_blowup}
Let \(0<a\leq1\). Then equation \eqref{eqt:gHL1_whole_space} admits a
nontrivial self-similar finite-time blowup solution.
More precisely, the blowup is focusing for \(0<a<2/3\), neither focusing
nor expanding for \(a=2/3\), and expanding for \(2/3<a\leq1\).
\end{theorem}

\begin{proof}
By Corollaries~\ref{corollary: neumann_profile_extension}
and~\ref{corollary: whole_space_compact_profile}, there exists a
symmetric whole-space profile \((\bar\Omega,\bar U,\bar\Theta)\) solving
\[
    (c_lx+a\bar U)\bar \Theta_x=(c_\theta+2\bar U_x)\bar\Theta,
    \qquad
    (c_lx+a\bar U)\bar\Omega_x=c_\omega\bar\Omega+\bar\Theta_x,
    \qquad
    -\bar U_{xx}=2\bar \Omega,
\]
\[
\lim_{|x|\to+\infty}\bar U_x(x)=0.
\]
with the parameters
\[
    c_l=1-\frac32a,\qquad c_\omega=a-1,\qquad c_\theta=-a.
\]
For any \(T>0\), set
\[
    \omega(x,t)
    :=
    \frac2a (T-t)^{2-\frac2a}\Omega\left(\frac{x}{(T-t)^{\frac2a-3}}\right),
\]
\[
    u(x,t)
    :=
    \frac2a (T-t)^{\frac2a-4}\bar U\left(\frac{x}{(T-t)^{\frac2a-3}}\right),
\]
\[
    \theta(x,t)
    :=
    \frac4{a^2}(T-t)^{-2}\bar\Theta\left(\frac{x}{(T-t)^{\frac2a-3}}\right).
\]
By direct substitution, the profile equations imply
\[
    \theta_t+a u\theta_x=2u_x\theta,
    \qquad
    \omega_t+a u\omega_x=\theta_x,
    \qquad
    -u_{xx}=2\omega.
\]
Thus \((\omega,u,\theta)\) is an exact solution of
\eqref{eqt:gHL1_whole_space}. Since the time-dependent prefactors are
singular at \(t=T\) and the profile is nontrivial, this gives a self-similar
finite-time blowup solution.

The constants are chosen according to the scaling invariance of the profile
system described in \eqref{eq:profile_scaling_invariance},  with scaling factor \(2/a\), so that the effective self-similar scaling parameters are
\[
    \widetilde c_l=\frac2a-3,\qquad
    \widetilde c_\omega=2-\frac2a,\qquad
    \widetilde c_\theta=-2.
\]
The sign of \(\widetilde c_l\) determines the nature of the self-similar
scale. Indeed, the length scale is
    $(T-t)^{\widetilde c_l}.$
If \(0<a<\frac23\), then \(\widetilde c_l>0\), and the length scale shrinks to
zero as \(t\to T^-\), so the blowup is focusing. If \(a=\frac23\), then
\(\widetilde c_l=0\), so the length scale remains fixed and the blowup is
neither focusing nor expanding. If \(\frac23<a\leq1\), then
\(\widetilde c_l<0\), and the length scale expands to infinity as
\(t\to T^-\), so the blowup is expanding.
\end{proof}

This completes the passage from the self-similar profiles constructed in Sections~2 and~3 to finite-time blowup solutions of the evolution equation \eqref{eqt:gHL1}.

\section{Numerical Simulation}

In this section we present numerical simulations of the self-similar profiles constructed in the previous sections. The numerical results serve two purposes. First, they validate the fixed-point formulation by comparing the computed profiles with explicit solutions in some special case. Second, they illustrate how the profiles depend on the advection parameter \(a\) in the periodic and whole-space settings.

\subsection{Numerical scheme}

We compute the profiles by an iterative scheme for the normalized variables \(v=u_x\), \(g=u/x\), and \(m=\theta/x^2\). In the periodic case we use the normalized system \eqref{eqt:gmu}, while in the whole-space case with a Neumann condition \(u_x(\infty)=0\) we use \eqref{eqt:gmu2}. The iteration is performed on the whole computational domain, rather than by first computing a local profile and then applying a separate ODE extension.

Given \(v^{(n)}\), we first update \(g^{(n)}\) using the equation for \(xg\). In the periodic case, the third equation of \eqref{eqt:gmu} gives
\[
    g^{(n)}(x)=\frac1x\int_0^x v^{(n)}(y)\idiff y,
    \qquad g^{(n)}(0)=1 .
\]
In the whole-space case, the third equation of \eqref{eqt:gmu2} gives
\[
    g^{(n)}(x)
    =
    \frac1x\int_0^x
    \left(v^{(n)}(y)+\frac1a-\frac32\right)\idiff y,
    \qquad
    g^{(n)}(0)=\frac1a-\frac12 .
\]
If \(g^{(n)}\) reaches zero on the computational grid, we truncate the profile at the first zero. More precisely, we define
\[
    L^{(n)}:=\inf\{x>0:\ g^{(n)}(x)\leq0\},
\]
with \(L^{(n)}\) determined by interpolation between grid points, and then carry out the remaining updates only on \([0,L^{(n)}]\). The value of \(g^{(n)}\) at the endpoint is set to be zero. This truncation is consistent with the endpoint behavior in the compactly supported or periodic cases, where the normalized profile reaches the boundary of its support when \(g\) vanishes. Beyond this point the active part of the profile is not evolved further in the iteration.

Once \(g^{(n)}\) is determined, we compute \(m^{(n)}\) explicitly from the first equation of the corresponding normalized system.  More precisely, integrating the first equation of \eqref{eqt:gmu} gives, in the periodic case,
\[
    m^{(n)}(x)
    =
    \frac12
    \bigl(g^{(n)}(x)\bigr)^{2/a}
    \exp\!\left(
    \frac{2(1-a)}{a}
    \int_0^x
    \frac{g^{(n)}(y)-1}{y\,g^{(n)}(y)}\idiff y
    \right).
\]
Similarly, integrating the first equation of \eqref{eqt:gmu2} gives, in the whole-space case,
\[
    m^{(n)}(x)
    =
    \frac{4-3a}{4}
    \left(\frac{g^{(n)}(x)}{g^{(n)}(0)}\right)^{2/a}
    \exp\!\left(
    \frac{2(1-a)}{a}
    \int_0^x
    \frac{g^{(n)}(y)-g^{(n)}(0)}
         {y\,g^{(n)}(y)}\idiff y
    \right).
\]
Thus the equations for \(g\) and \(m\) are solved explicitly at each step of the iteration.

It remains to update \(v\).  We treat the nonlinear equation for \(v\) by a semi-implicit linearization.  In the periodic case, the second equation of \eqref{eqt:gmu} contains the nonlinear factor \((av+3a-2)(v-1)\).  On \(0\leq x\leq1\), we freeze the first factor at the previous iterate and solve
\[
    2axg^{(n)}(x)\,\partial_x v^{(n+1)}(x)
    =
    \bigl(av^{(n)}(x)+3a-2\bigr)
    \bigl(v^{(n+1)}(x)-1\bigr)
    -4x^2m^{(n)}(x).
\]
On \(x>1\), we instead freeze the second factor and solve
\[
    2axg^{(n)}(x)\,\partial_x v^{(n+1)}(x)
    =
    \bigl(av^{(n+1)}(x)+3a-2\bigr)
    \bigl(v^{(n)}(x)-1\bigr)
    -4x^2m^{(n)}(x).
\]
In the whole-space case, the second equation of \eqref{eqt:gmu2} contains the nonlinear factor \(av(v-1)\).  We use the same strategy:
\[
    2axg^{(n)}(x)\,\partial_x v^{(n+1)}(x)
    =
    a v^{(n)}(x)\bigl(v^{(n+1)}(x)-1\bigr)
    -4x^2m^{(n)}(x),
    \qquad 0\leq x\leq1,
\]
and
\[
    2axg^{(n)}(x)\,\partial_x v^{(n+1)}(x)
    =
    a v^{(n+1)}(x)\bigl(v^{(n)}(x)-1\bigr)
    -4x^2m^{(n)}(x),
    \qquad x>1.
\]
Each of these updates is a first-order linear ODE for \(v^{(n+1)}\), and hence can be solved explicitly by an integrating factor.  The value \(v^{(n+1)}(0)=1\) is imposed by the normalization at the origin, and the solution on \(x>1\) is initialized by the value obtained from the interval \(0\leq x\leq1\).  The inner update is the same local fixed-point iteration used in the existence proof, while the outer update is chosen to better capture the far-field or endpoint behavior, where the local monotonicity used in the proof is no longer available globally.

We use a composite nonuniform grid on \([0,10^8+10]\). More precisely, we take a uniform grid on \([0,10]\) with step size \(h=10^{-5}\), and a stretched grid on \((10,10^8+10]\) defined by \(x_n=10^8(n/1000)^6+10\), \(n=1,2,\ldots,1000\). All integrals are evaluated by the trapezoidal rule. The apparent singularities at the origin are removed by using the limiting values dictated by the normalization.
We start from the smooth initial guess \(v_0(x)=1/(1+x^2)\). The iteration is robust with respect to this choice. We declare convergence when \(\|v^{(n+1)}-v^{(n)}\|_{L^\infty}<10^{-15}\), where the \(L^\infty\)-norm is computed over the numerical grid. In all cases tested, the iteration converges rapidly, typically within \(21\) steps.

After convergence, we reconstruct the profile variables from \(g\), \(m\), and \(v\). In the periodic case we set \(u=xg\), \(\theta=x^2m\), and \(\omega=-u_{xx}/2\). In the whole-space case, we use the normalization of Section~3 and reconstruct \(u=x(g-c_l/a)\), \(\theta=x^2m\), and \(\omega=-u_{xx}/2\). The computation is carried out only for \(x\geq0\), and the full profiles are recovered by symmetry: \(u\) and \(\omega\) are odd, while \(v\), \(g\), \(m\), and \(\theta\) are even.

\subsection{Validation at special parameter values}

We first validate the numerical method by comparing the computed profiles with explicit solutions at special parameter values. We use two checks. The first one is at \(a=1\), where explicit profiles are available in both the periodic and whole-space settings. The second one is in the whole-space setting for a small value of \(a\), where the numerical profile can be compared with the explicit limiting profile at \(a=0\). In all comparisons below, we focus on the normalized quantities \(v=u_x\) and \(m=\theta/x^2\).

We begin with \(a=1\). In the periodic setting, the explicit smooth profile is \(u_{\mathrm{ex}}(x)=2^{-1/2}\sin(\sqrt2 x)\), \(\omega_{\mathrm{ex}}(x)=2^{-1/2}\sin(\sqrt2 x)\), and \(\theta_{\mathrm{ex}}(x)=\frac14\sin^2(\sqrt2 x)\), with half-period \(L=\pi/\sqrt2\). The corresponding normalized quantities are
\[
    v_{\mathrm{ex}}(x)=\cos(\sqrt2 x),
    \qquad
    m_{\mathrm{ex}}(x)=\frac{\sin^2(\sqrt2 x)}{4x^2}.
\]
Here \(m_{\mathrm{ex}}(0)\) is understood in the limiting sense, so \(m_{\mathrm{ex}}(0)=1/2\).

In the whole-space setting with a Neumann condition, the explicit compactly supported profile at \(a=1\) is given on \([-L,L]\) by \(u_{\mathrm{ex}}(x)=x/2+\frac14\sin(2x)\), \(\omega_{\mathrm{ex}}(x)=\frac12\sin(2x)\), and \(\theta_{\mathrm{ex}}(x)=\frac1{16}\sin^2(2x)\), where \(L=\pi/2\). On \([-L,L]\), the corresponding normalized quantities are
\[
    v_{\mathrm{ex}}(x)=\frac12+\frac12\cos(2x)=\cos^2 x,
    \qquad
    m_{\mathrm{ex}}(x)=\frac{\sin^2(2x)}{16x^2}.
\]
Again \(m_{\mathrm{ex}}(0)\) is defined by its limiting value, which is \(1/4\). Figure~\ref{fig:validation_a1} compares the numerical profiles with these explicit solutions in both settings.

We also test the whole-space scheme near the limiting case \(a=0\). Although the main construction is for \(a>0\), the limiting equation at \(a=0\) admits an explicit profile. In this limit, \(u_0(x)=-x+2\sqrt2\arctan(x/\sqrt2)\) under the normalization \(u_0(0)=0\). The corresponding normalized quantities are
\[
    v_0(x)=\frac{1-x^2/2}{1+x^2/2},
    \qquad
    m_0(x)=\frac{1}{(1+x^2/2)^2}.
\]
We therefore compute the whole-space profile at \(a=0.001\) and compare its \(v\) and \(m\) components with \(v_0\) and \(m_0\). This comparison is shown in Figure~\ref{fig:validation_a0}. Together, these tests validate the semi-implicit iteration, the quadrature rule, and the reconstruction of the normalized profile variables in both the compactly supported and full-support regimes.

\begin{figure}[htbp]
    \centering
    \begin{subfigure}{0.49\textwidth}
        \includegraphics[width=\textwidth]{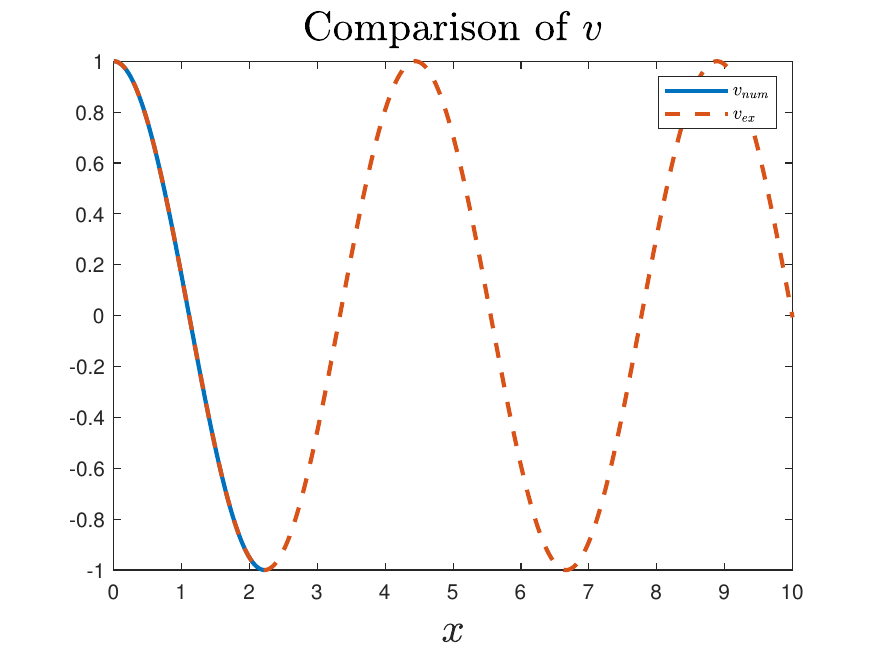}
        \caption{Periodic case: comparison of \(v\).}
        \label{fig:validation_periodic_a1_v}
    \end{subfigure}
    \begin{subfigure}{0.49\textwidth}
        \includegraphics[width=\textwidth]{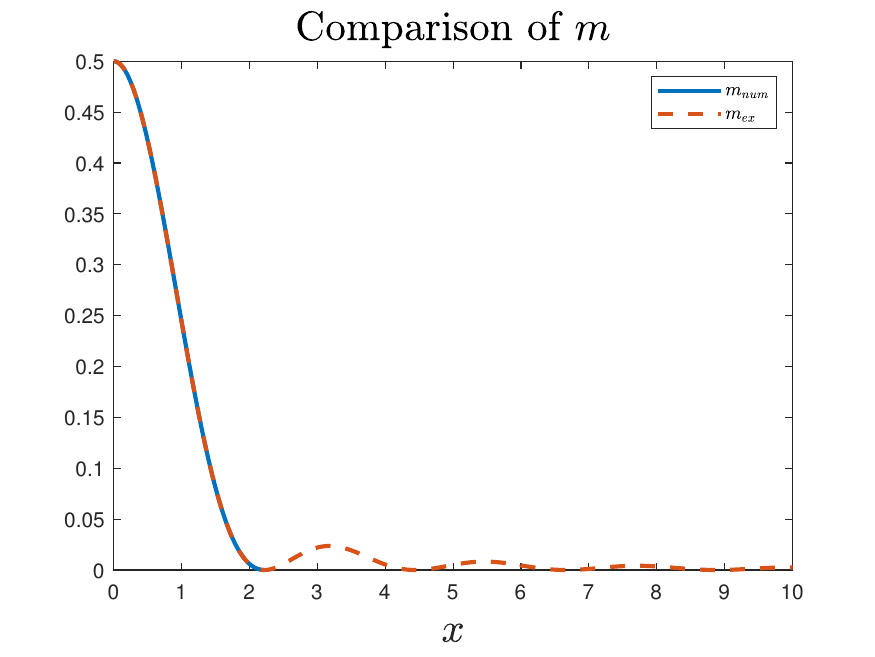}
        \caption{Periodic case: comparison of \(m\).}
        \label{fig:validation_periodic_a1_m}
    \end{subfigure}
    \begin{subfigure}{0.49\textwidth}
        \includegraphics[width=\textwidth]{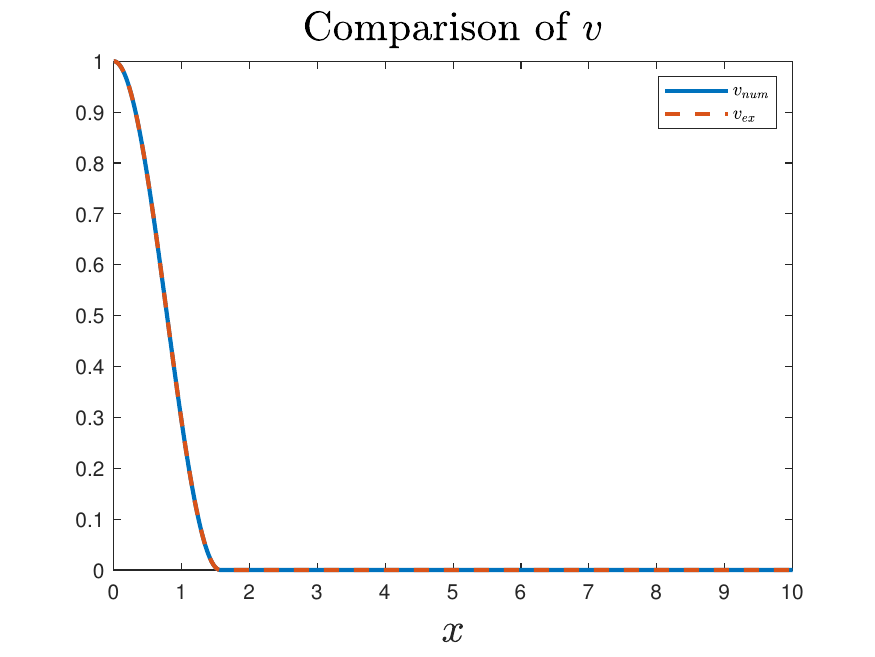}
        \caption{Whole-space case: comparison of \(v\).}
        \label{fig:validation_wholespace_a1_v}
    \end{subfigure}
    \begin{subfigure}{0.49\textwidth}
        \includegraphics[width=\textwidth]{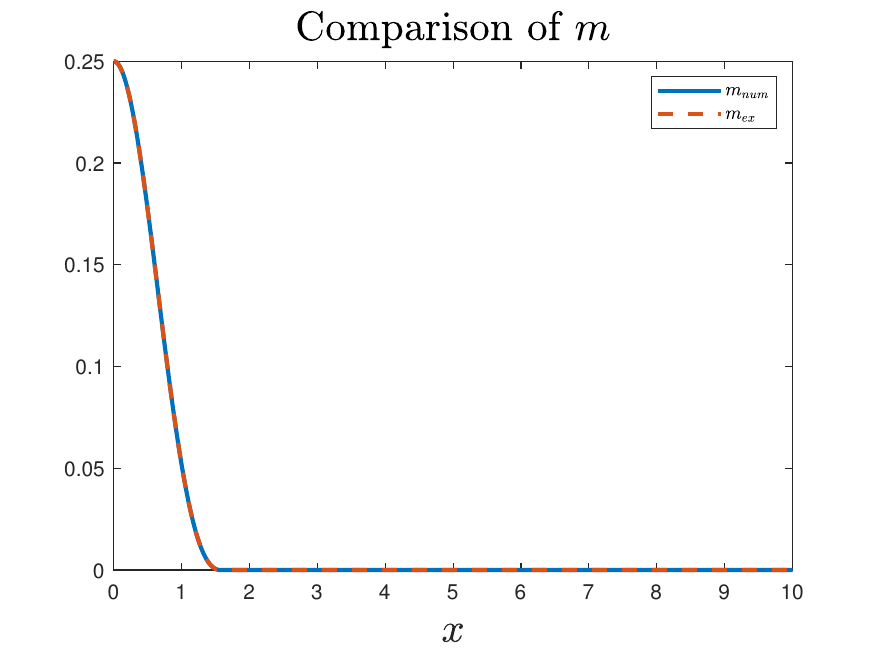}
        \caption{Whole-space case: comparison of \(m\).}
        \label{fig:validation_wholespace_a1_m}
    \end{subfigure}
    \caption{Validation of the numerical method at \(a=1\). The numerical profiles are compared with the explicit profiles for \(v=u_x\) and \(m=\theta/x^2\) in the periodic and whole-space settings.}
    \label{fig:validation_a1}
\end{figure}

\begin{figure}[htbp]
    \centering
    \begin{subfigure}{0.49\textwidth}
         \includegraphics[width=\textwidth]{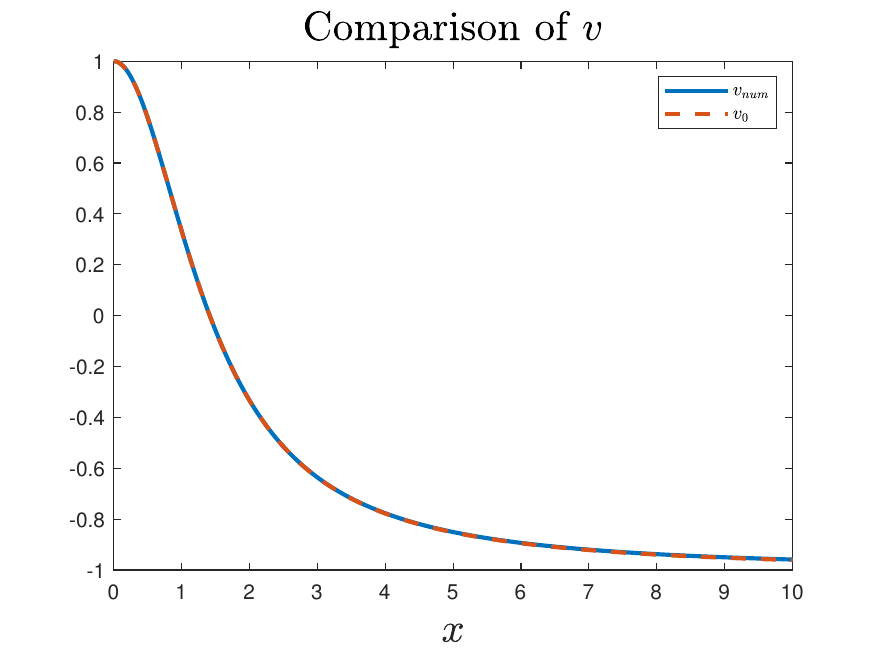}
        \caption{Comparison of \(v\).}
        \label{fig:validation_a0_v}
    \end{subfigure}
    \begin{subfigure}{0.49\textwidth}
         \includegraphics[width=\textwidth]{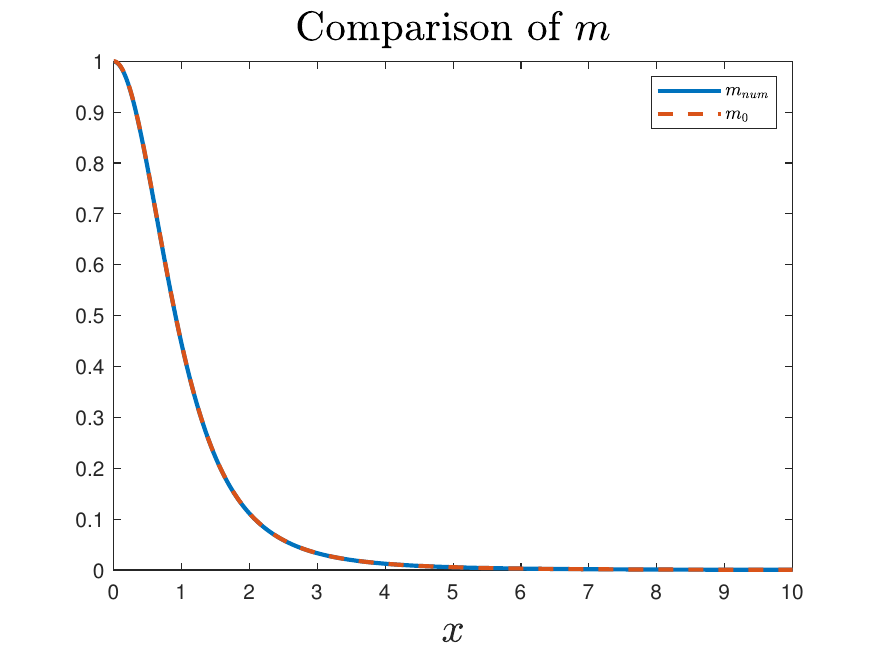}
        \caption{Comparison of \(m\).}
        \label{fig:validation_a0_m}
    \end{subfigure}
    \caption{Validation of the whole-space numerical profile near \(a=0\). The numerical profile at \(a=0.001\) is compared with the explicit limiting profile at \(a=0\) for \(v=u_x\) and \(m=\theta/x^2\).}
    \label{fig:validation_a0}
\end{figure}

\subsection{Dependence of the profiles on \(a\)}

We next compute profiles for a range of values of the advection parameter \(a\). In the periodic case, we compute profiles for \(2/3< a\leq1\), and the results are shown in Figure~\ref{fig:v_profiles}(a). These computations illustrate the periodic profile construction in Theorem~\ref{theorem: main_periodic}, together with its limiting behavior near the endpoints. In particular, for \(2/3<a<1\), the numerical profiles have a finite half-period and correspond to exact self-similar blowup with fixed scale. As \(a\) decreases from \(1\) toward \(2/3\), the half-period length increases, suggesting that the periodic profiles approach a whole-space limiting profile as \(a\to(2/3)^+\).

In the whole-space case with a Neumann condition \(u_x(\infty)=0\), we compute profiles for \(0<a\leq1\), as shown in Figure~\ref{fig:v_profiles}(b). These computations illustrate the three regimes in Theorem~\ref{theorem: main_whole_space}. For \(0<a<2/3\), the profiles have full support and correspond to focusing self-similar blowup. At the critical value \(a=2/3\), the self-similar scaling parameter \(c_l\) vanishes and the profile remains full-support. For \(2/3<a\leq1\), the profiles become compactly supported, in agreement with the expanding self-similar regime.

The plots also illustrate why the fixed-point argument in the proof is formulated only locally near the origin. In all cases, the profile \(v=u_x\) is monotone decreasing near the origin, which is consistent with the local fixed-point space used in Sections~2 and~3. However, this monotonicity is not global in general. In the outer region, especially for profiles with large support or full support, \(v\) may cease to be monotone. This numerical observation supports the structure of the proof: the monotonicity assumptions are used to construct the profile locally near the origin, while the behavior away from the origin is treated by a separate extension or continuation argument.

We single out the critical profile \(a=2/3\) in Figure~\ref{fig:critical_a23}. This profile plays a dual role. It is the critical case in the whole-space setting separating different self-similar regimes, and it also represents the endpoint of the exact self-similar periodic setting when the period tends to infinity. Since \(c_l=0\), the profile lies at the transition between the focusing regime \(0<a<2/3\) and the expanding regime \(2/3<a\leq1\) in the whole-space setting. The plot of \(v\) shows the local monotonicity near the origin and the noncompact far-field behavior, while the plot of \(u\) shows that the velocity profile does not terminate at a finite endpoint; instead, it approaches a positive constant as \(x\to+\infty\). This provides numerical evidence that the critical profile has full support rather than compact support. It confirms the full-support conclusion in Theorem~\ref{theorem: main_whole_space} and illustrates how the periodic profiles degenerate into a whole-space profile in the critical limit.

\begin{figure}[htbp]
    \centering
    \begin{subfigure}{0.49\textwidth}
        \includegraphics[width=\textwidth]{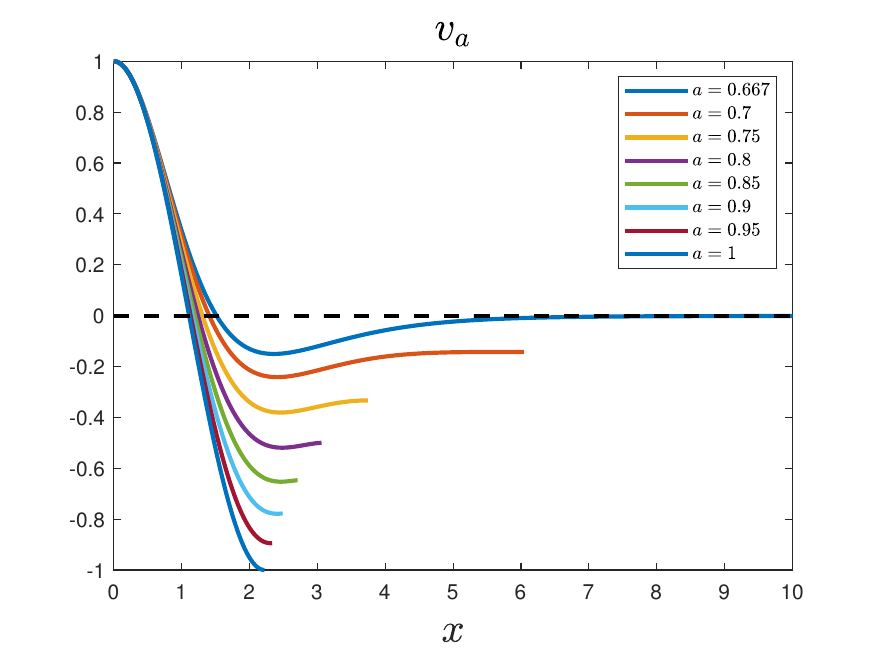}
        \caption{Periodic case.}
        \label{fig:v_periodic}
    \end{subfigure}
    \begin{subfigure}{0.49\textwidth}
        \includegraphics[width=\textwidth]{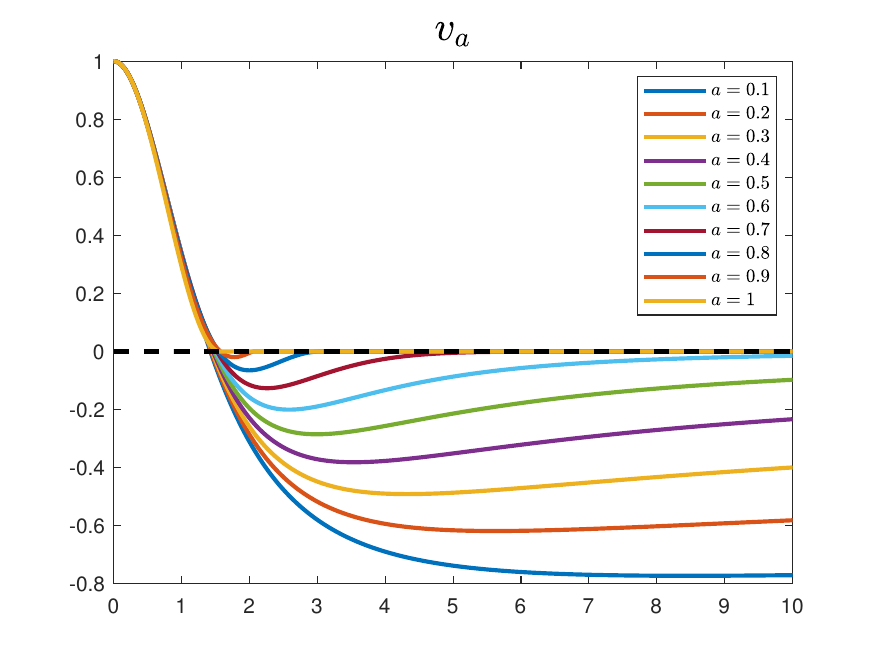}
        \caption{Whole-space case with Neumann condition.}
        \label{fig:v_nonperiodic}
    \end{subfigure}
\caption{Numerically computed profiles \(v_a=u_x\) for different values of \(a\). Left: periodic profiles restricted to the positive half-period near the origin, illustrating the local behavior of the fixed-scale periodic regime in Theorem~\ref{theorem: main_periodic}. Right: whole-space profiles with a Neumann condition, illustrating the focusing, critical fixed-scale, and expanding regimes in Theorem~\ref{theorem: main_whole_space}.}
    \label{fig:v_profiles}
\end{figure}

\begin{figure}[htbp]
    \centering
    \begin{subfigure}{0.49\textwidth}
        \includegraphics[width=\textwidth]{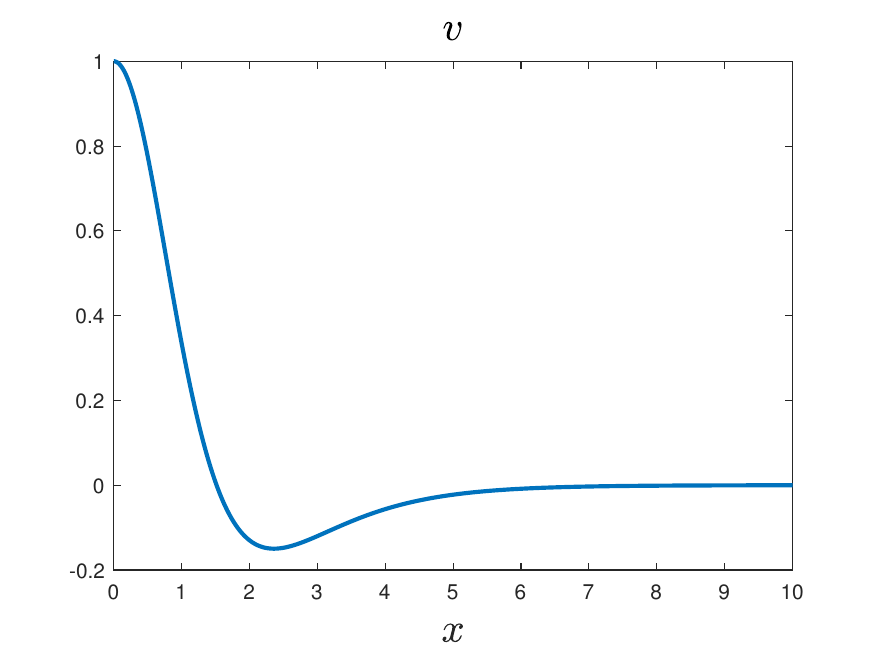}
        \caption{The profile \(v=u_x\).}
        \label{fig:v_critical_a23}
    \end{subfigure}
    \begin{subfigure}{0.49\textwidth}
        \includegraphics[width=\textwidth]{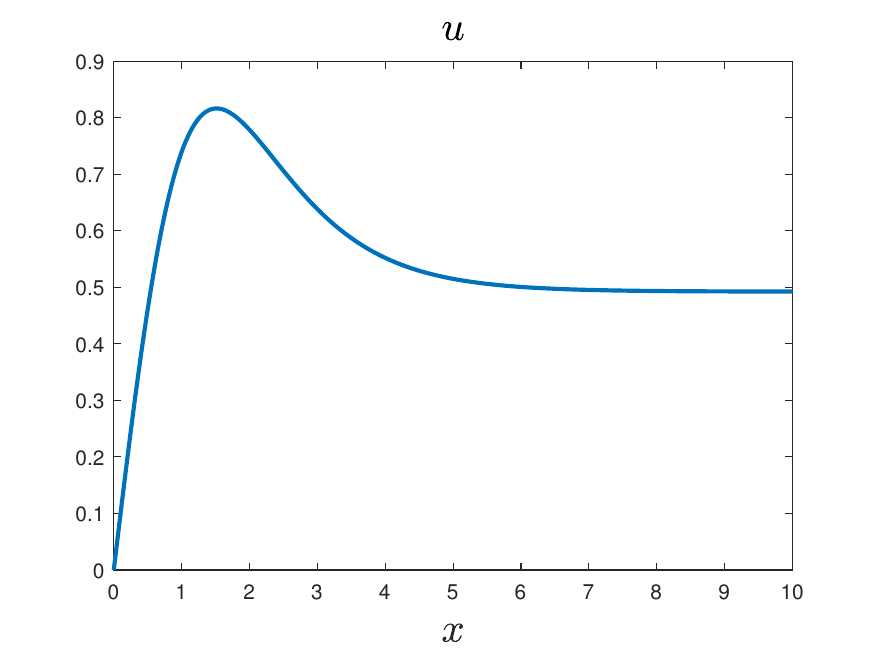}
        \caption{The velocity profile \(u\).}
        \label{fig:u_critical_a23}
    \end{subfigure}
    \caption{The critical whole-space profile at \(a=2/3\). The plots illustrate the full-support fixed-scale profile at the transition between the focusing regime \(0<a<2/3\) and the expanding regime \(2/3<a\leq1\). In particular, \(u\) approaches a positive constant as \(x\to+\infty\), which is consistent with the full-support behavior stated in Theorem~\ref{theorem: main_whole_space}.}
    \label{fig:critical_a23}
\end{figure}

\section*{Acknowledgments}
This research is supported in part by NSF grants DMS-2205590 and DMS-2512878, the Choi Family Gift Fund, and the Dr. Mike Yan Gift Fund. Xiang Qin and Xiuyuan Wang would also like to thank Caltech's Summer Undergraduate Research Fellowships (SURF) program for the opportunity to participate in this research.

\section*{Conflict of interest}

On behalf of all authors, the corresponding author states that there is no conflict of interest.

\section*{Data availability}

No external datasets were used in this study. The numerical results supporting the figures and validation presented in this article are available from the corresponding author upon request.

\bibliographystyle{myalpha}
\bibliography{reference}

@article{chen2021finite,
  title={On the finite time blowup of the {D}e {G}regorio model for the 3{D} {E}uler equations},
  author={Chen, Jiajie and Hou, Thomas Y and Huang, De},
  journal={Communications on pure and applied mathematics},
  volume={74},
  number={6},
  pages={1282--1350},
  year={2021},
  publisher={Wiley Online Library}
}

@article{chen2020singularity,
  title={Singularity formation and global well-posedness for the generalized {C}onstantin--{L}ax--{M}ajda equation with dissipation},
  author={Chen, Jiajie},
  journal={Nonlinearity},
  volume={33},
  number={5},
  pages={2502},
  year={2020},
  publisher={IOP Publishing}
}

@article{constantin1985simple,
  title={A simple one-dimensional model for the three-dimensional vorticity equation},
  author={Constantin, Peter and Lax, Peter D and Majda, Andrew},
  journal={Communications on pure and applied mathematics},
  volume={38},
  number={6},
  pages={715--724},
  year={1985},
  publisher={Wiley Online Library}
}

@article{elgindi2020effects,
  title={On the effects of advection and vortex stretching},
  author={Elgindi, Tarek M and Jeong, In-Jee},
  journal={Archive for Rational Mechanics and Analysis},
  volume={235},
  number={3},
  pages={1763--1817},
  year={2020},
  publisher={Springer}
}

@article{cordoba2005formation,
  title={Formation of singularities for a transport equation with nonlocal velocity},
  author={C{\'o}rdoba, Antonio and C{\'o}rdoba, Diego and Fontelos, Marco A},
  journal={Annals of mathematics},
  pages={1377--1389},
  year={2005},
  publisher={JSTOR}
}

@article{castro2010infinite,
  title={Infinite energy solutions of the surface quasi-geostrophic equation},
  author={Castro, A and C{\'o}rdoba, D},
  journal={Advances in Mathematics},
  volume={225},
  number={4},
  pages={1820--1829},
  year={2010},
  publisher={Elsevier}
}

@article{okamoto2008generalization,
  title={On a generalization of the {C}onstantin--{L}ax--{M}ajda equation},
  author={Okamoto, Hisashi and Sakajo, Takashi and Wunsch, Marcus},
  journal={Nonlinearity},
  volume={21},
  number={10},
  pages={2447},
  year={2008},
  publisher={IOP Publishing}
}

@article{hou2008dynamic,
  title={Dynamic stability of the three-dimensional axisymmetric {N}avier--{S}tokes equations with swirl},
  author={Hou, Thomas Y and Li, Congming},
  journal={Communications on Pure and Applied Mathematics: A Journal Issued by the Courant Institute of Mathematical Sciences},
  volume={61},
  number={5},
  pages={661--697},
  year={2008},
  publisher={Wiley Online Library}
}

@article{hou2023blowup,
  title={Blowup analysis for a quasi-exact 1D model of 3D Euler and Navier-Stokes},
  author={Hou, Thomas Y and Wang, Yixuan},
  journal={arXiv preprint arXiv:2306.04146},
  year={2023}
}

@article{huang2023self,
  title={Self-similar finite-time blowups with smooth profiles of the generalized Constantin-Lax-Majda model},
  author={Huang, De and Qin, Xiang and Wang, Xiuyuan and Wei, Dongyi},
  journal={arXiv preprint arXiv:2305.05895},
  year={2023}
}

@article{elgindi2021finite,
  title={Finite-time Singularity Formation for c\^{}1,$\alpha$ Solutions to the Incompressible Euler Equations on r\^{}3},
  author={Elgindi, Tarek M},
  journal={Annals of Mathematics},
  volume={194},
  number={3},
  pages={647--727},
  year={2021},
  publisher={Department of Mathematics, Princeton University Princeton, New Jersey, USA}
}

@article{luo2014potentially,
  title={Potentially singular solutions of the 3D axisymmetric Euler equations},
  author={Luo, Guo and Hou, Thomas Y},
  journal={Proceedings of the National Academy of Sciences},
  volume={111},
  number={36},
  pages={12968--12973},
  year={2014},
  publisher={National Acad Sciences}
}

@article{cordoba2023finite,
  author = {C{\'o}rdoba, Diego and Mart{\'i}nez-Zoroa, Luis and Zheng, Fan},
  title = {Finite Time Singularities to the 3D Incompressible Euler Equations for Solutions in {$C^{\infty}(\mathbb{R}^3\setminus\{0\})\cap C^{1,\alpha}\cap L^2$}},
  journal = {Annals of PDE},
  volume = {11},
  pages = {19},
  year = {2025},
  doi = {10.1007/s40818-025-00214-2},
  eprint = {2308.12197},
  archivePrefix = {arXiv},
  primaryClass = {math.AP}
}

@misc{shao2026self,
  author = {Shao, Feng and Wei, Dongyi and Zhang, Ping and Zhang, Zhifei},
  title = {Self-similar blow-up solutions of incompressible Euler equations in {$\mathbb{R}^d$}, {$d\geq 3$}, with {$C^{1,1-2/d-}$} velocity},
  year = {2026},
  eprint = {2605.19716},
  archivePrefix = {arXiv},
  primaryClass = {math.AP}
}

@misc{shkoller2026euler,
  author = {Shkoller, Steve},
  title = {Incompressible Euler Blowup at the {$C^{1,\frac{1}{3}}$} Threshold},
  year = {2026},
  eprint = {2603.10945},
  archivePrefix = {arXiv},
  primaryClass = {math.AP}
}

@misc{chen2026eulerI,
  author = {Chen, Jiajie},
  title = {Asymptotically Self-Similar Blowup for 3D Incompressible Euler with {$C^{1,1/3-}$} Velocity I: {$C^\infty$} 1D Limiting Profiles},
  year = {2026},
  eprint = {2605.15149},
  archivePrefix = {arXiv},
  primaryClass = {math.AP}
}

@misc{chen2026eulerII,
  author = {Chen, Jiajie},
  title = {Asymptotically Self-Similar Blowup for 3D Incompressible Euler with {$C^{1,1/3-}$} Velocity II: 3D Profiles, Blowup, and Limiting Behavior},
  year = {2026},
  eprint = {2605.15130},
  archivePrefix = {arXiv},
  primaryClass = {math.AP}
}

@misc{chenhou2022stable,
  author = {Chen, Jiajie and Hou, Thomas Y.},
  title = {Stable Nearly Self-Similar Blowup of the 2D Boussinesq and 3D Euler Equations with Smooth Data I: Analysis},
  year = {2022},
  eprint = {2210.07191},
  archivePrefix = {arXiv},
  primaryClass = {math.AP}
}

@article{jia2014local,
  author = {Jia, Hao and {\v S}ver{\'a}k, Vladim{\'i}r},
  title = {Local-in-space Estimates Near Initial Time for Weak Solutions of the Navier--Stokes Equations and Forward Self-Similar Solutions},
  journal = {Inventiones Mathematicae},
  volume = {196},
  number = {1},
  pages = {233--265},
  year = {2014},
  doi = {10.1007/s00222-013-0468-x},
  eprint = {1204.0529},
  archivePrefix = {arXiv},
  primaryClass = {math.AP}
}

@article{jia2015are,
  author = {Jia, Hao and {\v S}ver{\'a}k, Vladim{\'i}r},
  title = {Are the Incompressible 3D Navier--Stokes Equations Locally Ill-posed in the Natural Energy Space?},
  journal = {Journal of Functional Analysis},
  volume = {268},
  number = {12},
  pages = {3734--3766},
  year = {2015},
  doi = {10.1016/j.jfa.2015.04.006},
  eprint = {1306.2136},
  archivePrefix = {arXiv},
  primaryClass = {math.AP}
}

@article{guillod2023numerical,
  author = {Guillod, Julien and {\v S}ver{\'a}k, Vladim{\'i}r},
  title = {Numerical Investigations of Non-uniqueness for the Navier--Stokes Initial Value Problem in Borderline Spaces},
  journal = {Journal of Mathematical Fluid Mechanics},
  volume = {25},
  number = {3},
  pages = {46},
  year = {2023},
  doi = {10.1007/s00021-023-00789-5},
  eprint = {1704.00560},
  archivePrefix = {arXiv},
  primaryClass = {math.AP}
}

@article{albritton2022nonuniqueness,
  author = {Albritton, Dallas and Bru{\'e}, Elia and Colombo, Maria},
  title = {Non-uniqueness of Leray Solutions of the Forced Navier--Stokes Equations},
  journal = {Annals of Mathematics},
  volume = {196},
  number = {1},
  pages = {415--455},
  year = {2022},
  doi = {10.4007/annals.2022.196.1.3},
  eprint = {2112.03116},
  archivePrefix = {arXiv},
  primaryClass = {math.AP}
}

@misc{houwangyang2025nonuniqueness,
  author = {Hou, Thomas and Wang, Yixuan and Yang, Changhe},
  title = {Nonuniqueness of Leray--Hopf Solutions to the Unforced Incompressible 3D Navier--Stokes Equation},
  year = {2025},
  eprint = {2509.25116},
  archivePrefix = {arXiv},
  primaryClass = {math.AP}
}

@misc{vishik2018partI,
  author = {Vishik, Misha},
  title = {Instability and Non-uniqueness in the Cauchy Problem for the Euler Equations of an Ideal Incompressible Fluid. Part I},
  year = {2018},
  eprint = {1805.09426},
  archivePrefix = {arXiv},
  primaryClass = {math.AP}
}

@misc{vishik2018partII,
  author = {Vishik, Misha},
  title = {Instability and Non-uniqueness in the Cauchy Problem for the Euler Equations of an Ideal Incompressible Fluid. Part II},
  year = {2018},
  eprint = {1805.09440},
  archivePrefix = {arXiv},
  primaryClass = {math.AP}
}

@book{albritton2024vishik,
  author = {Albritton, Dallas and Bru{\'e}, Elia and Colombo, Maria and De Lellis, Camillo and Giri, Vikram and Janisch, Maximilian and Kwon, Hyunju},
  title = {Instability and Non-uniqueness for the 2D Euler Equations, after M. Vishik},
  series = {Annals of Mathematics Studies},
  volume = {219},
  publisher = {Princeton University Press},
  address = {Princeton},
  year = {2024}
}

@misc{mengualsolera2026sharp,
  author = {Mengual, Francisco and Solera, Marcos},
  title = {Sharp Nonuniqueness for the Forced 2D Navier--Stokes and Dissipative SQG Equations},
  year = {2026},
  eprint = {2601.00331},
  archivePrefix = {arXiv},
  primaryClass = {math.AP}
}

@article{houhuang2022twoscale, author = {Hou, Thomas Y. and Huang, De}, title = {A Potential Two-Scale Traveling Wave Singularity for 3D Incompressible Euler Equations}, journal = {Physica D: Nonlinear Phenomena}, volume = {435}, pages = {133257}, year = {2022}, doi = {10.1016/j.physd.2022.133257} }

@article{hou2023interior, author = {Hou, Thomas Y.}, title = {Potential Singularity of the 3D Euler Equations in the Interior Domain}, journal = {Foundations of Computational Mathematics}, volume = {23}, pages = {2203--2249}, year = {2023}, doi = {10.1007/s10208-022-09585-5}, eprint = {2107.05870}, archivePrefix = {arXiv}, primaryClass = {math.AP} }

@article{hou2023navierstokes, author = {Hou, Thomas Y.}, title = {Potentially Singular Behavior of the 3D Navier--Stokes Equations}, journal = {Foundations of Computational Mathematics}, volume = {23}, pages = {2251--2299}, year = {2023}, doi = {10.1007/s10208-022-09578-4}, eprint = {2107.06509}, archivePrefix = {arXiv}, primaryClass = {math.AP} }

@article{houhuang2023degenerate, author = {Hou, Thomas Y. and Huang, De}, title = {Potential Singularity Formation of Incompressible Axisymmetric Euler Equations with Degenerate Viscosity Coefficients}, journal = {Multiscale Modeling \& Simulation}, volume = {21}, number = {1}, pages = {218--268}, year = {2023}, doi = {10.1137/22M1470906}, eprint = {2102.06663}, archivePrefix = {arXiv}, primaryClass = {math.AP} }

@article{hou2026generalized, author = {Hou, Thomas Y.}, title = {Nearly Self-similar Blowup of Generalized Axisymmetric Navier--Stokes Equations}, journal = {Foundations of Computational Mathematics}, year = {2026}, doi = {10.1007/s10208-026-09748-8} }
\end{document}